\documentclass[11pt,letterpaper,reqno]{amsart}
\usepackage{fullpage}
\usepackage{amsmath,amsthm,amsfonts,amssymb,amscd}
\usepackage{empheq}
\usepackage{lipsum}
\usepackage{lastpage}
\usepackage{enumerate}
\usepackage[shortlabels]{enumitem}
\usepackage{mathrsfs}
\usepackage{xcolor}
\usepackage{listings}
\usepackage{chngcntr}

\usepackage{bbding}
\newcommand{\cmark}{\text{\Checkmark}}
\newcommand{\xmark}{\text{\XSolidBrush}}
\newcommand{\namark}{\text{\rule[2pt]{11.5pt}{1.2pt}}}

\usepackage[utf8]{inputenc}

\usepackage{hyperref}
\hypersetup{
    colorlinks=true,
    linkcolor=blue,
    citecolor=blue,
    urlcolor=blue
}

\usepackage{array}
\usepackage{makecell}
\usepackage{diagbox}

\newcommand{\tentry}[4]{%
  $\begin{array}{r@{\,:\,}l}
    \rule{0pt}{2.5ex}#1 & \mathrm{#2}\\
    \rule{0pt}{2.5ex}#3 & \mathrm{#4}
  \end{array}$%
}
\newcommand{\theading}[2]{%
  \makecell{\rule{0pt}{2.5ex}$#1$\\
            \rule{0pt}{2.5ex}$#2$}%
}

\newtheorem{theorem}{Theorem}[section]
\newtheorem{corollary}[theorem]{Corollary}
\newtheorem{lemma}[theorem]{Lemma}

\newtheorem{mytable}[theorem]{Table}
\newtheorem{proposition}[theorem]{Proposition}

\theoremstyle{definition}
\newtheorem{definition}[theorem]{Definition}

\mathtoolsset{showonlyrefs=true}
\numberwithin{equation}{section}
\numberwithin{figure}{section}
\numberwithin{table}{section}

\newcommand{\lrabs}[1]{\!\left\lvert #1 \right\lvert}
\newcommand{\lrp}[1]{\!\left(#1\right)}
\newcommand{\lrb}[1]{\!\left[#1\right]}
\newcommand{\lrcb}[1]{\!\left\{#1\right\}}
\newcommand{\lrfloor}[1]{\!\left\lfloor #1 \right\rfloor}

\mathtoolsset{showonlyrefs=true}
\numberwithin{equation}{section}
\numberwithin{figure}{section}
\numberwithin{table}{section}

\allowdisplaybreaks

\let\subsectiontemp\subsection
\renewcommand{\subsection}[1]{ 
    \subsectiontemp{#1} \hfill\vspace{0.5\linespacing} 
}

\makeatletter
\NewDocumentCommand{\sump}{e{_}}
 {%
  \DOTSB
  \mathop{\IfNoValueTF{#1}{\sump@{}}{\sump@{#1}}}%
  \nolimits
 }
\newcommand{\sump@}[1]{\mathpalette\sump@@{#1}}
\newcommand{\sump@@}[2]{%
  \ifx#1\displaystyle
    {\sump@display{#2}}%
  \else
    \sum@\nolimits'_{#2}%
  \fi
}
\newcommand{\sump@display}[1]{%
  \sbox\z@{$\m@th\displaystyle\sum@\nolimits'$}%
  \sbox\tw@{$\m@th\displaystyle\sum@\limits_{#1}$}%
  \sbox\@tempboxa{$\m@th\displaystyle'$}
  \mathop{\sum@\nolimits' \kern-\wd\@tempboxa}\limits_{#1}%
  \ifdim\wd\z@>\wd\tw@
    \kern\dimexpr\wd\z@-\wd\tw@\relax
  \fi
}
\makeatother

\newcommand{\ben}{\begin{equation}}
\newcommand{\een}{\end{equation}}

\newcommand{\genlegendre}[4]{%
  \genfrac{(}{)}{}{#1}{#3}{#4}%
  \if\relax\detokenize{#2}\relax\else_{\!#2}\fi
}

\author[M. Bohanek]{Maggie Bohanek}
\address[M. Bohanek]{Department of Mathematical Sciences, Smith College, Northampton, MA}
\email{maggie.bohanek@yahoo.com}

\author[O. McGinty]{Owen McGinty}
\address[O. McGinty]{Department of Mathematics, University of California Berkeley, Berkeley, CA }
\email{owen\_mcginty@berkeley.edu}

\author[E. Ross]{Erick Ross}
\address[E. Ross]{School of Mathematical and Statistical Sciences, Clemson University, Clemson, SC}
\email{erickjohnross@gmail.com}

\author[Y. Su]{Yanhui Su}
\address[Y. Su]{School of Mathematical and Statistical Sciences, Clemson University, Clemson, SC}
\email{yanhuis@clemson.edu}

\author[H. Xue]{Hui Xue}
\address[H. Xue]{School of Mathematical and Statistical Sciences, Clemson University, Clemson, SC}
\email{huixue@clemson.edu}


\title[Interlacing for Zeros]{Interlacing for zeros of the Serre \\ derivative of Eisenstein series}

\begin{document}

\begin{abstract}
    In 1970, Rankin and Swinnerton-Dyer showed that the non-elliptic zeros of Eisenstein series $E_k$ in the fundamental domain all lie on the lower arc $\{ e^{i\theta}: \frac{\pi}{2} < \theta < \frac{2\pi}{3}\}$. Very recently, Sugibayashi showed that the same property also holds for the Serre derivative $\vartheta_k(E_k)$ of Eisenstein series. In this paper, we first give very precise estimates for where exactly these zeros are located on the lower arc. These location estimates then allow us to prove four main results. First, we show that the zeros of $\vartheta_\ell(E_\ell)$ Stieltjes interlace with the zeros of $\vartheta_k(E_k)$ on the lower arc for all $\ell > k$. Second, we classify precisely when the zeros of $\vartheta_\ell(E_\ell)$ (standard) interlace with the zeros of $\vartheta_k(E_k)$ on the lower arc. Third, we show that the zeros of $\vartheta_k(E_k)$ always (standard) interlace with the zeros of $E_{k+2}$ on the lower arc. Fourth, as an application of the third main result, we show that the zeros of the cuspidal projection of $\vartheta_k(E_k)$ all lie on the lower arc, extending a result of Xue and Zhu.
\end{abstract}

\maketitle
\tableofcontents

\section{Introduction}
For an even integer $k \geq 4$, the normalized Eisenstein series of weight $k$ on the upper half-plane is defined by
\begin{align}
    E_k(z) = \frac{1}{2}\sum_{\substack{c,d \in \mathbb{Z} \\ \gcd(c,d)=1}} \frac{1}{(cz+d)^k}.
\end{align}
Recall that $E_k(z)$ is a modular form of weight $k$ for $\text{SL}_2(\mathbb{Z})$, and that the action of $\text{SL}_2(\mathbb{Z})$ on the upper half plane has a standard fundamental domain $\mathcal{F}$ given by
\begin{align}
    \mathcal{F} = \left\{|z| \geq 1, -\frac{1}{2} \leq \text{Re}(z) \leq 0 \right\} \cup \left\{|z| > 1, 0 < \text{Re}(z) < \frac{1}{2}\right\}.
\end{align} Rankin and Swinnerton-Dyer \cite{rankin1970zeros} showed that all of the non-elliptic zeros of $E_k(z)$ in $\mathcal F$ are simple and lie on the lower arc $\mathcal{A}:=\left\{e^{i\theta}: \theta \in \left(\frac{\pi}{2}, \frac{2\pi}{3}\right)\right\}$. Note that the elliptic zeros $z=i,e^\frac{2\pi i }{3}$ correspond to the endpoints of this arc. For the rest of the paper, whenever we refer to a zero of a modular form, we will always be referring to a non-elliptic zero in the fundamental domain.

To prove that the zeros lie on the lower arc, Rankin and Swinnerton-Dyer studied the zeros of $(e^{i\theta})^\frac{k}{2}E_k\left(e^{i\theta}\right)$, which is a real-valued function with the same zeros as $E_k\left(e^{i\theta}\right)$. Later, Nozaki \cite{nozaki-separation} used the same function to prove that the zeros of $E_{k+12}(z)$ interlace with those of $E_k(z)$ on the lower arc $\mathcal{A}$; a result originally conjectured by Gekeler \cite{gekeler2001}. This method was later expanded by
Griffin et al. \cite{interlacing-Eisenstein} to establish necessary and sufficient conditions for the zeros of $E_{\ell}$ to interlace with the zeros of $E_{k}$. More recently, Frendreiss et al. \cite{Frendreiss2022} proved that the zeros of $E_\ell(z)$ Stieltjes interlace with the zeros of $E_k(z)$ for all $\ell >k$.

We recall here the definition of Stieltjes interlacing and (standard) interlacing.
\begin{definition}
    Let $\lrcb{\alpha_i}_{1 \le i \le m}$ and $\lrcb{\beta_j}_{1 \le j \le n}$ be two sequences of real numbers written in increasing/decreasing order. 
    Then $\lrcb{\alpha_i}_{1 \le i \le m}$ \textit{Stieltjes interlaces} with $\lrcb{\beta_j}_{1 \le j \le n}$ if there exists at least one $\alpha_i$ value between any two consecutive $\beta_j$ values. 
\end{definition}

\begin{definition}\label{defn:interlacing-defn}
    Let $\lrcb{\alpha_i}_{1 \le i \le m}$ and $\lrcb{\beta_j}_{1 \le j \le n}$ be two sequences of real numbers written in increasing/decreasing order.
    Then $\lrcb{\alpha_i}_{1 \le i \le m}$ \textit{interlaces} with $\lrcb{\beta_j}_{1 \le j \le n}$ if there exists exactly one $\alpha_i$ value between any two consecutive $\beta_j$ values, and exactly one $\beta_j$ value between any two consecutive $\alpha_i$ values. 
\end{definition}

It was recently shown by Sugibayashi \cite{sugibayashi-serre} that the zeros of the Serre derivative of Eisenstein series $\vartheta_k(E_k)$ all lie on the lower arc.
(In fact, Sugibyashi showed this more generally for $\vartheta_k(f)$, where $f$ has all its zeros on the lower arc.)
Recall that the Serre derivative is a differential operator preserving modularity, defined for a weight $k$ modular form $f$ as
\begin{align}
    \vartheta_k(f) := D(f) - \frac{k}{12} E_2f,
\end{align}
where $D := \frac{1}{2\pi i } \frac{d}{dz}$ is the normalized differential operator. Note that the Serre derivative $\vartheta_k : M_k \to M_{k+2}$ maps a modular form of weight $k$ to a modular form of weight $k+2$.

In this paper, we first extend Sugibayashi's result to obtain quantitative estimates for the locations of the zeros of $\vartheta_k(E_k)$, up to error $O\left(\frac{1}{k^2}\right)$. 
We will then use our location estimates to prove the following four results. 
\begin{itemize}
    \item We show that the zeros of $\vartheta_\ell(E_\ell)$ Stieltjes interlace with the zeros of $\vartheta_k(E_k)$.
    \item We classify precisely when the zeros of $\vartheta_\ell(E_\ell)$ interlace with the zeros of $\vartheta_k(E_k)$.
    \item We show that the zeros of $\vartheta_\ell(E_\ell)$ interlace with the zeros of $E_{k+2}$.
    \item We show that the zeros of the cuspidal projection $\operatorname{cuspproj}\!\big(\vartheta_k(E_k)\big)$ all lie on the lower arc.
\end{itemize}

We now give an overview of how our results can be understood within a larger context. One of the overall motivations for this study is the philosophy that $\vartheta_k(E_k)$ can, in a certain sense, be viewed as the modified product ``$E_2\cdot E_k$"; see \cite{Xue-Prabhath-2023} for details. Many interlacing results are known for the product $E_w E_k$ for certain values of $w \ge 4$. So one is led to ask if any corresponding results hold for $w=2$. Since $E_2 E_k$ itself is not even modular, it seems that the most natural generalization of this product to $w=2$ is to use the modified product $``E_2\cdot E_k" = \vartheta_k(E_k)$. We note six natural questions that arise from the above philosophy. In the following, we will state each of the relevant facts in terms of $E_w E_k$ for purposes of symmetry; however note that $E_w$ does not have any non-elliptic zeros for $w=4,6,8,10,14$. 
\begin{itemize}
    \item 
    For $w \ge 4$, the results of \cite{rankin1970zeros} imply that $E_w E_k$ has all its zeros located on the lower arc. So the first natural question is to ask if the same property continues to hold for $``E_2\cdot E_k"=\vartheta_k(E_k)$. This question is answered in the affirmative by Sugibayashi's work. 
    \item 
    The second natural question is to ask where exactly on $\mathcal A$ these zeros of $``E_2\cdot E_k"=\vartheta_k(E_k)$ are located. We will determine the answer to this question in Corollary \ref{cor:actual-zero-sample-zero-distance-bound}.
    \item 
    For $w=4,6,8,10,14$, the results of \cite{Frendreiss2022} imply that the zeros of $E_wE_{\ell}$ Stieltjes interlace with the zeros  $E_wE_{k}$ for all $\ell>k$. So the third natural question is to ask if the same property continues to hold for $``E_2\cdot E_k"=\vartheta_k(E_k)$. We will answer this question in the affirmative in Theorem \ref{thm:stieltjes-theorem}.
    \item 
    For $w=4,6,8,10,14$, the results of \cite{interlacing-Eisenstein} yield a complete classification of when the zeros of $E_w E_\ell$ interlace with the zeros of $E_w E_k$ for all $\ell > k$. So the fourth natural question is to ask if one can also give a similar classification for $``E_2\cdot E_k"=\vartheta_k(E_k)$. We give such a classification in Theorem \ref{thm:perfect-interlacing-theorem}.
    \item 
    For $w=4,6,8$, the results of \cite{interlacing-Eisenstein} also show that the zeros of $E_w E_k$ interlace with the zeros of $E_{k+w}$. So the fifth natural question to ask is if 
    the same property also holds for $``E_2\cdot E_k"=\vartheta_k(E_k)$.
    We will answer this question in the affirmative in Theorem \ref{thm:interlacing_thetakEk_Ek+2}.
    \item 
    For $w=4,6,8$, the results of \cite{Xue-Zhu-2021} show that the zeros of the cuspidal projection of a product of Eisenstein series $\operatorname{cuspproj}(E_w E_k) = E_wE_k-E_{k+w}$ all lie on the lower arc. 
    So the sixth natural question is to ask is if the same property also holds for 
    $\operatorname{cuspproj}(``E_2 \cdot E_k") = \operatorname{cuspproj}\!\big(\vartheta_k(E_k)\big) = \vartheta_k(E_k) + \frac{k}{12}E_{k+2}$.
    We will answer this question in the affirmative in Corollary \ref{cor:cuspidal-proj}.
\end{itemize}

The precise details of our location estimates for the zeros of $\vartheta_k(E_k)$ are given in Corollary \ref{cor:actual-zero-sample-zero-distance-bound}.
We state the four remaining main results here.
\begin{theorem}\label{thm:stieltjes-theorem}
    For all even integers $\ell > k \ge 4$, the zeros of $\vartheta_\ell(E_\ell)$ Stieltjes interlace with the zeros of $\vartheta_k(E_k)$ on the lower arc.
\end{theorem}

\begin{theorem} \label{thm:perfect-interlacing-theorem} 
    For all even integers $\ell > k \ge 4$, the zeros of $\vartheta_\ell(E_\ell)$ interlace with the zeros of $\vartheta_k(E_k)$ on the lower arc if and only if $\ell - k \in \{2,4,6,8,12\}$ or $(k,\ell)$ is one of the exceptional pairs given in Table \ref{table:exceptional-pairs-(k,ell)}.
\end{theorem} 

\begin{theorem}\label{thm:interlacing_thetakEk_Ek+2}
    For all even integers $k \ge 4$, the zeros of $\vartheta_k(E_k)$ interlace with the zeros of $E_{k+2}$ on the lower arc.
\end{theorem}

\begin{corollary} \label{cor:cuspidal-proj}
    For all even integers $k \ge 4$ with $k\notin\{4,6,8,12\}$, every (non-cuspidal, non-elliptic) zero of the cuspidal projection $\operatorname{cuspproj}\!\big(\vartheta_k(E_k)\big)$ lies on the lower arc $\mathcal{A}$.
\end{corollary}

We note that Corollary \ref{cor:cuspidal-proj} is a corollary of Theorem \ref{thm:interlacing_thetakEk_Ek+2}. This type of result is one of the applications of studying interlacing properties.

We also note here that Theorem \ref{thm:perfect-interlacing-theorem} can also be viewed as a modular analog of the following result of Markov: If the zeros of real polynomials $p(x)$ and $q(x)$ are real, simple, and interlacing, then so are the zeros of their derivatives; see  \cite{Dimitrov2012}. Conversely, Markov's result can also be applied to demonstrate certain special cases of Theorem \ref{thm:perfect-interlacing-theorem}; see the discussion at the end of Section \ref{sec:standard}.

Now, to prove our results, we will construct a sequence of increasingly accurate estimates for the zeros of $\vartheta_k(E_k)$. To study the zeros of $\vartheta_k(E_k)$, we first define the real-valued function 
\begin{align}
    F_k(\theta):=\frac{2\pi}{k}z^{\frac{k+2}{2}}\vartheta_k\big(E_k\big)(z)= \frac{2\pi}{k} \text{Re} \left[ z^{\frac{k+2}{2}} \lrp{D(E_k) - \frac{k}{12}  E_2 E_k} \right]\label{eq:actual-F},
\end{align}
where $z=e^{i\theta}$ for $\theta\in\big(\frac{\pi}{2},\frac{2\pi}{3}\big)$.
Note that this function $F_k(\theta)$ has precisely the same zeros as $\vartheta_k(E_k)(e^{i\theta})$. 
Also note that taking the real part in the formula above is unnecessary, since $z^{\frac{k+2}{2}}\vartheta_k(E_k)$ is already real-valued. We state the formula in this way to show the symmetry between $F_k(\theta)$ and its variations below.

We also introduce the following two approximations of $E_k(z)$, ordered from least to most accurate:
\begin{align}
E^*_k(z)&:=1+z^{-k}, \\
    \hat E_k(z)&:=1+z^{-k}+(z+1)^{-k}.
\end{align}
These two approximations of $E_k(z)$ can then be used to construct the following approximations for $F_k(\theta)$, ordered from least to most accurate:

\begin{align}
F_k^{**}(\theta) &:= \frac{2\pi}{k} \text{Re} \left[ z^\frac{k+2}{2} D(E^*_k) \right] \\
    &\ = \sin\left(\frac{k}{2} \theta\right), \label{eq:F^**}\\
    F^*_k(\theta)&:=\frac{2\pi}{k} \text{Re} \left[ z^\frac{k+2}{2}\left(D(E^*_k) - \frac{k}{12}  E_2  E^*_k\right) \right] \\
    &\ = \sin\left(\frac{k}{2} \theta\right) - B_0(\theta) \cos\left(\frac{k}{2} \theta\right)\label{eq:F-star-second} \\
    &\ = B_1(\theta) \sin\left(\frac{k}{2} \theta - B_2(\theta)\right), \label{eq:F^*}\\
    \hat F_k(\theta) &:= \frac{2\pi}{k} \text{Re} \left[ z^\frac{k+2}{2} \left(D(\hat E_k) - \frac{k}{12}  E_2 \hat E_k\right) \right]  \\
    &\ = \sin\left(\frac{k}{2} \theta\right) - B_0(\theta) \cos\left(\frac{k}{2} \theta\right) + B_3(\theta) d(\theta)^{-k}\label{eq:F-hat-second} \\ 
    &\ = B_1(\theta) \sin\left(\frac{k}{2} \theta - B_2(\theta)\right) + B_3(\theta) d(\theta)^{-k} \label{eq:F-hat-final},
\end{align}
where
\begin{align}
    B_0(\theta) &:= \frac{\pi}{3} \text{Re}\lrb{z E_2(z)}, \\
    B_1(\theta) &:= \sqrt{1 + B_0(\theta)^2}, \\
    B_2(\theta) &:= \arctan(B_0(\theta)), \\
    B_3(\theta) &:= -\frac{1}{2} \lrp{
        \tan\lrp{\frac{\theta}{2}} + B_0(\theta)
    },\\
    d(\theta) &:= 2 \cos\left(\frac{\theta}{2}\right).
\end{align}

Next, we define the \textit{presample zeros} $\{\theta^{**}_{k,i}\}_{1\le i \le n_k}$ to be the zeros of $F^{**}_k(\theta)$ in $\big(\frac{\pi}{2}, \frac{2\pi}{3}\big)$; i.e., the integer multiples of $\frac{2\pi}{k}$ in $\lrb{\frac{\pi}{2}+\frac{\pi}{k}, \  \frac{2\pi}{3}-\frac{2\pi}{3k}}$. It is straightforward to compute that the number of presample zeros is given by $n_k:= \dim S_{k+2} = \lrfloor{\frac{k+2}{12}} - 1_{k \equiv 0 \text{ mod } 12}$.
The presample zeros can be written explicitly as 
\begin{align}
    \lrcb{\theta^{**}_{k,i}}_{1\le i \le n_k} 
    &=  
    \lrcb{
        \frac{2\pi}{k} \cdot 2 \lrfloor{\frac{k}{6}} - \frac{2\pi}{k} \lrp{i - \delta_k}
    }_{1\le i \le n_k}
    =  
    \lrcb{
        \frac{2\pi}{3} - \frac{2\pi}{k} \lrp{i - \frac{\delta_k}{3}}
    }_{1\le i \le n_k} \label{eq:pre-sample-zeros}\\
    \text{where } \delta_k &:=
    \begin{cases}
        0 & \text{if } k \equiv 0 \mod 6 \\
        1 & \text{if } k \equiv 2 \mod 6 \\
        2 & \text{if } k \equiv 4 \mod 6
    \end{cases}. 
\end{align}
Note that the presample zeros are indexed from right to left; i.e., $\theta^{**}_{k,i} > \theta^{**}_{k,i+1}$. Also note that the above explicit formula for $\theta^{**}_{k,i}$ implies that
\begin{align} \label{eqn:_theta**_equiv_(i-deltak)pi}
    \frac{k}{2} \theta^{**}_{k,i} \equiv (i-\delta_k)\pi \mod 2\pi.
\end{align}

We then define the \textit{sample zeros} $\theta^*_{k,i}$, the \textit{postsample zeros} $\hat \theta_{k,i}$, and the \textit{actual zeros} $\theta_{k,i}$ to be the corresponding zeros of $F_k^*(\theta)$, $\hat F_k(\theta)$ and $F_k(\theta)$, respectively. Specifically, $\theta^*_{k,i}$, $\hat \theta_{k,i}$, and $\theta_{k,i}$ are defined as the unique zeros of $F_k^*(\theta)$, $\hat F_k(\theta)$ and $F_k(\theta)$ in the interval $\lrb{\theta^{**}_{k,i} - \frac{1.2}{k}, \theta^{**}_{k,i} + \frac{1.2}{k}}$. We show in Lemma \ref{lem:preliminary-estimate} that this definition is well-posed. 
Note that these intervals are non-overlapping since the $\theta^{**}_{k,i}$ are spaced apart by $\frac{2\pi}{k}$.

Now, since $n_k = \dim S_{k+2} = \lrfloor{\frac{k+2}{12}} - 1_{k \equiv 0 \text{ mod } 12}$, observe that
\begin{align}
    k = 12n_k + r_k \qquad \text{where}\quad r_k \in \left \{-2, 12, 2, 4, 6, 8 \right\}.
\end{align}
This identity is equivalent to $\frac{k+2}{12} = n_k + \frac{r_k+2}{12}$, so that $\frac{r_k+2}{12}$ corresponds to the elliptic zeros of $\vartheta_k(E_k)$ guaranteed by the valence formula, and $n_k$ corresponds to the (maximum) number of non-elliptic zeros of $\vartheta_k(E_k)$ allowed by the valence formula. In particular, this means that $F_k(\theta)$ can have at most $n_k$ zeros in $\big(\frac{\pi}{2}, \frac{2\pi}{3}\big)$.
Hence the $n_k$ zeros $\{\theta_{k,i}\}_{1\le i \le n_k}$ we have constructed above are precisely the zeros of $F_k(\theta)$ in $\big(\frac{\pi}{2}, \frac{2\pi}{3}\big)$.

In closing, we give a brief overview of the strategies we use to prove our results. The first half of the paper is devoted to calculating precise location estimates for the zeros of $F_k(\theta)$, up to order $O\big(\frac{1}{k^2}\big)$. This is accomplished by first locating the presample zeros by the explicit formula \eqref{eq:pre-sample-zeros}. Then the sample zeros can be located in terms of the presample zeros, utilizing certain functions of $\theta$ (Proposition \ref{prop:sample-zero-locations}). Then the postsample zeros can be located in terms of the sample zeros (Proposition \ref{prop:postsample-sample-distance-formula}); however, the key technical property of this approximation is that it utilizes a function of $s = k(\frac{2\pi}{3}-\theta)$ (i.e. it is not just a function of $\theta$). Lastly, the postsample zeros are exponentially close to the actual zeros (Proposition \ref{prop:actual-postsample-zero-bound}), yielding the desired zero estimate for actual zeros (Corollary \ref{cor:actual-zero-sample-zero-distance-bound}).

With these precise location estimates, it is then fairly straightforward to prove Stieltjes interlacing. It essentially suffices to just calculate that the distance between consecutive weight $k$ zeros $\theta_{k,i} - \theta_{k,i+1}$ is greater than the distance between consecutive weight $\ell$ zeros $\theta_{\ell,j} - \theta_{\ell,j+1}$. 
Then (standard) interlacing can be obtained from Stieltjes interlacing by comparing the number of weight $\ell$ zeros to the number of weight $k$ zeros, and by studying the behavior of the first/last weight $k$ and weight $\ell$ zeros. Lastly, interlacing of the zeros of $\vartheta_k(E_k)$ with the zeros of $E_{k+2}$ follows from calculating that each zero of $E_{k+2}$ lies just to the right of a corresponding zero of $\vartheta_k(E_k)$.

We note here that our strategy differs somewhat from several previous works studying the zeros of Eisenstein series \cite{nozaki-separation,interlacing-Eisenstein,Frendreiss2022}. The key difference of our strategy is that by including an extra postsample approximation term (see \eqref{eqn:temp-postsample-term}), we are able to give postsample approximations  $\hat \theta_{k,i}$ that are very close to all of the $\theta_{k,i}$. 
In contrast, the works \cite{nozaki-separation,interlacing-Eisenstein,Frendreiss2022} used sample approximations (analagous to our $\theta^*_{k,i}$) that are only accurate away from $\theta = \frac{2\pi}{3}$.
This then required the above works to break into several cases to carefully analyze the ``movement" of zeros near $\frac{2\pi}{3}$.

Because our approach already gives very precise location estimates, we are able to completely avoid any ``movement near $\frac{2\pi}{3}$" arguments (which it seems would be rather complicated and require quite a bit of casework). 
For this reason, compared to the previous approaches, our uniform approach seems to be conceptually simpler and more amenable to generalizations in other settings. For example, because of the technical difficulties that arise near $\frac{2\pi}{3}$, the interlacing result \cite[Theorem 1.3]{jenkins-pratt} was only stated over the interval $\big(\frac{\pi}{2}, \frac{2\pi}{3}-\varepsilon\big)$. In such a setting, we believe that a more uniform approach could be useful.

\section{Preliminary estimates for zeros}
The following lemma provides a preliminary estimate for the locations of the sample, postsample, and actual zeros.

\begin{lemma} \label{lem:preliminary-estimate}
    There exists a unique zero of $F_k^*(\theta)$, $\hat F_k(\theta)$, and $F_k(\theta)$ in the interval
    \begin{align}
        \left[\theta^{**}_{k,i} - \frac{1.2}{k},\theta^{**}_{k,i} + \frac{1.2}{k}\right].
    \end{align}
\end{lemma}
\begin{proof} ~\\
    \textbf{Existence: \\} 
    We show the result for $F_k(\theta)$, as the proof for $\hat F_k(\theta)$ and $F_k^*(\theta)$ is nearly identical.

    Observe that
    \begin{align}
        F_k(\theta) 
        &= \hat F_k(\theta) + (F_k(\theta) - \hat F_k(\theta)) \\
        &= B_1(\theta) \sin\left(\frac{k}{2} \theta - B_2(\theta)\right) + B_3(\theta) d(\theta)^{-k} + (F_k(\theta) - \hat F_k(\theta)) \\
        &= \alpha_1\sin\lrp{\frac{k}{2} \theta - \alpha_2} + B_3(\theta) d(\theta)^{-k} + \alpha_3, \\
        \text{where}\quad 
        \alpha_1 &:= B_1(\theta) \ge 1, \qquad
        0 \geq \alpha_2 := B_2(\theta)  \ge -0.53, 
        \qquad
        \text{(by Lemma \ref{lem:function-bounds})}
        \\
        |\alpha_3| &:= |F_k(\theta) - \hat F_k(\theta)| \leq 0.016. \qquad\qquad\qquad\qquad\quad \text{(by Corollary \ref{cor:bound-for-Fk-minus-Fhatk})} 
    \end{align}
    This  means that
    \begin{align}
        F_k\lrp{\theta^{**}_{k,i} \pm \frac{1.2}{k}} 
        &= \alpha_1 \sin\lrp{\frac{k}{2}\lrp{\theta^{**}_{k,i} \pm \frac{1.2}{k}} - \alpha_2} + B_3\left(\theta^{**}_{k,i} \pm \frac{1.2}{k}\right) d\left(\theta^{**}_{k,i} \pm \frac{1.2}{k}\right)^{-k} + \alpha_3 \\
        &=: \alpha_1 \sin\lrp{\frac{k}{2}\theta^{**}_{k,i} \pm 0.6 - \alpha_2} + \alpha_4^{\pm} + \alpha_3   \\
        &\qquad \left(\text{where $|\alpha_4^{\pm}| \le 0.58(0.2853 \pm 0.194)$ by Lemmas \ref{lem:d-theta-k-endpoint-bounds-cos} and \ref{lem:function-bounds}}\right)\\
        &= \alpha_1 (-1)^{i-\delta_k} \sin\lrp{\pm 0.6 - \alpha_2} + \alpha_4^{\pm} + \alpha_3 \\ 
        &\qquad \left(
        \text{since  $\tfrac{k}{2}\theta^{**}_{k,i}\equiv\pi(i-\delta_k)\mod{2\pi}$ by \eqref{eqn:_theta**_equiv_(i-deltak)pi}}
        \right)\\
        &= \pm \alpha_1 (-1)^{i-\delta_k} \sin\lrp{0.6 \mp \alpha_2} + \alpha_4^{\pm} + \alpha_3,
    \end{align}
    which has opposite signs since
    \begin{align}
        \lrabs{\alpha_1 (-1)^{i-\delta_k} \sin\lrp{0.6 - \alpha_2}} &\ge 1 \cdot \sin(0.6 - 0) \ge 0.564 > 
        |\alpha_4^+ + \alpha_3 |\end{align} \text{and} 
        \begin{align}
        \lrabs{\alpha_1 (-1)^{i-\delta_k} \sin\lrp{0.6 + \alpha_2}} &\ge 1 \cdot \sin(0.6 - 0.53) \ge 0.069 > 
        |\alpha_4^- + \alpha_3|.
    \end{align}
    Hence, $F_k(\theta)$ has a zero in the interval $\left[\theta^{**}_{k,i} - \frac{1.2}{k},\theta^{**}_{k,i} + \frac{1.2}{k}\right]$.

    \textbf{Uniqueness: \\}
    Uniqueness follows immediately for $F_k(\theta)$; as discussed at the end of the introduction, $F_k^{**}(\theta)$ and $F_k(\theta)$ have the same number of zeros by the valence formula, so the zero of $F_k(\theta)$ in the interval must be unique.
    
    We demonstrate uniqueness for $\hat F_k(\theta)$; the proof for $F^*_k(\theta)$ is nearly identical. It suffices to show that $\hat F'_k(\theta)$ is bounded away from zero on the interval. We define $\phi = \theta^{**}_{k,i} + \frac{t}{k}$ for $-1.2 \leq t \leq 1.2$:

    We first bound
    \begin{align}
       \left | \frac{d}{d\theta}\left[B_3(\theta)d(\theta)^{-k} \right]_{\theta = \phi} \right | &= \left |B_3'(\phi)d(\phi)^{-k} + k\sin\left(\frac{\phi}{2}\right)B_3(\phi)d(\phi)^{-k-1} \right | \\
        &\leq \left |B_3'(\phi)d\left(\frac{2\pi}{3} - \frac{2\pi}{3k} + \frac{1.2}{k}\right)^{-k} + \frac{k}{2}\tan\left(\frac{\phi}{2}\right)B_3(\phi)d\left(\frac{2\pi}{3} - \frac{2\pi}{3k} + \frac{1.2}{k}\right)^{-k}\right | \\[-0.35cm]
        &\qquad \text{(by the definition of $d$ and since $d(\theta)^{-k}$ increasing in $\theta$)}\\
        &\leq 0.67(0.2853 + 0.194) + \frac{k}{2}(1.733)(0.58)(0.2853 + 0.194)\\
        &\qquad\text{(by Lemmas \ref{lem:d-theta-k-endpoint-bounds-cos} and \ref{lem:function-bounds})} \\
        &\leq 0.322 + 0.241k. \label{eq:b3-d-theta-derivative-bound}
    \end{align}
    We then have
    \begin{align}
        &\left | \hat F_k'(\phi) \right | \\
        &=\left |\frac{d}{d\theta}\left[ \sin\left(\frac{k}{2} \theta\right) - B_0(\theta) \cos\left(\frac{k}{2} \theta\right) + B_3(\theta) d(\theta)^{-k} \right]_{\theta = \phi} \right |\\
        &= \left | \frac{k}{2}\cos\left(\frac{k}{2}\phi\right) + \frac{k}{2}B_0(\phi)\sin\left(\frac{k}{2}\phi\right) - B_0'(\phi)\cos\left(\frac{k}{2}\phi\right) + \frac{d}{d\theta}\left[B_3(\theta)d(\theta)^{-k}\right]_{\theta = \phi} \right |
        \\
        & = \left | (-1)^{i-\delta_k} 
        \lrb{
        \frac{k}{2}\left(\cos\left(\frac{t}{2}\right) + B_0(\phi)\sin\left(\frac{t}{2}\right)\right) - B_0'(\phi)\cos\left(\frac{t}{2}\right)
        } 
        + \frac{d}{d \theta}\left[B_3(\theta)d(\theta)^{-k}\right]_{\theta = \phi} \right| \\
        &\qquad\left(\text{since} \ 
        \tfrac{k}{2}\phi
        =
        \tfrac{k}{2}\theta^{**}_{k,i} +
        \tfrac{t}{2}
        \equiv
        \pi(i-\delta_k)
        + \tfrac{t}{2}
        \mod{2\pi} 
        \text{ by \eqref{eqn:_theta**_equiv_(i-deltak)pi}}
        \right)\\
        & \geq \left |\frac{k}{2}\cos\left(\frac{t}{2}\right) + \frac{k}{2} B_0(\phi)\sin\left(\frac{t}{2}\right) - B_0'(\phi)\cos\left(\frac{t}{2}\right) \right | - \left| \frac{d}{d \theta}\left[B_3(\theta)d(\theta)^{-k}\right]_{\theta = \phi} \right|\\
        &\geq \left |\frac{k}{2}\Big(\cos(0.6) -0.58\sin(0.6)\Big) - (-0.66)\cos(0.6)  \right| - \left|0.321 + 0.241k\right| \\& \qquad\left(\text{by Lemma \ref{lem:function-bounds} and \eqref{eq:b3-d-theta-derivative-bound}}\right)
        \\
        &\geq 0.248k + 0.544 - 0.322 - 0.241k\\
        &> 0 \quad \text{for }k\geq 10.
    \end{align}
    Thus, $\hat F_k(\theta)$ is monotonic, and the zero $\hat \theta_{k,i}$ must be unique.
\end{proof}

We now prove the bound used in the proof of Lemma \ref{lem:preliminary-estimate}.
\begin{lemma}\label{lem:d-theta-k-endpoint-bounds-cos}
    For $k \geq 10$,
    \begin{align}
        d\left(\theta^{**}_{k,i}\pm \frac{1.2}{k}\right)^{-k} &\leq 0.2853 \pm 0.194.
    \end{align}
\end{lemma}
\begin{proof}
We have
    \begin{align}
        d\left(\theta^{**}_{k,i}\pm \frac{1.2}{k}\right)^{-k} &\leq d\left(\frac{2\pi}{3} - \frac{2\pi}{3k} \pm \frac{1.2}{k}\right)^{-k} &\text{(by Lemma \ref{lem:location-of-first-and-last-presample-zero}, since $d(\theta)^{-k}$ increasing in $\theta$)} \\
        &\leq d\left(\frac{2\pi}{3} - \frac{2\pi}{30} \pm \frac{1.2}{10}\right)^{-10} &\left(\text{since $d\left(\tfrac{2\pi}{3} - \tfrac{2\pi}{30} \pm \tfrac{1.2}{10}\right)^{-k}$ is decreasing in $k$}\right) \\
        &\leq 0.2853 \pm 0.194,
    \end{align}
    as desired.
\end{proof}

We now note where the first and last presample zeros can occur. This lemma follows directly from the definition of $\theta^{**}_{k,i}$.
\begin{lemma} \label{lem:location-of-first-and-last-presample-zero}
    We have
    \begin{align}
        \frac{2\pi}{3} - \frac{2\pi}{k} \le \theta^{**}_{k,1} \le \frac{2\pi}{3} - \frac{2\pi}{3k} 
        \qquad \text{and} \qquad 
        \frac{\pi}{2} + \frac{\pi}{k} \le \theta^{**}_{k,n_k} \le \frac{\pi}{2} + \frac{2\pi}{k}.
    \end{align}
\end{lemma}

As an immediate corollary of Lemma \ref{lem:preliminary-estimate}, we can easily show that the actual zeros are equidistributed along the lower arc. This result could also be shown using the techniques of \cite{sugibayashi-serre}.

\begin{corollary}
    The zeros of $\vartheta_k(E_k)$ are equidistributed along the lower arc $\mathcal{A}$ according to the uniform angle measure.
\end{corollary}
    
\begin{proof}
    Observe that for all intervals $[X,Y] \subseteq \left[\frac{\pi}{2}, \frac{2\pi}{3}\right]$,
    \begin{align}
        &\#\left(\{\theta_{k,i}\}_{1\leq i \leq n_k} \cap [X,Y]\right) \\
        &= \#\left(\{\theta^{**}_{k,i}\}_{1\leq i \leq n_k} \cap [X,Y]\right) + O(1) \\
        &\qquad\left(\text{since the $\theta_{k,i}$ are in the disjoint intervals $\lrcb{[\theta^{**}_{k,i}- \tfrac{1.2}{k}, \theta^{**}_{k,i}+\tfrac{1.2}{k}]}_{1\le i\le n_k}$}\right) \\
        &=\frac{k}{2\pi}\left(Y-X\right) + O(1). \\
        &\qquad \left(\text{since the presample zeros are precisely the multiples of $\tfrac{2\pi}{k}$ in $(\tfrac{\pi}{2}, \tfrac{2\pi}{3})$}\right)
    \end{align}
    This then implies that as $k \to \infty$,
    \begin{align}
        \frac{
            \#\left(\{\theta_{k,i}\}_{1\leq i \leq n_k} \cap [X,Y]\right)
        }{
            \#\{\theta_{k,i}\}_{1\leq i \leq n_k} 
        }
        =
        \frac{\frac{k}{2\pi}(Y-X)+O(1)}{\frac{k}{2\pi}\left(\frac{2\pi}{3} - \frac{\pi}{2}\right) + O(1)} \longrightarrow \frac{Y-X}{\frac{2\pi}{3} - \frac{\pi}{2}} = \frac{6}{\pi}\left(Y-X\right)=\int^Y_{X} \frac{6}{\pi} \ d\theta, \label{eq:equidistribution-equation}
    \end{align}
    as desired.
\end{proof}

\section{Precise estimates for actual zeros}

In this section, we give precise estimates for the locations of the actual zeros of $\vartheta_k(E_k)$ on the lower arc. We first estimate the locations of the sample zeros, then the locations of the postsample zeros, and finally the locations of the actual zeros. Throughout this section, we will only consider $k \geq 1000$ to obtain more accurate estimates.

We first determine the location of the sample zeros, i.e., the zeros of $F^*_k(\theta)$. 
 
\begin{proposition}\label{prop:sample-zero-locations}
    Let $k\geq1000$ and define for $\theta\in(\frac{\pi}{2},\frac{2\pi}{3})$
    \begin{align}
        U(\theta) &:= 2B_2(\theta), \\
        V(\theta) &:= U'(\theta)U(\theta).
    \end{align}
    Then, the sample zeros are given by
    \begin{align}
        \lrcb{\theta^*_{k,i}}_{1\le i \le n_k} 
        =
        \lrcb{ 
            \theta^{**}_{k,i}
            + \frac{U(\theta^{**}_{k,i})}{k} 
            + \frac{V(\theta^{**}_{k,i})}{k^2} 
            + \frac{\alpha^*_{k,i}}{k^3}
        }_{1 \le i \le n_k} \label{eq:sample-zero-formula},\quad \text{where}\ |\alpha^*_{k,i}| \le 5.
    \end{align}
\end{proposition}

\begin{proof}
    For ease of notation, we will denote  $\phi = \theta^{**}_{k,i}$ for the entire proof.

    Recall the definition of $F^*_k(\theta)$ from \eqref{eq:F^*}. Note that $B_1(\theta) > 0$ by Lemma \ref{lem:function-bounds}. Hence, to prove the desired result, it suffices to show that $\sin\left(\frac{k}{2}\theta-B_2(\theta)\right)$ has opposite signs for 
    \begin{align}
        \theta = \phi
            + \frac{U(\phi)}{k} 
            + \frac{V(\phi)}{k^2} 
            \pm \frac{5}{k^3}.
    \end{align}

    We have that
  \begin{align}
      &\sin\left(\frac{k}{2}\left(\phi+\frac{U(\phi)}{k}+\frac{V(\phi)}{k^2}\pm\frac{5}{k^3}\right)-B_2\left(\phi+\frac{U(\phi)}{k}+\frac{V(\phi)}{k^2}\pm\frac{5}{k^3}\right)\right)\\
      &=(-1)^{i-\delta_k}
      \sin\left(\frac{U(\phi)}{2}+\frac{V(\phi)}{2k}\pm\frac{2.5}{k^2}-B_2\left(\phi+\frac{U(\phi)}{k}+\frac{V(\phi)}{k^2}\pm\frac{5}{k^3}\right)\right)\\
      &\qquad \left(\text{since} \ \tfrac{k}{2}\phi=\tfrac{k}{2}\theta^{**}_{k,i}\equiv\pi(i-\delta_k)\mod{2\pi}
      \text{ by \eqref{eqn:_theta**_equiv_(i-deltak)pi}}
      \right)\\
      &=(-1)^{i-\delta_k}\sin\Bigg(\frac{U(\phi)}{2}+\frac{V(\phi)}{2k}\pm\frac{2.5}{k^2}-B_2(\phi)-B'_2(\phi)\left(\frac{U(\phi)}{k}+\frac{V(\phi)}{k^2}\pm\frac{5}{k^3}\right)\\
      & \hspace{33mm}
      -\frac{B''_2(c)}{2}\left(\frac{U(\phi)}{k}+\frac{V(\phi)}{k^2}\pm\frac{5}{k^3}\right)^2\Bigg)\\
      &\qquad \left(\text{for some } c \text{ between $\phi$ and } \phi+\tfrac{U(\phi)}{k}+\tfrac{V(\phi)}{k^2}\pm\tfrac{5}{k^3}\right)\\
      &=(-1)^{i-\delta_k}\sin\Bigg(
        \lrb{
            \frac{U(\phi)}{2} - B_2(\phi)
        }
        +
        \frac{1}{k}
        \lrb{
            \frac{V(\phi)}{2} - B_2'(\phi) U(\phi)
        } \\
        &\hspace{33mm}
        \pm \frac{2.5}{k^2}
        +
        \frac{1}{k^2}
        \lrb{
            - B_2'(\phi) V(\phi) \mp \frac{5B_2'(\phi)}{k}
            - \frac{B_2''(c)}{2}
            \lrp{U(\phi)+\frac{V(\phi)}{k}\pm\frac{5}{k^2}}^2
        } 
        \Bigg)
        \\
        &= (-1)^{i-\delta_k}\sin\Bigg(
        \pm \frac{2.5}{k^2}
        +
        \frac{1}{k^2}
        \lrb{
            - B_2'(\phi) V(\phi) \mp \frac{5B_2'(\phi)}{k}
            - \frac{B_2''(c)}{2}
            \lrp{U(\phi)+\frac{V(\phi)}{k}\pm\frac{5}{k^2}}^2
        } 
        \Bigg) \\
        &=: (-1)^{i-\delta_k}\sin\Bigg(
        \pm \frac{2.5}{k^2}
        +
        \frac{\beta}{k^2} 
        \Bigg),
  \end{align}
  which has opposite signs since
  \begin{align}
      |\beta| 
      &= 
      \lrabs{
            - B_2'(\phi) V(\phi) \mp \frac{5B_2'(\phi)}{k}
            - \frac{B_2''(c)}{2}
            \lrp{U(\phi)+\frac{V(\phi)}{k}\pm\frac{5}{k^2}}^2
        }  \\
        &\leq \lrabs{
            - B_2'(\phi) V(\phi)
            - \frac{B_2''(c)}{2}
            \lrp{U(\phi)+\frac{V(\phi)}{k}\pm\frac{5}{k^2}}^2
        } + \left | \frac{5B_2'(\phi)}{k}\right | \\
        &\leq\text{max}\lrcb{1.30\cdot1.43,\frac{2.49}{2}\left(1.05+\frac{1.43}{1000}+\frac{5}{1000^2}\right)^2} +\frac{6.5}{1000} \qquad(\text{by Lemma \ref{lem:function-bounds}})\\
        &\qquad \text{(since this is the difference of positive terms)}\\
        &\le 1.87.
    \end{align}
    This completes the proof.
\end{proof}

Using this explicit formula, we can show that the sample zeros are not too close to the endpoints of $(\frac{\pi}{2}, \frac{2\pi}{3})$.
\begin{lemma} \label{lem:sample-zero-at-least-3.1/k-away-from-2pi/3}
    For $k \ge 1000$, every sample zero $\theta^*_{k,i}$ satisfies
    \begin{align}
        \frac{\pi}{2} + \frac{3.1}{k} \le \theta^*_{k,i} \le \frac{2\pi}{3} - \frac{3.1}{k}.
    \end{align}
\end{lemma}
\begin{proof}
    For $k\ge1000$,
    \begin{align}
        \theta^*_{k,i}
        &\le \theta^*_{k,1} \\
        &=
        \theta^{**}_{k,1}
            + \frac{U(\theta^{**}_{k,1})}{k} 
            + \frac{V(\theta^{**}_{k,1})}{k^2} 
            + \frac{\alpha^*_{k,1}}{k^3}
        \\
        &\le\frac{2\pi}{3}-\frac{2\pi}{3k}+\frac{U(\frac{2\pi}{3}-\frac{2\pi}{k})}{k}+\frac{V(\frac{2\pi}{3}-\frac{2\pi}{k})}{k^2}+\frac{\alpha^*_{k,1}}{k^3} \\
        &\qquad
        \left(\text{by Lemma \ref{lem:location-of-first-and-last-presample-zero} and since $U$ and $V$ are both decreasing near $\tfrac{2\pi}{3}$}\right)
        \\
        &\le\frac{2\pi}{3} + \frac{1}{k}
        \lrb{
            \frac{-2\pi}{3} + U\lrp{\frac{2\pi}{3} - \frac{2\pi}{1000}} +\frac{V\lrp{\frac{2\pi}{3} - \frac{2\pi}{1000}}}{1000}+\frac{5}{1000^2}
        }\\
        &\le\frac{2\pi}{3}-\frac{3.1}{k},
    \end{align}
    and
    \begin{align}
        \theta^*_{k,i}
        &\ge \theta^*_{k,n_k} \\
        &=
        \theta^{**}_{k,n_k}
            + \frac{U(\theta^{**}_{k,n_k})}{k} 
            + \frac{V(\theta^{**}_{k,n_k})}{k^2} 
            + \frac{\alpha^*_{k,n_k}}{k^3}
        \\
        & \ge \frac{\pi}{2}+\frac{\pi}{k}+\frac{U(\frac{\pi}{2}+\frac{2\pi}{k})}{k}+\frac{V(\frac{\pi}{2}+\frac{\pi}{k})}{k^2}+\frac{\alpha^*_{k,n_k}}{k^3}\\
        & \qquad  \left(\text{by Lemma \ref{lem:location-of-first-and-last-presample-zero} and since $U$ is decreasing and $V$ is increasing near $\tfrac{\pi}{2}$}\right)
        \\
        &\ge \frac{\pi}{2} + \frac{1}{k}
        \lrb{
            \pi + U\left(\frac{\pi}{2}+\frac{2\pi}{1000}\right) + 0 + \frac{-5}{1000^2}
        } \\
        &\ge\frac{\pi}{2}+\frac{3.1}{k}.
    \end{align}
    This completes the proof.
\end{proof}

We record the following observation that will prove useful throughout the paper. The proof is clear from the definition of $F^*_k(\theta)$; see \eqref{eq:F^*}. 
\begin{lemma} \label{lem:sample-presample-exact-expression}
    We have that
    \begin{align}
        \frac{k}{2} \theta^{**}_{k,i}  = \frac{k}{2} \theta^*_{k,i} - B_2(\theta^*_{k,i}),
    \end{align}
    or equivalently,
    \begin{align}
    \theta^*_{k,i} &= \theta^{**}_{k,i} + \frac{U(\theta^*_{k,i})}{k}.
    \end{align}
\end{lemma}

As an immediate corollary, we bound the distance between consecutive sample zeros:
\begin{corollary}\label{cor:sample-zero-2pi-k-distance-bound}
We have that
\begin{align} 
    \theta^*_{k,i} - \theta^*_{k,i+1} \leq \frac{2\pi}{k}.
\end{align}
\begin{proof}
    Observe that
    \begin{align}
        \theta^*_{k,i} - \theta^{*}_{k,i+1} &= \theta^{**}_{k,i} - \theta^{**}_{k,i+1} + \frac{U(\theta^*_{k,i}) - U(\theta^*_{k,i+1})}{k} &\text{(by Lemma \ref{lem:sample-presample-exact-expression})} \\
        &= \frac{2\pi}{k} + \frac{U'(c)(\theta^*_{k,i} - \theta^*_{k,i+1})}{k} &(\text{for $c$ between $\theta^*_{k,i+1}$ and $\theta^*_{k,i}$}) \\
        &\leq \frac{2\pi}{k}, &\text{(since $U'(c)$ $\leq 0$ by Lemma \ref{lem:function-bounds})}
    \end{align}
    as desired.
\end{proof}
    
\end{corollary}

We now estimate the locations of the postsample zeros. 
\begin{proposition}
    \label{prop:postsample-sample-distance-formula}
     Assume $k\ge1000$. Let
    \begin{align}
        T(W) &:= 2 \sin\lrp{\frac{W}{2}} e^{-\tfrac{\sqrt 3}{2} W} 
        \qquad \text{for } W \in [-0.1, 0.1].
    \end{align}
   Then, the postsample zeros are given by
    \begin{align}
        \lrcb{\hat \theta_{k,i}}_{1\le i \le n_k} 
        =
        \lrcb{ 
            \theta^{*}_{k,i}
            + \frac{W_{k,i}}{k}
            + \frac{\hat \alpha_{k,i}}{k^2}
        }_{1 \le i \le n_k}\label{eq:postsample-zero-formula},
    \end{align}
    where $\left|\hat \alpha_{k,i} \right | \le 0.6$ and $W_{k,i}$ is the unique value of $W \in [-0.1, 0.1]$ such that
    \begin{align} \label{eqn:temp-postsample-term}
        T(W_{k,i})  = (-1)^{i-\delta_k} e^{- \frac{\sqrt 3}{2} s^*_{k,i}} \qquad \text{where }\  s^*_{k,i} &:= k\left(\frac{2\pi}{3} - \theta^*_{k,i}\right).
    \end{align}
\end{proposition}
\begin{proof}
    First, note that $W_{k,i}$ is well-defined since $T(W)$ is monotonic with range $[T(-0.1),T(0.1)] = [-0.109...,0.091... ]$. And $(-1)^{i-\delta_k} e^{- \frac{\sqrt 3}{2} s_{k,i}^*}$ lies within this range since 
    $\lrabs{(-1)^{i-\delta_k} e^{- \frac{\sqrt 3}{2} s_{k,i}^*}} \le e^{- \frac{\sqrt 3}{2} \cdot 3.1} \le 0.07$ by Lemma \ref{lem:sample-zero-at-least-3.1/k-away-from-2pi/3}.

    For ease of notation, we will write
    $\phi = \theta^*_{k,i}$ and
    $W = W_{k,i}$ for rest of the proof.
    To prove the desired result, it suffices to show that the values
    $
        \hat F_k\lrp{\phi + \frac{W}{k} \pm \frac{0.6}{k^2}} 
    $
    have opposite signs.
    Recalling \eqref{eq:F-hat-final}, we write
    \begin{align}
        \hat F_k\lrp{\phi + \frac{W}{k}  \pm \frac{0.6}{k^2}} &= B_1\lrp{\phi + \frac{W}{k}  \pm \frac{0.6}{k^2}}\sin\lrp{\frac{k}{2} \lrp{\phi + \frac{W}{k}  \pm \frac{0.6}{k^2}} - B_2\lrp{\phi + \frac{W}{k}  \pm \frac{0.6}{k^2}} } \\
        &\quad \ + B_3\lrp{\phi+ \frac{W}{k}  \pm \frac{0.6}{k^2}} d\lrp{\phi + \frac{W}{k}  \pm \frac{0.6}{k^2}}^{-k}. 
        \label{eqn:postsample-loc-proof-fourterms}
    \end{align}
    We estimate the terms $B_1(\cdot)$, $\sin(\cdot)$, $B_3(\cdot)$, and $d(\cdot)^{-k}$ separately.

    First, we compute that
    \begin{align}
        &B_1\lrp{\phi + \frac{W}{k} \pm \frac{0.6}{k^2}} \\
        &= B_1(\phi) + B_1'(c_1) \lrp{\frac{W}{k} \pm \frac{0.6}{k^2}}  \qquad\qquad\left(\text{for some $c_1$ between $\phi$ and }\phi + \tfrac{W}{k} \pm \tfrac{0.6}{k^2}\right) \\
        &= B_1(\phi) + \frac{\alpha_1}{k}  \label{eqn:postsample-loc-proof-term1} \\
        &\text{where } |\alpha_1| = \lrabs{B_1'(c_1)\left(W \pm \frac{0.6}{k}\right)} \le 0.42 \cdot (0.1 + 0.0006) \le 0.043. \quad
        (\text{Lemma \ref{lem:function-bounds}})
    \end{align}

    Second, we compute that
    \begin{align}
        &\sin\lrp{\frac{k}{2} \lrp{\phi + \frac{W}{k}  \pm \frac{0.6}{k^2}} - B_2\lrp{\phi + \frac{W}{k}  \pm \frac{0.6}{k^2}} }  \\
        &= \sin\lrp{
            \lrp{\frac{k}{2}\phi - B_2(\phi)}
            + \frac{W}{2} \pm \frac{0.3}{k}
            - B_2'(c_2) \lrp{\frac{W}{k} \pm \frac{0.6}{k^2}} 
        } \\ &\qquad \left(\text{for some }c_2 \text{ between $\phi$ and }\phi + \tfrac{W}{k} \pm \tfrac{0.6}{k^2}\right)\\
        &= (-1)^{i-\delta_k} \sin\lrp{
            \frac{W}{2} \pm \frac{0.3}{k}
            - B_2'(c_2) \lrp{\frac{W}{k} \pm \frac{0.6}{k^2}}
        } \\
        &\qquad 
        \left(\text{since } \tfrac{k}{2}\phi - B_2(\phi) = \tfrac{k}{2} \theta_{k,i}^{**} 
        \equiv (i-\delta_k)\pi \mod{2\pi} \text{ by Lemma \ref{lem:sample-presample-exact-expression} and \eqref{eqn:_theta**_equiv_(i-deltak)pi}}\right) \\
        &= (-1)^{i-\delta_k} \sin\lrp{
            \frac{W}{2}
            + \frac{1}{k}
            \lrb{
            \pm 0.3
            - B_2'(c_2) \lrp{W \pm \frac{0.6}{k}}
            }
        } \\
        &= (-1)^{i-\delta_k}
        \lrb{
            \sin\lrp{\frac{W}{2}}
            + 
            \sin'(c)
            \frac{1}{k}
            \lrb{
            \pm 0.3
            - B_2'(c_2) \lrp{W \pm \frac{0.6}{k}}
            }
        }
        \\ &\qquad \left(\text{for some }c \text{ between $\tfrac{W}{2}$ and }\tfrac{W}{2}
            + \tfrac{1}{k}
            \lrb{
            \pm 0.3
            - B_2'(c_2) \lrp{W \pm \tfrac{0.6}{k}}}\right)\\
        &= (-1)^{i-\delta_k}
        \lrb{
            \sin\lrp{\frac{W}{2}}
            \pm 
            \frac{\beta}{k} 
            + \frac{\alpha_2}{k}
        } 
        \label{eqn:postsample-loc-proof-term2}
        \\
        &\text{where } |\alpha_2| = \lrabs{\cos(c) B_2'(c_2) \lrp{W \pm \tfrac{0.6}{k}}} \le 1 \cdot 1.3 \cdot (0.1 + \tfrac{0.6}{1000}) = 0.131, \quad (\text{Lemma \ref{lem:function-bounds}})\\
        &\text{and }\beta= 0.3\cos(c) \ge 
        0.3 \cos\lrp{\tfrac{|W|}{2} + \tfrac{1}{k} \lrb{0.3 + B_2'(c_2)(|W| + \tfrac{0.6}{k})}} \\
        &\hspace{30.mm} 
        \ge 0.3 \cos\left(\tfrac{0.1}{2} + \tfrac{1}{1000} \lrb{0.3 + 0.131}\right) \ge 0.299. \label{eqn:postsample-loc-proof-cos(c)-lb}
    \end{align}

    Third, we compute that
    \begin{align}
        &B_3\lrp{\phi + \frac{W}{k} \pm \frac{0.6}{k^2}} \\
        &= B_3(\phi) + B_3'(c_3) \lrp{\frac{W}{k} \pm \frac{0.6}{k^2}} \qquad \left(\text{for some }c_3 \text{ between $\phi$ and }\phi + \tfrac{W}{k} \pm \tfrac{0.6}{k^2}\right)\\
        &= B_3(\phi) + \frac{\alpha_3}{k} 
        \label{eqn:postsample-loc-proof-term3}
        \\
        &\text{where } |\alpha_3| = \lrabs{B_3'(c_3) \left(W \pm \tfrac{0.6}{k}\right)} \le 0.67 \cdot (0.1 + \tfrac{0.6}{1000}) \le 0.068.\quad (\text{Lemma \ref{lem:function-bounds}})
    \end{align}

    Fourth, we have from Lemma \ref{lem:b()^-k_estimate} that
    \begin{align}
        &d\,\lrp{\phi + \frac{W}{k} \pm  \frac{0.6}{k^2} }^{-k} \\
        &= \frac{-B_1(\phi)}{2B_3(\phi)} e^{-\frac{\sqrt 3}{2} (s^*_{k,i}-W)} + \frac{\alpha_4}{k} 
        \qquad\qquad \text{where } |\alpha_4| \le 0.226
        \\
        &= (-1)^{i-\delta_k} e^{-\frac{\sqrt 3}{2} s^*_{k,i}} 
        \cdot (-1)^{i-\delta_k}\frac{-B_1(\phi)}{2B_3(\phi)}e^{\frac{\sqrt 3}{2} W} + \frac{\alpha_4}{k} \\
        &= 2 \sin\lrp{\frac{W}{2}} e^{-\frac{\sqrt 3}{2} W} \cdot (-1)^{i-\delta_k}
        \frac{-B_1(\phi)}{2B_3(\phi)}
        e^{\frac{\sqrt 3}{2} W} + \frac{\alpha_4}{k}  
        \qquad \text{(by definition of $W$)} \\
        &= (-1)^{i-\delta_k} \frac{-B_1(\phi)}{B_3(\phi)} \sin\lrp{\frac{W}{2}} + \frac{\alpha_4}{k}. 
        \label{eqn:postsample-loc-proof-term4}
    \end{align}

    Then substituting \eqref{eqn:postsample-loc-proof-term1}, \eqref{eqn:postsample-loc-proof-term2}, \eqref{eqn:postsample-loc-proof-term3}, \eqref{eqn:postsample-loc-proof-term4} into \eqref{eqn:postsample-loc-proof-fourterms}, we obtain that
    \begin{align}
        &\hat F_k\lrp{\phi + \frac{W}{k}  \pm \frac{0.6}{k^2}} \\
        &= B_1\lrp{\phi + \frac{W}{k}  \pm \frac{0.6}{k^2}}\sin\lrp{\frac{k}{2} \lrp{\phi + \frac{W}{k}  \pm \frac{0.6}{k^2}} - B_2\lrp{\phi + \frac{W}{k}  \pm \frac{0.6}{k^2}} } \\
        &\quad \ + B_3\lrp{\phi + \frac{W}{k}  \pm \frac{0.6}{k^2}} d\lrp{\phi + \frac{W}{k}  \pm \frac{0.6}{k^2}}^{-k} \\
        &=
        \lrb{B_1(\phi) + \frac{\alpha_1}{k} }
        (-1)^{i-\delta_k}
        \lrb{
            \sin\lrp{\frac{W}{2}}
            \pm 
            \frac{\beta}{k} 
            + \frac{\alpha_2}{k}
        } \\
        &\quad\ +
        \lrb{B_3(\phi) + \frac{\alpha_3}{k} }
        \lrb{
            (-1)^{i-\delta_k} \frac{-B_1(\phi)}{B_3(\phi)} \sin\lrp{\frac{W}{2}} + \frac{\alpha_4}{k} 
        } \\
        &= 
        \lrb{
            (-1)^{i-\delta_k}
            B_1(\phi) \sin\lrp{\frac{W}{2}}
            - 
            (-1)^{i-\delta_k}B_1(\phi)
            \sin\lrp{\frac{W}{2}}
        } \\
        &\quad\ 
        \pm \frac{(-1)^{i-\delta_k}}{k} 
        \Big[
            B_1(\phi) \cdot \beta
        \Big]
        \\
        &\quad\ 
        + \frac{1}{k}
        \Bigg[
            (-1)^{i-\delta_k}
            \alpha_1
            \lrp{
                \sin\lrp{\frac{W}{2}}
                \pm \frac{\beta}{k} + \frac{\alpha_2}{k}
            }
            +
            (-1)^{i-\delta_k}
            \alpha_2 B_1(\phi)
        \\
        &\qquad\qquad +
            \alpha_4 \lrp{B_3(\phi) + \frac{\alpha_3}{k}}
            +
            \alpha_3 
            (-1)^{i-\delta_k} \frac{-B_1(\phi)}{B_3(\phi)} \sin\lrp{\frac{W}{2}}
        \Bigg] \\
        &=: 
        \pm \frac{(-1)^{i-\delta_k}}{k} 
        \Big[
            B_1(\phi) \cdot \beta
        \Big]
        + \frac{\alpha}{k},
        \label{eqn:postsample-loc-proof-canceledmainterm}
    \end{align}
    which have opposite signs, since
    \begin{align}
        B_1(\phi)  & \cdot \beta \ge 1 \cdot 0.299 = 0.299
        \qquad \text{by \eqref{eqn:postsample-loc-proof-cos(c)-lb} and Lemma \ref{lem:function-bounds}},   
        \\
        \text{and }\ 
        |\alpha| &\le 
        0.043 \left( \frac{0.1}{2} + \frac{0.3}{1000} + \frac{0.131}{1000} \right) + 0.131 \cdot 1.16 \\
        &\quad\ + 
        0.226 \left(0.58 + \frac{0.068}{1000}\right) + 0.068 \cdot \frac{1.16}{0.48} \cdot \frac{0.1}{2} \qquad(\text{by Lemma \ref{lem:function-bounds}})\\
        &\le 0.294.
    \end{align}
    This completes the proof.
\end{proof}

We can now replace our postsample zeros, the zeros of $\hat F_k(\theta)$, with actual zeros, the zeros of $F_k(\theta)$, since the postsample zeros are exponentially close to the actual zeros. The following corollary is immediate from Proposition \ref{prop:postsample-sample-distance-formula} and Proposition \ref{prop:actual-postsample-zero-bound}. 
\begin{corollary} \label{cor:actual-zero-sample-zero-distance-bound}
    Assume $k \geq 1000$. Then the actual zeros of $F_k(\theta)$ are given by
    \begin{align}
        \lrcb{\theta_{k,i}}_{1\le i \le n_k} 
        =
        \lrcb{ 
            \theta^{*}_{k,i}
            + \frac{W_{k,i}}{k}
            + \frac{\alpha_{k,i}}{k^2}
        }_{1 \le i \le n_k}\label{eq:actual-zero-formula},\qquad\ \text{where}\ \left |\alpha_{k,i} \right | \leq 0.60001.
    \end{align}
\end{corollary}

\section{Stieltjes interlacing}

In this section, we prove Theorem \ref{thm:stieltjes-theorem}, which states that the zeros of $\vartheta_\ell(E_\ell)$ Stieltjes interlace with the zeros of $\vartheta_k(E_k)$ on the lower arc for all $\ell > k$.  
We will assume $k \ge 1000$ for most of the section. The cases where $k < 1000$ will be verified via direct computation at the end.

We first find the distance between consecutive sample zeros, which is a simple corollary of our sample zero formula. 
\begin{corollary} \label{cor:sample-zero-distances}
    For $k \geq 1000$, the distance between any two consecutive sample zeros is given by 
    \begin{align} 
        \theta^*_{k,i} - \theta^*_{k,i+1} &= \frac{2\pi}{k} + \frac{2\pi U'(\theta^{**}_{k,i})}{k^2} + \frac{\beta^{*}_{k,i}}{k^3}, \qquad\text{where}\ |\beta^{*}_{k,i}|\le 150.6.
    \end{align}
\end{corollary}
\begin{proof}
    Recall the sample zero formula given by Proposition \ref{prop:sample-zero-locations}. Then, it follows from a Taylor approximation of $U$ and $V$ that
    \begin{align}
        \theta^*_{k,i}-\theta^*_{k,i+1}&=\frac{2\pi}{k}+\frac{U(\theta^{**}_{k,i})}{k}-\frac{U(\theta^{**}_{k,i}-\frac{2\pi}{k})}{k}\\
        &\quad\ +\frac{V(\theta^{**}_{k,i})}{k^2}-\frac{V(\theta^{**}_{k,i}-\frac{2\pi}{k})}{k^2}+\frac{\alpha^*_{k,i}-\alpha^*_{k,i+1}}{k^3}\\
        &=\frac{2\pi}{k}+\frac{U(\theta^{**}_{k,i})}{k}-\left(\frac{U(\theta^{**}_{k,i})}{k}-\frac{2\pi U'(\theta^{**}_{k,i})}{k^2}+\frac{2\pi^2U''(c_3)}{k^3}\right)\\
        &\quad\ + \frac{V(\theta^{**}_{k,i})}{k^2} -\left(\frac{V(\theta^{**}_{k,i})}{k^2}-\frac{2\pi V'(c_4)}{k^3}\right)+\frac{\alpha^*_{k,i}-\alpha^*_{k,i+1}}{k^3}\\
        &=\frac{2\pi}{k}+\frac{2\pi U'(\theta^{**}_{k,i})}{k^2}+\frac{1}{k^3}\Bigg[-2\pi^2U''(c_3)+2\pi V'(c_4)+\alpha^*_{k,i}-\alpha^*_{k,i+1}\Bigg]\\
        &=:\frac{2\pi}{k}+\frac{2\pi U'(\theta^{**}_{k,i})}{k^2}+\frac{\beta^*_{k,i}}{k^3},
    \end{align}
    where
    \begin{align}
        \lrabs{\beta^*_{k,i}}
        &= 
        \lrabs{ 
            -2\pi^2U''(c_3)+2\pi V'(c_4)+\alpha^*_{k,i}-\alpha^*_{k,i+1}
        } 
        \\
        &\le2\pi^2\cdot 4.97+2\pi\cdot 6.75+ 5 + 5  
        \qquad\text{(by Lemma \ref{lem:function-bounds} and Proposition \ref{prop:sample-zero-locations}})
        \\
        &\le 150.6,
    \end{align}
    as desired.
\end{proof}

We note the following lemma, which deals with the first and last weight $\ell$ sample zeros. The proof is given in Appendix \ref{app:sample-zero-at-edges}.
\begin{lemma} 
\label{lem:zero-endpoint-behavior-k-ell}
    Let $\ell > k \geq 1000$. Then 
    \begin{align}
        \theta^*_{\ell,1} > \theta^*_{k,2} + \frac{3(\ell -k)}{\ell k}
        \quad  \text{and} \quad
        \theta^*_{\ell,n_\ell} <\theta^*_{k,n_k-1} - \frac{3(\ell - k)}{\ell k}. 
    \end{align}
\end{lemma}

We can then show Stieltjes interlacing of sample zeros (with some extra error tolerance).
\begin{lemma} \label{lem:interlacing-sample-zeros-with-gap}
    For $k \ge 1000$, let $\theta^*_{k,i}$ and $\theta^*_{k,i+1}$ be any pair of consecutive sample zeros. Then there exists a weight $\ell$ sample zero $\theta^*_{\ell,j}$ such that
    \begin{align}
        \theta^*_{k,i+1} + \frac{3(\ell-k)}{\ell k}
        <
        \theta^*_{\ell,j} 
        <
        \theta^*_{k,i} - \frac{3 (\ell-k)}{\ell k}.
    \end{align}
\end{lemma}
\begin{proof}
    Let $\theta^*_{\ell,j}$ denote the smallest sample zero greater than $\theta^*_{k,i+1}+\frac{3(\ell-k)}{\ell k}$.
    Note that $\theta^*_{\ell,j}$ here is guaranteed to exist since, for example, 
    $\theta^*_{\ell,1} > \theta^*_{k,2}+\frac{3(\ell-k)}{\ell k} \ge \theta^*_{k,i+1}+\frac{3(\ell-k)}{\ell k}$ by
    Lemma \ref{lem:zero-endpoint-behavior-k-ell}.
    
    If $j = n_\ell$, the result follows immediately from Lemma \ref{lem:zero-endpoint-behavior-k-ell}. So we will assume that $j < n_\ell$. Then we need to show
    \begin{align}
        \theta^*_{k,i}-\frac{3(\ell-k)}{\ell k}>\theta^*_{\ell,j}.
    \end{align}
    
    We first show $\theta^{**}_{\ell,j} \leq \theta^{**}_{k,i}$, a necessary fact for our proof. Let $f_k(\theta):= \theta - \frac{U(\theta)}{k}$, which satisfies $f_k(\theta^*_{k,i}) = \theta^{**}_{k,i}$ by Lemma \ref{lem:sample-presample-exact-expression}. Suppose for contradiction $\theta^{**}_{\ell,j} > \theta^{**}_{k,i}$. Then
    \begin{align}
        f_k(\theta^*_{\ell,j+1}) - f_k(\theta^*_{k,i+1}) &\geq f_\ell(\theta^*_{\ell,j+1}) - f_k(\theta^*_{k,i+1}) \\& = \theta_{\ell,j+1}^{**}-\theta_{k,i+1}^{**}
    =
    \left(\theta_{\ell,j}^{**}-\theta_{k,i}^{**}\right)
    +\frac{2\pi}{k}-\frac{2\pi}{\ell}
    >
    \frac{2\pi(\ell-k)}{\ell k}.
    \end{align}
    Since $f_k(\theta)$ is increasing, we must have $\theta^*_{\ell,j+1} > \theta^*_{k,i+1}$. Thus by the mean value theorem, for $c$ between $\theta^*_{\ell, j+1} $ and $ \theta^*_{k,i+1}$, using $f_k'(\theta) \leq 1 + \frac{2.6}{k}$,
    \begin{align}
        \theta^*_{\ell, j+1} - \theta^*_{k,i+1} = \frac{f_k(\theta^*_{\ell,j+1}) - f_k(\theta^*_{k,i+1})}{f'_k(c)} \geq \frac{2\pi(\ell -k)}{\ell k} \cdot \frac{1}{1 + \tfrac{2.6}{k}} > \frac{3(\ell - k)}{\ell k},
    \end{align}
    contradicting that $\theta^*_{\ell,j}$ is the smallest sample zero greater than $\theta^*_{k,i+1}+\frac{3(\ell-k)}{\ell k}$. Thus $\theta^{**}_{\ell,j} \leq \theta^{**}_{k,i}$.

    Now,
    \begin{align}
      &\theta_{k,i}^* - \frac{3(\ell-k)}{\ell k}-\theta_{\ell,j}^*\\
       &= \theta_{k,i+1}^* + \frac{3(\ell-k)}{\ell k}+\left(\theta_{k,i}^* - \theta_{k,i+1}^*-\frac{6(\ell-k)}{\ell k}\right) - \theta_{\ell,j+1}^* - (\theta_{\ell,j}^* - \theta_{\ell,j+1}^*) \\
        &\ge \left(\theta_{k,i}^* - \theta_{k,i+1}^*-\frac{6(\ell-k)}{\ell k}\right) - (\theta_{\ell,j}^* - \theta_{\ell,j+1}^*) \quad\ \qquad\qquad\text{(since $\theta^*_{\ell,j+1} \le \theta^*_{k,i+1} + \tfrac{3(\ell-k)}{\ell k}$)} \\
        &= \lrp{\frac{2\pi}{k} + \frac{2\pi U'(\theta_{k,i}^{**})}{k^2} -\frac{6(\ell-k)}{\ell k} +\frac{\beta^*_{k,i}}{k^3}} - 
        \lrp{\frac{2\pi}{\ell} + \frac{2\pi U'(\theta_{\ell,j}^{**})}{\ell^2} + \frac{\beta^*_{\ell,j}}{\ell^3}}  \quad(\text{by Corollary \ref{cor:sample-zero-distances}})\\
        &= \frac{2\pi}{k} - \frac{2\pi}{\ell}+ \frac{2\pi U'(\theta_{k,i}^{**}) - 2\pi U'(\theta_{\ell,j}^{**})}{k^2} + 2\pi U'(\theta_{\ell,j}^{**}) \lrp{\frac{1}{k^2} - \frac{1}{\ell^2}} \\
        &\quad\ -\frac{6(\ell-k)}{\ell k}+ \frac{\beta^*_{k,i}}{k^3} -\frac{\beta^*_{\ell,j}}{\ell^3} \\
        &\ge \frac{2\pi}{k} - \frac{2\pi}{\ell} - 16.34 \lrp{\frac{1}{k^2} - \frac{1}{\ell^2}} -\frac{6(\ell-k)}{\ell k} +\frac{\beta^*_{k,i}}{k^3} -\frac{\beta^*_{\ell,j}}{\ell^3} \\
        &\qquad \text{(since $2\pi U'(\theta)$ is increasing, $\theta^{**}_{\ell,j} \le \theta^{**}_{k,i}$, and $2\pi U'(\theta) \ge -16.34$\ \text{by Lemma \ref{lem:function-bounds})}}\\
        &=2\pi\frac{\ell-k}{\ell k}-6\frac{\ell-k}{\ell k}
        - 16.34 \frac{\ell+k}{\ell k} \frac{\ell-k}{\ell k}
        + \frac{\beta^*_{k,i}}{k^3} -\frac{\beta^*_{\ell,j}}{\ell^3}
        \\
        &\ge \lrp{2\pi -6- \frac{32.68}{k}} \frac{\ell-k}{\ell k}  - \frac{2 \cdot 150.6}{k^3} 
        \qquad (\text{by Corollary} \ \ref{cor:sample-zero-distances} )\\ 
        &\ge 0.25 \frac{\ell-k}{\ell k}  -\frac{0.3012}{k^2} \qquad \qquad\  (\text{since}\ k\ge 1000 ) \\
        &\ge 0.25 \frac{2}{(k+2) k}  -\frac{0.3012}{k^2} \qquad \quad (\text{since } \ell \ge k+2), \\
        &> 0,
    \end{align}
    completing the proof.
\end{proof}

The following proposition then implies Theorem \ref{thm:stieltjes-theorem} when $k \ge 1000$. In light of our location estimates for $\theta_{k,i}$ from Corollary \ref{cor:actual-zero-sample-zero-distance-bound}, the proof is relatively straightforward. In short, we show that if the inequality below is not immediate, then $W_{k,i}$ and $W_{\ell,j}$ are either very close to each other or exponentially small; in either case, this ensures the actual zeros maintain the ordering given by Lemma \ref{lem:interlacing-sample-zeros-with-gap}.

However, the calculations involved are rather lengthy, so we have relegated the complete details of the proof to Appendix \ref{sec:proof_of_prop:stieltjes-interlacing-large-k}.
\begin{proposition}\label{prop:stieltjes-interlacing-large-k}
    For $k \ge 1000$, let
    \begin{align}
        \theta^*_{k,i+1} + \frac{3(\ell-k)}{\ell k}
        <
        \theta^*_{\ell,j} 
        <
        \theta^*_{k,i} - \frac{3 (\ell-k)}{\ell k}
    \end{align}
    be the sample zeros from Lemma \ref{lem:interlacing-sample-zeros-with-gap}. Then
    \begin{align}
        \theta_{k,i+1} 
        <
        \theta_{\ell,j} 
        <
        \theta_{k,i}.
    \end{align}
\end{proposition}

Finally, we verify Theorem \ref{thm:stieltjes-theorem} for $k < 1000$ via direct computation; see \cite{ross-code} for the code. For each of these values of $k$, we only need to check $k < \ell < 2.3k$ in light of Proposition \ref{prop:actual-interlace-ell-ge-2.3k}. The proof of this proposition only uses the preliminary location estimates of Lemma \ref{lem:preliminary-estimate}. 
\begin{proposition}\label{prop:actual-interlace-ell-ge-2.3k}
    For all $\ell \ge 2.3 k$, the actual zeros of $F_\ell(\theta)$ Stieltjes interlace with the actual zeros of $F_k(\theta)$. 
\end{proposition}
\begin{proof}
    Note that we can assume $n_k \ge 2$ (and thus $k \ge 22$) since otherwise, Stieltjes interlacing holds vacuously.
    By Lemma \ref{lem:preliminary-estimate}, we have
    \begin{align}
        -\frac{1.2}{k} \leq \theta_{k,i} - \theta^{**}_{k,i} \leq \frac{1.2}{k},
    \end{align}
    so using the fact that $\theta^{**}_{k,i} - \theta^{**}_{k,i+1} = \frac{2\pi}{k}$, we have
    \begin{align}
        \frac{2\pi - 2.4}{k} &\le \theta_{k,i} - \theta_{k,i+1} \le \frac{2\pi + 2.4}{k}, \quad \text{and similarly,}\\
         \quad 
        \frac{2\pi - 2.4}{\ell} &\le \theta_{\ell,j} - \theta_{\ell,j+1} \le \frac{2\pi + 2.4}{\ell}.
    \end{align}
    Then since $\ell \geq 2.3k$, we have
    \begin{align}
        \theta_{\ell,j} - \theta_{\ell,j+1} \le \frac{2\pi+2.4}{\ell} \leq \frac{2\pi+2.4}{2.3k} < \frac{2\pi -2.4}{k} \le \theta_{k,i} - \theta_{k,i+1}.
    \end{align}
    In other words, this shows that the distance between consecutive weight $\ell$ actual zeros must be strictly smaller than the distance between consecutive weight $k$ actual zeros. Together with Lemma \ref{lem:first-and-last-weight-ell-zeros_when-ell-le-2.3k}, this implies the desired result.
\end{proof}

Lemma \ref{lem:first-and-last-weight-ell-zeros_when-ell-le-2.3k} deals with the first and last weight $\ell$ zeros.
\begin{lemma}
    \label{lem:first-and-last-weight-ell-zeros_when-ell-le-2.3k}
    For all $k$ such that $n_k \ge 2$ and $\ell \ge 2.3k$, we have that
    \begin{align}
        \theta_{\ell,1} > \theta_{k,2} 
        \quad  \text{and} \quad
        \theta_{\ell,n_\ell} <\theta_{k,n_k-1}. 
    \end{align}
\end{lemma}
\begin{proof}
    First, recall by Lemma \ref{lem:location-of-first-and-last-presample-zero} that
    \begin{align}
        \theta^{**}_{\ell, 1}  &\geq \frac{2\pi}{3} - \frac{2\pi}{\ell}, \quad \theta^{**}_{k, 2} \leq \frac{2\pi}{3} - \frac{8\pi}{3k} \\
        \text{and} \quad
        \theta^{**}_{\ell, n_\ell} &\leq \frac{\pi}{2} + \frac{2\pi}{\ell}, \quad \theta^{**}_{k, n_k - 1} \geq \frac{\pi}{2} + \frac{3\pi}{k}.
    \end{align}
    This then implies with Lemma \ref{lem:preliminary-estimate} that
    \begin{align}
        \theta_{\ell, 1} \geq \theta^{**}_{\ell, 1} - \frac{1.2}{\ell} &\geq \frac{2\pi}{3} - \frac{2\pi + 1.2}{\ell} \\
        &\geq \frac{2\pi}{3} - \frac{2\pi + 1.2}{2.3k} \\
        &> \frac{2\pi}{3} - \frac{8\pi}{3k} + \frac{1.2}{k} \geq \theta^{**}_{k, 2} + \frac{1.2}{k} \geq \theta_{k,2}
    \end{align}
    and 
    \begin{align}
        \theta_{\ell, n_\ell} \leq \theta^{**}_{\ell, n_\ell} + \frac{1.2}{\ell}
        &\leq \frac{\pi}{2} +  \frac{2\pi + 1.2}{\ell} \\
        &\leq  \frac{\pi}{2} + \frac{2\pi + 1.2}{2.3k} \\ 
        &< \frac{\pi}{2} + \frac{3\pi}{k} - \frac{1.2}{k} \leq \theta^{**}_{k, n_k-1} - \frac{1.2}{k} \leq \theta_{k,n_k -1},
    \end{align}
    as desired.
\end{proof}

\section{Standard interlacing} \label{sec:standard}

In this section, we classify precisely when the zeros of $F_\ell(\theta)$ (standard) interlace with the zeros of $F_k(\theta)$. The following is an immediate corollary of Stieltjes interlacing, established in Theorem \ref{thm:stieltjes-theorem}.

\begin{corollary}\label{cor:perfect-interlacing-condition-theta}
    Let $n_k \geq 1$, $\ell > k$ (and so $n_\ell \ge n_k - 1$). 
    \begin{enumerate}
        \item Suppose $n_\ell = n_k - 1$. Then we have interlacing.
        \item Suppose $n_\ell = n_k$. Then we have interlacing iff $\theta_{\ell,1} > \theta_{k,1}$ or $\theta_{\ell,n_\ell} < \theta_{k, n_k}$.
        \item Suppose $n_\ell = n_k+1$. Then we have interlacing iff\ $\theta_{\ell,1} > \theta_{k,1}$ and $\theta_{\ell,n_\ell} < \theta_{k, n_k}$.
        \item Suppose $n_\ell \ge n_k+2$. Then we do not have interlacing.
    \end{enumerate}
\end{corollary}

Now, observe that
\begin{align}
    \theta^{**}_{k,1} &= \frac{2\pi}{3} - \frac{\gamma_k}{3} \frac{2\pi}{k} \qquad \text{where } \gamma_k := 
    \begin{cases}
        3 & \text{if } k \equiv 0 \mod 6 \\
        2 & \text{if } k \equiv 2 \mod 6 \\
        1 & \text{if } k \equiv 4 \mod 6
    \end{cases}, 
    \label{eq:endpoint_theta_1}
    \\
    \theta^{**}_{k,n_k} &= \frac{\pi}{2} + \frac{\gamma'_k}{2} \frac{2\pi}{k} \qquad \text{where } \gamma'_k := 
    \begin{cases}
        2 & \text{if } k \equiv 0 \mod 4  \\
        1 & \text{if } k \equiv 2 \mod 4
    \end{cases}.
    \label{eq:endpoint_theta_2}
\end{align}
If $n_\ell = n_k$ or $n_k+1$, we must have $\ell \le k+26$; we can thus apply Lemma \ref{lem:perfect-interlacing}, which states that for $k \geq 1000$,
\begin{align}
    \theta_{\ell,1} > \theta_{k,1} \quad \text{iff} \quad \gamma_\ell  \le \gamma_k 
    \qquad \text{and} \qquad
    \theta_{\ell,n_\ell} < \theta_{k,n_k} \quad \text{iff} \quad \gamma'_\ell  \le \gamma'_k. 
\end{align}
Hence, we can restate Corollary \ref{cor:perfect-interlacing-condition-theta} (for $k \ge 1000$) as the following:
\begin{corollary}
    \label{cor:perfect-interlacing-condition-gamma}
    Let $\ell > k \geq 1000$ (and so $n_\ell \ge n_k - 1$). 
    \begin{enumerate}
        \item Suppose $n_\ell = n_k - 1$. Then we have interlacing.
        \item Suppose $n_\ell = n_k$. Then we have interlacing iff $\gamma_\ell \le \gamma_k$ or $\gamma'_\ell \le \gamma'_k$.
        \item Suppose $n_\ell = n_k+1$. Then we have interlacing iff $\gamma_\ell \le \gamma_k$ and $\gamma'_\ell \le \gamma'_k$.
        \item 
        Suppose $n_\ell \ge n_k+2$. Then we do not have interlacing.
    \end{enumerate}
\end{corollary}

This allows us to prove Theorem \ref{thm:perfect-interlacing-theorem}.
{
\renewcommand{\thetheorem}{\ref{thm:perfect-interlacing-theorem}}
\addtocounter{theorem}{-1}
\begin{theorem} 
    For all even integers $\ell > k \ge 4$, the zeros of $\vartheta_\ell(E_\ell)$ interlace with the zeros of $\vartheta_k(E_k)$ on the lower arc if and only if $\ell - k \in \{2,4,6,8,12\}$ or $(k,\ell)$ is one of the exceptional pairs given in Table \ref{table:exceptional-pairs-(k,ell)}.
\end{theorem} 
}

\begin{mytable}
    \label{table:exceptional-pairs-(k,ell)}
    The following table gives the complete list of exceptional pairs $(k,\ell)$ from Theorem \ref{thm:perfect-interlacing-theorem}.
    
    \centering
    \begin{tabular}{|c|c|c|c|c|c|c|c|c|}
        \hline
        (4, 14) &  (4, 18) &  (4, 20) &  (4, 24) &  (6, 16) &  (6, 20) &  (6, 24) &  (8, 18) &  (8, 24)  
        \\
        \hline
        (10, 20) &  (10, 24) &  (10, 26) &  (10, 28) &  (10, 30) &  (10, 32) &  (10, 36) &  (14, 24) &  (14, 30)  
        \\
        \hline
        (14, 32) &  (14, 36) &  (16, 32) &  (20, 30) &  (20, 36) &  (22, 38) &  (26, 42) &  (32, 48) &  (38, 54)
        \\
        \hline
    \end{tabular}
\end{mytable}

\begin{proof} 
    First, we computationally check when interlacing holds for all $k < 1000$; see \cite{ross-code} for the code. Note that for each of these values of $k$, we only need to check $k < \ell \le 26$. Otherwise, we will have $n_\ell \ge n_k+2$ and so interlacing cannot occur, as predicted (c.f. Case 1 below). This computation produces the finite list of exceptions given in Table \ref{table:exceptional-pairs-(k,ell)}. For the rest of the proof, we will assume that $k \geq 1000$.
    
    Throughout the remainder of the proof, recall that $k=12n_k + r_k$, where $r_k \in \{-2, 12, 2, 4, 6, 8\}$.
    
    \textbf{Case 1: $n_\ell \ge n_k+2$. \\}
    In this case, we do not have interlacing. Additionally,
    we have that
    \begin{align}
        \ell-k &= 12 (n_\ell - n_k) + (r_{\ell}-r_k) \ge 12 \cdot 2 + (-14) = 10, \\
        \text{and}\quad
        \ell-k &= 12 (n_\ell - n_k) + (r_{\ell}-r_k) \ne 12. \qquad (\text{since $r_\ell - r_k \neq -12$})
    \end{align}
    Hence $\ell - k \notin \{2,4,6,8,12\}$, as predicted.

    \textbf{Case 2: $n_\ell = n_k-1$. \\}
    In this case, we have interlacing. Additionally,
    \begin{align}
        \ell-k &= 12 (n_\ell - n_k) + (r_{\ell}-r_k) = -12 + (r_\ell - r_k). 
    \end{align}
    Since $\ell > k$, this then implies that $\ell - k = -12 + 14 = 2 \in \{2,4,6,8,12\}$, as predicted.

    \textbf{Case 3: $n_\ell = n_k$ or $n_\ell = n_k+1$. \\}
    In the following table, we consider each possible equivalence class modulo $12$ for $k$ and $\ell$.
    In the left column, we list the form of $k$, as well as $(\gamma_k, \gamma'_k)$. 
    On the top row, we list the form of $\ell$, as well as $(\gamma_\ell, \gamma'_\ell)$.
    For each entry in the table, we list the two values of $\ell-k$ for $n_\ell = n_k$, and $n_\ell = n_k+1$. Also, for these two values of $\ell-k$, we give a corresponding symbol: \cmark, \xmark, or $\namark$. This symbol indicates whether or not the condition of Corollary \ref{cor:perfect-interlacing-condition-gamma} is satisfied for $(\gamma_k,\gamma'_k)$ and $(\gamma_\ell,\gamma'_\ell)$.
    \begin{itemize}
        \item $\cmark$: The condition of Corollary \ref{cor:perfect-interlacing-condition-gamma} is satisfied (so we do have interlacing).
        \item $\xmark$: The condition of Corollary \ref{cor:perfect-interlacing-condition-gamma} is not satisfied (so we do not have interlacing).
        \item $\namark$: This scenario is impossible since $\ell - k \le 0$.
    \end{itemize}

\begin{center}
\begin{tabular}{c||c|c|c|c|c|c|}
\diagbox[width=7em,height=5.5em]
   {\theading{
        k\vspace{-0.7mm}
    }{
        (\gamma_k,\gamma'_k)\vspace{1mm}
    }}
    {\theading{
        \ell \vspace{-0.7mm}
    }{
        (\gamma_\ell,\gamma'_\ell)
    }}
&
\theading{12n_\ell-2}{(1,1)}
&
\theading{12n_\ell+12}{(3,2)}
&
\theading{12n_\ell+2}{(2,1)}
&
\theading{12n_\ell+4}{(1,2)}
&
\theading{12n_\ell+6}{(3,1)}
&
\theading{12n_\ell+8}{(2,2)}
\\ \hline\hline

\theading{12n_k-2}{(1,1)}
&
\tentry{0}{\namark}{12}{\cmark}
&
\tentry{14}{\xmark}{26}{\xmark}
&
\tentry{4}{\cmark}{16}{\xmark}
&
\tentry{6}{\cmark}{18}{\xmark}
&
\tentry{8}{\cmark}{20}{\xmark}
&
\tentry{10}{\xmark}{22}{\xmark}
\\ \hline

\theading{12n_k+12}{(3,2)}
&
\tentry{-14}{\namark}{-2}{\namark}
&
\tentry{0}{\namark}{12}{\cmark}
&
\tentry{-10}{\namark}{2}{\cmark}
&
\tentry{-8}{\namark}{4}{\cmark}
&
\tentry{-6}{\namark}{6}{\cmark}
&
\tentry{-4}{\namark}{8}{\cmark}
\\ \hline

\theading{12n_k+2}{(2,1)}
&
\tentry{-4}{\namark}{8}{\cmark}
&
\tentry{10}{\xmark}{22}{\xmark}
&
\tentry{0}{\namark}{12}{\cmark}
&
\tentry{2}{\cmark}{14}{\xmark}
&
\tentry{4}{\cmark}{16}{\xmark}
&
\tentry{6}{\cmark}{18}{\xmark}
\\ \hline

\theading{12n_k+4}{(1,2)}
&
\tentry{-6}{\namark}{6}{\cmark}
&
\tentry{8}{\cmark}{20}{\xmark}
&
\tentry{-2}{\namark}{10}{\xmark}
&
\tentry{0}{\namark}{12}{\cmark}
&
\tentry{2}{\cmark}{14}{\xmark}
&
\tentry{4}{\cmark}{16}{\xmark}
\\ \hline

\theading{12n_k+6}{(3,1)}
&
\tentry{-8}{\namark}{4}{\cmark}
&
\tentry{6}{\cmark}{18}{\xmark}
&
\tentry{-4}{\namark}{8}{\cmark}
&
\tentry{-2}{\namark}{10}{\xmark}
&
\tentry{0}{\namark}{12}{\cmark}
&
\tentry{2}{\cmark}{14}{\xmark}
\\ \hline

\theading{12n_k+8}{(2,2)}
&
\tentry{-10}{\namark}{2}{\cmark}
&
\tentry{4}{\cmark}{16}{\xmark}
&
\tentry{-6}{\namark}{6}{\cmark}
&
\tentry{-4}{\namark}{8}{\cmark}
&
\tentry{-2}{\namark}{10}{\xmark}
&
\tentry{0}{\namark}{12}{\cmark}
\\ \hline
\end{tabular}
\end{center}

The desired result is then immediate, because one can observe that for each entry in the table with $\ell - k > 0$, we have 
$\begin{cases}
    \cmark& \text{if } \ell-k \in \{2,4,6,8,12\} \\
    \xmark & \text{if } \ell-k \notin \{2,4,6,8,12\}
\end{cases}$.
\end{proof}

To end the discussion we would like to point out that in certain special cases (specifically, when the weights are multiples of $12$), Markov's result mentioned in the introduction concerning zeros of polynomials can be used to show interlacing  for zeros of $\vartheta_k(E_k)$ in the current setting. 

Recall that the modular discriminant $\Delta$ and $j$-invariant are given by
\begin{align*}
    \Delta=\frac{E_4^3-E_6^2}{1728},\qquad j=\frac{E_4^3}{\Delta},
\end{align*}
with derivatives given by
\begin{align*}
    D(\Delta)=E_2\Delta, \qquad D(j)=-\frac{E_6}{E_4} j.
\end{align*}
(See \cite[paragraph above Theorem 1]{Kaneko-Zagier-1998} and \cite[formula above (38)]{Kaneko-Zagier-1998}.)

Now, suppose $f\in M_k$ for $k=12n$. Then we can write
\begin{align*}
    f=\Delta^n P_f(j),
\end{align*}
where $P_f(x)$ is a real coefficient polynomial of degree (at most) $n$ in $x$. Thus
\begin{align} 
\vartheta_k(f)
&= Df - \frac{k}{12} E_2 f \\
&= Df - nf \frac{D\Delta}{\Delta} \\
&=\Delta^n D(f\Delta^{-n}) \\
&=\Delta^n D(P_f(j))\\
&=\Delta^n { P_f'(j)} D(j)\\
&=\Delta^n P_f'(j) \frac{-E_6}{E_4}\frac{E_4^3}{\Delta}\\
&= -\Delta^{n-1}E_4^2E_6 P_f'(j).\label{eq:Serre-polynomial}
\end{align}

Recall that zeros of $E_k$ for $k=12n$ are all simple and have $j$-invariants in $(0,1728)$ (corresponding to the lower arc), and that the zeros of $E_{k+12}$ and $E_k$ are interlacing by \cite{nozaki-separation}. Hence combining Markov's result and \eqref{eq:Serre-polynomial} shows that the zeros of $\vartheta_{k+12}(E_{k+12})$ and $\vartheta_k(E_k)$ also interlace.

\section{Interlacing of the zeros of \texorpdfstring{$\vartheta_k(E_k)$ and $E_{k+2}$}{Serre derivative of Ek and Ek+2}}

In this section, we prove that the zeros of $F_{k}(\theta)$ interlace with the zeros of $F_k^\sharp(\theta)$, where
\begin{align}
    F^\sharp_{k}(\theta) := -\frac{1}{2}\mathrm{Re}\left(z^{\tfrac{k+2}{2}}E_{k+2}(z)\right) = -\cos\left(\frac{k+2}{2}\theta\right) - \frac{1}{2}d(\theta)^{-k-2} + R_k(\theta),
\end{align}
for $z = e^{i\theta}$. $F^\sharp_{k}(\theta)$ is a real-valued function with the same zeros as $E_{k+2}\left(e^{i\theta}\right)$.

Analogously to our definition of $\hat F_k(\theta)$, we define
\begin{align}
    \hat F^\sharp_{k}(\theta) := -\cos\left(\frac{k+2}{2}\theta\right) - \frac{1}{2}d(\theta)^{-k-2}.\label{eq:F-hat-sharp}
\end{align}
We define $\hat \theta^\sharp_{k,i}$ and $\theta^\sharp_{k,i}$ to be the zeros of $\hat F^\sharp_{k}(\theta)$ and $F^\sharp_{k}(\theta)$ respectively corresponding to $\theta^*_{k,i}$. More precisely, these are the unique zeros of $\hat F^\sharp_{k}(\theta)$ and $F^\sharp_{k}(\theta)$ in the interval $\left(\theta^*_{k,i} - \tfrac{0.25}{k}, \theta^*_{k,i} + \frac{0.25}{k}\right)$. We will show this definition is well-posed in Lemma \ref{lem:theta-star-F-sharp-interval}. Note that these intervals are disjoint because sample zeros are spaced out by at least $\frac{2\pi - 2.4}{k}$ by Lemma \ref{lem:preliminary-estimate}.

By the valence formula,
the number of non-elliptic zeros of $E_{k+2}$ is
$n_k = \dim S_{k+2}$. Moreover, all these zeros lie on the lower arc $\mathcal{A}$, by the work of Rankin and Swinnerton-Dyer \cite{rankin1970zeros}. Therefore, $F^\sharp_{k}(\theta)$ has exactly $n_k$ zeros on $\mathcal{A}$, which is precisely the same number as sample zeros and actual zeros on $\mathcal{A}$. Thus, the zeros $\{\theta^\sharp_{k,i}\}_{1 \leq i \leq n_k}$ constructed above constitute all of the non-elliptic zeros of $F^\sharp_k(\theta)$.

\subsection{Preliminary location estimates}

We can restate Lemma \ref{lem:sample-presample-exact-expression} as
\begin{align}
         \quad \frac{k+2}{2}\theta^*_{k,i}  = \frac{k}{2}\theta^{**}_{k,i} + \frac{\pi}{2} + B_4(\theta^*_{k,i}),\label{eq:theta-star-close-to-mpi-pi-2}
\end{align}
where
\begin{align}
    B_4(\theta) := \theta + B_2(\theta) - \tfrac{\pi}{2}, \quad -0.06 \leq B_4(\theta) \leq 0. \label{eq:b-4-defn}
\end{align}
This will be useful for proving the results in this section. We first find a preliminary estimate for the zeros of $\hat F^\sharp(\theta)$ and $F^\sharp(\theta):$
\begin{lemma}\label{lem:theta-star-F-sharp-interval}
    For all sample zeros $\theta^*_{k,i}$ there is a unique zero of $\hat F^\sharp_{k}(\theta)$ and $F^\sharp_{k}(\theta)$ in the interval $\left(\theta^*_{k,i} - \tfrac{0.25}{k}, \theta^*_{k,i} + \frac{0.25}{k}\right)$.
\end{lemma}
\begin{proof}
    We first verify the lemma for all $k < 1000$ by computation; see \cite{ross-code} for the code. We can then assume $k \geq 1000$ for the rest of the proof. 
    
    \textbf{Existence of zeros of $\hat F^\sharp_k(\theta)$: \\}
    We will actually prove existence of a zero of $\hat F_k^\sharp(\theta)$ within $\frac{0.24}{k}$ of $\theta^*_{k,i}$ to ease our existence proof for $F_k^\sharp(\theta)$. We have
    \begin{align}
         d\left(\theta^*_{k,i} \pm \frac{0.24}{k}\right)^{-k-2} &\leq d\left(\frac{2\pi}{3} - \frac{3.1}{k} + \frac{0.24}{k}\right)^{-k-2 }  &\hspace{-4mm}\text{(by Lemma \ref{lem:sample-zero-at-least-3.1/k-away-from-2pi/3}, since $d(\theta)$ is decreasing)} \\
         &\leq d\left(\frac{2\pi}{3} - \frac{2.86}{k} \right)^{-k}&\text{(since $d$($\theta$) $\geq$ 1)} \\
         &\le d\left(\frac{2\pi}{3} - \frac{2.86}{1000}\right)^{-1000} \\
         &\leq 0.0851.  \label{eq:d-theta-k-2-bound-cuspidal}
    \end{align}
    Then,
    \begin{align}
        & \hat F_{k}^\sharp\left(\theta^*_{k,i} \pm \frac{0.24}{k}\right) \\
        &= -\cos\left(\frac{k+2}{2}\left(\theta^*_{k,i} \pm \frac{0.24}{k}\right)\right) - \frac{1}{2}d\left(\theta^*_{k,i} \pm \frac{0.24}{k}\right)^{-k-2}\\
        &= -\cos\left(\frac{k+2}{2}\theta^*_{k,i} \pm \frac{0.24(k+2)}{2k}\right) - \frac{1}{2}d\left(\theta^*_{k,i} \pm \frac{0.24}{k}\right)^{-k-2} \qquad \\
        &= -\cos\left(\frac{k}{2}\theta^{**}_{k,i} + \frac{\pi}{2} + B_4(\theta^*_{k,i}) \pm \frac{0.24(k+2)}{2k}\right) - \frac{1}{2}d\left(\theta^*_{k,i} \pm \frac{0.24}{k}\right)^{-k-2} \quad\text{(by \eqref{eq:theta-star-close-to-mpi-pi-2})}\\
        &= -(-1)^{i- \delta_k}\cos\left(\frac{\pi}{2} + B_4(\theta^*_{k,i}) \pm \frac{0.24(k+2)}{2k}\right) - \frac{1}{2}d\left(\theta^*_{k,i} \pm \frac{0.24}{k}\right)^{-k-2} \\
        &\quad \left(\text{since }\tfrac{k}{2} \theta_{k,i}^{**} \equiv (i-\delta_k)\pi \mod{2\pi} \text{ by \eqref{eqn:_theta**_equiv_(i-deltak)pi}}\right)\\
        &= (-1)^{i- \delta_k}\sin\left(B_4(\theta^*_{k,i}) \pm \frac{0.24(k+2)}{2k}\right) - \frac{1}{2}d\left(\theta^*_{k,i} \pm \frac{0.24}{k}\right)^{-k-2} \\
        &= \pm(-1)^{i- \delta_k}\sin\left(\frac{0.24(k+2)}{2k} \pm B_4(\theta^*_{k,i})\right) - \frac{1}{2}d\left(\theta^*_{k,i} \pm \frac{0.24}{k}\right)^{-k-2},
    \end{align}
    which has opposite signs since
    \begin{align}
        \sin\left(\frac{0.24(k+2)}{2k}\pm B_4(\theta^*_{k,i})\right) &\geq \sin\left(0.12-0.06\right) \geq 0.0599, \quad \text{and} \\ 
        \left |\frac{1}{2}d\left(\theta^*_{k,i} \pm \frac{0.24}{k}\right)^{-k-2} \right | 
        &\le \frac{0.0851}{2} = 0.04255 
         \qquad \qquad \text{(by \eqref{eq:d-theta-k-2-bound-cuspidal}).}
    \end{align}
    Thus, there must exist a zero $\hat \theta^\sharp$ of $\hat F_{k}^\sharp(\theta)$ in the interval
    \begin{align} \label{eqn:temp-theta-hat-sharp-interval}
        \hat \theta^\sharp \in
        \left(\theta^*_{k,i} - \tfrac{0.24}{k}, \theta^*_{k,i} + \frac{0.24}{k}\right) \subseteq \left(\theta^*_{k,i} - \tfrac{0.25}{k}, \theta^*_{k,i} + \frac{0.25}{k}\right),
    \end{align}
    as desired. \\

    \textbf{Uniqueness of zeros of $\hat F^\sharp_k(\theta)$: \\} By Lemma \ref{lem:f-k-hat-sharp-deriv-bound}, $\hat F^\sharp_k(\theta)$ is monotonic on the interval, so the located zero must be unique. \\

    \textbf{Existence of zeros of $ F^\sharp_k(\theta)$: \\}
    By Proposition \ref{prop:zero-f-sharp-near-f-sharp-hat}, there is a zero $\theta^\sharp$ of $F^\sharp_{k}(\theta)$ such that $ \left|\hat\theta^\sharp_{k,i}-\theta^\sharp\right| < \frac{0.00001}{k^2} < \frac{0.01}{k}$. Combining this with \eqref{eqn:temp-theta-hat-sharp-interval}, we have $\theta^\sharp \in \left(\theta^*_{k,i} - \tfrac{0.25}{k}, \theta^*_{k,i} + \frac{0.25}{k}\right)$, as desired.
    \\

    \textbf{Uniqueness of zeros of $F^\sharp_k(\theta)$: \\}
    Because we located $n_k$ zeros in $n_k$ disjoint intervals, and $F_k^\sharp(\theta)$ has precisely $n_k$ zeros (by the valence formula), there must be exactly one zero in each $\left(\theta^*_{k,i} - \tfrac{0.25}{k}, \theta^*_{k,i} + \frac{0.25}{k}\right)$. This completes the proof and justifies our indexing $\hat \theta^\sharp_{k,i}$ and $\theta^\sharp_{k,i}$.
\end{proof}
We now show the derivative bound used in the proof of Lemma \ref{lem:theta-star-F-sharp-interval}.
\begin{lemma}\label{lem:f-k-hat-sharp-deriv-bound}
    For $k\geq 1000$ and $\theta \in (\theta^{*}_{k,i} - \frac{0.26}{k}, \theta^*_{k,i} + \frac{0.26}{k}),$
    \begin{align}
        \left|\frac{d}{d\theta}\hat F_k^\sharp(\theta) \right| > 0.45k.
    \end{align}
\end{lemma}
\begin{proof}
    By \eqref{eq:theta-star-close-to-mpi-pi-2},
    \begin{align}
        \left |\frac{k+2}{2}\theta  - \frac{k}{2}\theta^{**}_{k,i} - \frac{\pi}{2}\right | &= \left | B_4(\theta^*_{k,i}) + \frac{k+2}{2}\left(\theta - \theta^*_{k,i}\right) \right | \\&\leq 0.06 + \frac{k+2}{2}\left(\frac{0.26}{k}\right) \leq 0.191.\label{eq:deriv-bound-f-sharp-phase}
    \end{align}
    We also obtain by the same argument as for \eqref{eq:d-theta-k-2-bound-cuspidal} that
    \begin{align}
        d(\theta)^{-k-2} \leq d\left(\frac{2\pi}{3} - \frac{3.1}{k} + \frac{0.26}{k}\right)^{-k} \leq 0.086. \label{eq:deriv-bound-f-sharp-d-theta-k-2}
    \end{align}
    Thus
    \begin{align}
        \left|\frac{d}{d\theta}\hat F_k^\sharp(\theta) \right| &= \left |\frac{k+2}{2}\sin\left(\frac{k+2}{2}\theta\right) - \frac{k+2}{4}\tan\left(\frac{\theta}{2}\right)d(\theta)^{-k-2}\right | \\
        &\geq \frac{k+2}{2}\sin\left(\frac{\pi}{2} + 0.191\right) - \frac{k+2}{4}\sqrt{3}\cdot 0.086 &\text{(by \eqref{eq:deriv-bound-f-sharp-phase} and \eqref{eq:deriv-bound-f-sharp-d-theta-k-2})} \\
        &>0.45k,
    \end{align}
    as desired.
\end{proof}

Finally, we prove the zero distance bound we assumed for the $F_k^\sharp(\theta)$ part of Lemma \ref{lem:theta-star-F-sharp-interval}. We will assume only the $\hat F_k^\sharp(\theta)$ part of Lemma \ref{lem:theta-star-F-sharp-interval} (and in particular use the notation $\hat \theta^\sharp_{k,i}$) in the proof.
\begin{proposition}\label{prop:zero-f-sharp-near-f-sharp-hat}
    For $k\ge 1000$, there is a zero $\theta^\sharp$ of $F^\sharp_k(\theta)$ such that
    \begin{align}
        \left|\hat\theta^\sharp_{k,i}-\theta^\sharp\right|\leq\frac{0.006}{2^{\frac{k}{2}}}<\frac{0.00001}{k^2}.
    \end{align}
\end{proposition}
\begin{proof}
Let $\Delta:=\frac{0.006}{2^{\frac{k}{2}}}<\frac{0.00001}{k^2}.$ Note that 
\begin{align}
    &\left(\hat\theta^\sharp_{k,i}-\Delta,\hat\theta^\sharp_{k,i}+\Delta\right)\subseteq\left(\theta^*_{k,i}-\frac{0.26}{k},\theta^*_{k,i}+\frac{0.26}{k}\right), \quad \text{since}\\
    &\left|(\hat\theta^\sharp_{k,i}+\Delta)-\theta^*_{k,i}\right|\leq\frac{0.25}{k}+\frac{0.00001}{k^2}<\frac{0.26}{k}\qquad(\text{by the $\hat F_k^\sharp(\theta)$ part of Lemma \ref{lem:theta-star-F-sharp-interval}}).
\end{align}
It follows from \cite[Section 2]{nozaki-separation} that $\left |F^\sharp_k(\theta) - \hat F^\sharp_k(\theta) \right | \leq\tfrac{2.5}{2^\frac{k}{2}}$. (Note that our $F^\sharp_k(\theta)$ and $\hat F^\sharp_k(\theta)$ are Nozaki's $F_{k+2}(\theta)$ and $F_{k+2}(\theta) - R_{k+2}(\theta)$ scaled by $-\frac{1}{2}$.) Then observe that
\begin{align}
    F^\sharp_k(\hat\theta^\sharp_{k,i}\pm\Delta)&=\hat F^\sharp_k(\hat\theta^\sharp_{k,i}\pm\Delta)+\left(F^\sharp_k(\hat\theta^\sharp_{k,i}\pm\Delta)-\hat F^\sharp_k(\hat\theta^\sharp_{k,i}\pm\Delta)\right)\\
    &=\hat F^\sharp_k(\hat\theta^\sharp_{k,i}\pm\Delta)+\alpha_1\qquad\left(\text{where }|\alpha_1|\leq\tfrac{2.5}{2^\frac{k}{2}}\ \right)\\
    &=\pm\hat F_k^{\sharp}{'}(c^\pm)\Delta+\alpha_1\qquad \left(\text{where}\ c^\pm \text{ lies between} \ \hat\theta^\sharp_{k,i}\text{ and} \ \hat\theta^\sharp_{k,i}\pm\Delta\right),
\end{align}
which has opposite signs since
\begin{align}
    \left|\hat F_k^{\sharp}{'}(c^\pm)\Delta\right| > 0.45k \Delta\ge\frac{2.5}{2^\frac{k}{2}}\ge|\alpha_1|.\qquad \text{for } k \geq 1000 \ \ \text{(by Lemma \ref{lem:f-k-hat-sharp-deriv-bound})}
\end{align}
   Hence there must be at least one zero in the interval $\left(\hat\theta^\sharp_{k,i}-\Delta,\hat\theta^\sharp_{k,i}+\Delta\right)$, completing the proof.
\end{proof}

\subsection{Precise location estimates}

In this subsection, we obtain location estimates on $\hat F^\sharp_{k}(\theta)$ and $ F^\sharp_{k}(\theta)$ that are needed for our interlacing proof.

We will utilize the following lemma, which bounds $\hat F_k(\hat\theta^\sharp_{k,i})$. The proof is rather technical, so we have relegated it to Appendix \ref{sec:proof-of-lem-f-k-theta-sharp-bound}. 
\begin{lemma}\label{lem:f-k-theta-sharp-bound}
    Let $k\ge1000$. Then,
    \begin{align}
        \frac{0.3}{k}<(-1)^{i-\delta_k}\hat F_k(\hat\theta^\sharp_{k,i})<0.1.
    \end{align}
\end{lemma}

We now show that each $\hat \theta_{k,i}^\sharp$ lies just to the right of $\hat \theta_{k,i}$.

\begin{proposition}\label{prop:theta-sharp-postsample-separation}
    For $k \geq 1000$,
    \begin{align}
        \hat \theta_{k,i} +\frac{0.47}{k^2} < \hat \theta_{k,i}^\sharp < \hat \theta_{k,i} + \frac{0.23}{k}. 
    \end{align}
\end{proposition}
\begin{proof}
    We have
    \begin{align}
        \hat F_k(\hat\theta^\sharp_{k,i})-\hat F_k(\hat\theta_{k,i}) = \hat F_k'(c)\left(\hat\theta^\sharp_{k,i}-\hat\theta_{k,i} \right) \qquad \text{for some $c$ between $\hat \theta_{k,i}$ and $\hat\theta^\sharp_{k,i}$}.
    \end{align}
    So,
    \begin{align}
         \hat \theta^\sharp_{k,i} - \hat \theta_{k,i}  =  \frac{ \hat F_k(\hat\theta^\sharp_{k,i})-\hat F_k(\hat\theta_{k,i})}{\hat F_k'(c)}  
        =  \frac{\hat F_k(\hat \theta^\sharp_{k,i})}{\hat F_k'(c)}, 
    \end{align}
    which will always be positive since $\hat F_k(\hat \theta^\sharp_{k,i})$ and $\hat F_k'(c)$ both have sign $(-1)^{i - \delta_k}$ (the second fact coming from Lemma \ref{lem:f-hat-prime-lb}, since $\hat \theta^\sharp_{k,i}$ and $\hat \theta_{k,i}$ are within $\tfrac{0.25}{k}$ of $\theta^*_{k,i}$). So, we have by Lemmas \ref{lem:f-k-theta-sharp-bound} and \ref{lem:f-hat-prime-lb} that
\begin{align}
    \hat \theta^\sharp_{k,i} - \hat \theta_{k,i}&\leq\frac{0.1}{0.45k}<\frac{0.23}{k}, \quad\text{and}\\
     \hat \theta^\sharp_{k,i} - \hat \theta_{k,i}&\ge \frac{\frac{0.3}{k}}{0.63k}>\frac{0.47}{k^2},
\end{align}
as desired.
\end{proof}
\begin{corollary}\label{cor:actual-zero-sharp-zero-gap}
    For $k \geq 1000$,
    \begin{align}
        \theta_{k,i} +\frac{0.46}{k^2} < \theta_{k,i}^\sharp < \theta_{k,i} + \frac{0.24}{k}. 
    \end{align}
\end{corollary}
\begin{proof}
    By Proposition \ref{prop:theta-sharp-postsample-separation},
    \begin{align}
        \hat \theta_{k,i} +\frac{0.47}{k^2} < \hat \theta_{k,i}^\sharp < \hat \theta_{k,i} + \frac{0.23}{k},
    \end{align}
    which is equivalent to
    \begin{align}
        \theta_{k,i} + (\hat \theta_{k,i} - \theta_{k,i}) - (\hat \theta^\sharp_{k,i} - \theta^\sharp_{k,i}) +\frac{0.47}{k^2} < \theta_{k,i}^\sharp < \theta_{k,i} + (\hat \theta_{k,i} - \theta_{k,i}) - (\hat \theta^\sharp_{k,i} - \theta^\sharp_{k,i})+ \frac{0.23}{k}.
    \end{align}
    By our zero distance bounds from Propositions \ref{prop:actual-postsample-zero-bound} and \ref{prop:zero-f-sharp-near-f-sharp-hat}, this implies
    \begin{align}
        \theta_{k,i} + \frac{0.47}{k^2} - 2\cdot \frac{0.00001}{k^2} < \theta^\sharp_{k,i} < \theta_{k,i}  + \frac{0.23}{k} + 2\cdot \frac{0.00001}{k^2},
    \end{align}
    which implies the desired inequality.
\end{proof}

\subsection{Interlacing}

We are finally ready to prove Theorem \ref{thm:interlacing_thetakEk_Ek+2}.
{
\renewcommand{\thetheorem}{\ref{thm:interlacing_thetakEk_Ek+2}}
\addtocounter{theorem}{-1}
\begin{theorem} 
       For all even integers $k \ge 4$, the zeros of $\vartheta_k(E_k)$ interlace with the zeros of $E_{k+2}$ on the lower arc.
\end{theorem} 
}
\begin{proof}
    We verify the cases where $k < 1000$ by computation; see \cite{ross-code} for the code. For $k \geq 1000$, Corollary \ref{cor:actual-zero-sharp-zero-gap}
    states that each $\theta^\sharp_{k,i}$ appears just to the right of $\theta_{k,i}$; namely,
    \begin{align}
        \theta_{k,i} +\frac{0.46}{k^2} < \theta_{k,i}^\sharp < \theta_{k,i} + \frac{0.24}{k}
    \end{align}
    for all $1 \le i \le n_k$.
    Since the $\{\theta_{k,i}\}_{1\le i \le n_k}$ are spaced apart by $\theta_{k,i} - \theta_{k,i+1} \ge \frac{2\pi - 2.4}{k}$, this immediately implies that the $\{\theta_{k,i}^\sharp\}_{1 \le i \le n_k}$ interlace with the $\{\theta_{k,i}\}_{1 \le i \le n_k}$, as desired.   
\end{proof}

\subsection{An application to the cuspidal projection of \texorpdfstring{$\vartheta_k(E_k)$}{thetak(Ek)}}

As an application of Theorem \ref{thm:interlacing_thetakEk_Ek+2}, we show that the zeros of the cuspidal projection of $\vartheta_k(E_k)$ lie on the lower arc $\mathcal{A}$. Observe that for $k\notin\{4,6,8,12\}$, the cuspidal projection of $\vartheta_k(E_k)$
\begin{align}
    \operatorname{cuspproj}\!\big(\vartheta_k(E_k)\big) = \vartheta_k(E_k) + \frac{k}{12}E_{k+2}
\end{align}
is a nonzero cusp form of weight $k + 2$. For $n_k \geq 2$, the nonzero property follows immediately from the fact that $\vartheta_k(E_k) $ and $ \frac{k}{12}E_{k+2}$ have distinct zeros by Theorem \ref{thm:interlacing_thetakEk_Ek+2}. And it is straightforward to check the nonzero property for the remaining cases of  $n_k \leq 1$ by examining Fourier coefficients. Then,
\begin{align}
    F^\text{cproj}_k(\theta) &:= \frac{2\pi}{k}e^{\tfrac{i\theta(k+2)}{2}}\left(\vartheta_k(E_k)(e^{i\theta})+ \frac{k}{12}E_{k+2}(e^{i\theta})\right)\\
    &=\frac{2\pi}{k}e^{\tfrac{i\theta(k+2)}{2}}\vartheta_k(E_k)(e^{i\theta})+\frac{\pi}{6}\cdot e^{\tfrac{i\theta(k+2)}{2}}E_{k+2}(e^{i\theta})\\
    &=F_k(\theta)-\frac{\pi}{3}F^\sharp_k(\theta)\qquad(\text{by definition of $F_k(\theta)$ and $F^\sharp_k(\theta)$})
\end{align}
is a real-valued function with the same zeros as $\operatorname{cuspproj}\!\big(\vartheta_k(E_k)\big)(e^{i\theta})$.

{
\renewcommand{\thetheorem}{\ref{cor:cuspidal-proj}}
\addtocounter{theorem}{-1}
\begin{corollary}
    For all even integers $k \ge 4$ with $k\notin\{4,6,8,12\}$, every (non-cuspidal, non-elliptic) zero of the cuspidal projection $\operatorname{cuspproj}\!\big(\vartheta_k(E_k)\big)$ lies on the lower arc $\mathcal{A}$.
\end{corollary}
}
\begin{proof}
    In Theorem \ref{thm:interlacing_thetakEk_Ek+2}, we showed that the zeros of $F_k$ interlace with the zeros of $F^\sharp_k$. By \cite[Lemma 3]{Xue-Zhu-2021}, this implies that $F^\text{cproj}_k(\theta) = F_k(\theta)-\frac{\pi}{3}F^\sharp_k(\theta)$ has at least $n_k-1$ zeros, so in particular, $\operatorname{cuspproj}\!\big(\vartheta_k(E_k)\big)$ has at least $n_k-1$ zeros on the lower arc.
    
    On the other hand, $\operatorname{cuspproj}\!\big(\vartheta_k(E_k)\big)$ has at most $n_k - 1$ non-cuspidal non-elliptic zeros by the valence formula. (Note that the ``$-1$" comes from the zero of $\operatorname{cuspproj}\!\big(\vartheta_k(E_k)\big)$ at the cusp $i\infty$.) Hence, the $n_k-1$ zeros constructed above must constitute all the non-cuspidal, non-elliptic zeros of $\operatorname{cuspproj}\!\big(\vartheta_k(E_k)\big)$, completing the proof.
\end{proof}

\appendix

\section{Numerical bounds for certain functions of \texorpdfstring{$\theta$}{theta}}

The following bounds can be easily verified via direct computation.
\begin{lemma} \label{lem:function-bounds} 
For $\theta \in \left(\frac{\pi}{2}, \frac{2\pi}{3}\right)$,
    \begin{align}
-0.58&\le B_0(\theta)\le0  \qquad \qquad -1.30\le B_0'(\theta)\le-0.66 \qquad \qquad 0\le B_0''(\theta)\le2.93\\
1&\le B_1(\theta)\le1.16 \qquad \qquad \ \ \  \hspace{.6mm}0\le B_1'(\theta)\le0.42 \qquad\qquad\\
-0.53&\leq B_2(\theta)\hspace{.15mm}\leq0\qquad\qquad-1.30\leq B_2'(\theta)\leq-0.50\qquad\qquad 0\leq B_2''(\theta)\leq2.49\\
-0.58&\leq B_3(\theta)\leq -0.48\qquad-0.67\le B_3'(\theta)\le 0.15\\
-1.05&\leq \hspace{.65mm}U(\theta)\hspace{.65mm}\leq 0 \qquad\qquad \ \hspace{-1mm} -2.60\le U'(\theta)\le-1\qquad 
\qquad\quad\ 0\le U''(\theta)\le 4.97\\
0&\le\hspace{.5mm} V(\theta)\hspace{.5mm}\le1.43\qquad \quad\hspace{-.4mm} -4.20\le V'(\theta)\le6.75.
\end{align}
\end{lemma}

\section{Bounding the difference between \texorpdfstring{$F_k(\theta)$ and $\hat F_k(\theta)$}{F and F hat}}
In this section, we show that the postsample zeros are exponentially close to the actual zeros. Recall that for $z=e^{i\theta}$ with $\theta \in \left(\frac{\pi}{2}, \frac{2\pi}{3}\right)$, 
\begin{align}
    F_k(\theta) &:= \frac{2 \pi}{k} \text{Re} \lrb{ z^\frac{k+2}{2} (D(E_k) - \frac{k}{12}  E_2 E_k) } \\
    \hat F_k(\theta) &:= \frac{2\pi}{k} \text{Re} \lrb{ z^\frac{k+2}{2} (D(\hat E_k) - \frac{k}{12}  E_2 \hat E_k) }.  
\end{align}
Then by the triangle inequality,
\begin{align}
    \left|F_k(\theta)-\hat F_k(\theta)\right|&\leq \left\lvert \frac{2\pi}{k} \textnormal{Re}\lrb{z^\frac{k+2}{2}D(E_k)} - \frac{2\pi}{k} \textnormal{Re}\lrb{ z^\frac{k+2}{2}D(\hat E_k)} \right\lvert + \left\lvert \frac{2\pi}{k} \textnormal{Re}\lrb{\frac{-kz^\frac{k+2}{2}}{12}E_2(E_k-\hat{E_k})}\right\lvert\\
    &\leq \left\lvert \frac{2\pi}{k} \lrb{z^\frac{k+2}{2}D(E_k)} - \frac{2\pi}{k} \lrb{ z^\frac{k+2}{2}D(\hat E_k)} \right\lvert + \left\lvert \frac{2\pi}{k} \lrb{\frac{-kz^\frac{k+2}{2}}{12}E_2(E_k-\hat{E_k})}\right\lvert.
\end{align}
To easily bound these separate terms, we split them into two propositions: Proposition \ref{lem:derivative-estimate-bound} and Proposition \ref{lem:E2Ek-estimate-bound}.

\begin{proposition}\label{lem:derivative-estimate-bound} For $k \geq 22$ and $z=e^{i\theta}$ with $\theta \in \left(\frac{\pi}{2}, \frac{2\pi}{3}\right)$,
    \begin{align}
        \left\lvert \frac{2\pi}{k} \lrb{z^\frac{k+2}{2}D(E_k)} - \frac{2\pi}{k} \lrb{ z^\frac{k+2}{2}D(\hat E_k)} \right\lvert \leq \frac{3.8}{2^\frac{k}{2}}.
    \end{align}
\end{proposition}
\begin{proof}
    We follow the argument of \cite[Proof of Lemma 1]{saha-interlacing}. 
    For clearer notation,  for a function $f(c,d)$ on $\mathbb{Z}^2$ we define 
    \begin{align}
        \sideset{}{^*}\sum  f(c,d) &\ := \sum_{\substack{\gcd(c,d)=1 \\ (c,d) \neq \pm(1,0),\\ \pm(0, 1), \pm(1, 1)}} f(c,d).
    \end{align}
    Then,
    \begin{align}
        &\left \lvert\frac{2\pi}{k}\lrb{z^\frac{k+2}{2}D(E_k-\hat E_k)} \right\lvert \\
        &\leq \left \lvert 2\pi\frac{z^\frac{k+2}{2}}{k}D\left(\frac{1}{2}\sideset{}{^*}\sum \frac{1}{(cz+d)^k}\right)\right\lvert \\
        &= \left \lvert \frac{z^\frac{k+2}{2}}{2}\sideset{}{^*}\sum \frac{c}{(cz+d)^{k+1}}\right\lvert \\
        &= \left \lvert \frac{e^{\frac{i\theta}{2}}}{2}\sideset{}{^*}\sum \frac{c}{\left(ce^{\frac{i\theta}{2}}+de^{\frac{-i\theta}{2}}\right)^{k+1}}\right\lvert.
        \\
        &\leq  \frac{1}{2} \left \lvert \sum_{(c, d) = \pm(1,-1)} \frac{c}{\left(ce^{\frac{i\theta}{2}}+de^{\frac{-i\theta}{2}}\right)^{k+1}}\right\lvert + \frac{1}{2}\left \lvert \sum_{N\geq 5} 
        \sum_{\substack{c^2+d^2=N \\ \text{gcd}(c,d)=1}} \frac{c}{\left(ce^{\frac{i\theta}{2}} + de^{\frac{-i\theta}{2}}\right)^{k+1}}
        \right\lvert\label{eq:A.1-2-sums}.
    \end{align}

    We bound the two terms of \eqref{eq:A.1-2-sums} separately. The first term is
    \begin{align} 
       \frac{1}{2} \left \lvert \sum_{(c, d) = \pm(1,-1)} \frac{c}{\left(ce^{\frac{i\theta}{2}}+de^{\frac{-i\theta}{2}}\right)^{k+1}}\right\lvert= \left \lvert \frac{1}{\left(e^{\frac{i\theta}{2}}-e^{\frac{-i\theta}2}\right)^{k+1}}\right\lvert \leq \left \lvert \frac{1}{e^{\frac{i\pi}4} - e^{\frac{-i \pi}4}}\right \rvert^{k+1} = \left(\frac{1}{\sqrt{2}}\right)^{k+1}.
    \end{align}
    Then, using
    \begin{align} 
        \left\lvert ce^{\frac{i\theta}2} + de^{\frac{-i\theta}{2}}\right\lvert^2 \geq \frac{c^2 + d^2}{2}, \quad \#\{(c,d): c^2 + d^2 = N\} \leq 4N^{\frac{1}2} + 2, \quad |c| \leq N^{\frac{1}{2}},
    \end{align}
    we obtain for the second term
    \begin{align}
        \frac{1}{2} \left \lvert \sum_{N \geq 5} \sum_{\substack{\text{gcd}(c,d)=1 \\ c^2 + d^2 = N}} \frac{c}{\left(ce^{\frac{i\theta}{2}} + de^{-\frac{i\theta}{2}}\right)^{k+1}}\right \lvert
        &\leq \frac{1}{2}\sum_{N \geq 5} \frac{N^{\frac{1}2}}{\left(\frac{N}{2}\right)^{\frac{k+1}2}}\left(4N^{\frac{1}2}+2\right)  \\ 
        & \leq 2^{\frac{k+1}{2}}\sum_{N \geq 5} \left (\frac{2}{N^{\frac{k-1}{2}}} + \frac{1}{N^{\frac{k-1}{2}}\sqrt{5}}\right)  \qquad(\text{since } N \geq 5) \\
        &\leq 2^{\frac{k+1}{2}}\left(2 + \frac{1}{\sqrt{5}}\right) \int_4^\infty x^{\frac{1-k}{2}} \ dx \\
        &= 2^{\frac{k+1}{2}} \left(2 + \frac{1}{\sqrt{5}} \right) \frac{2^{4-k}}{k-3} \\
        &\leq \frac{56}{(k-3)2^{\frac{k}{2}}}.
    \end{align}
    Thus, we have the upper bound
    \begin{align}
        \left \lvert\frac{2\pi}{k}\lrb{z^\frac{k+2}{2}D(E_k-\hat E_k)} \right\lvert \le\frac{1}{2^{\frac{k+1}{2}}} + \frac{56}{(k-3)2^{\frac{k}{2}}} \leq \frac{3.8}{2^{\frac{k}{2}}} \quad \text{for } k \geq 22,
    \end{align}
    as desired.
\end{proof}

\begin{proposition}\label{lem:E2Ek-estimate-bound}  For $k \geq 22$ and $z=e^{i\theta}$ with $\theta \in \left(\frac{\pi}{2}, \frac{2\pi}{3}\right)$,
    \begin{align}
        \left\lvert \frac{2\pi}{k} \lrb{\frac{-kz^{\frac{k+2}{2}}}{12}E_2(E_k-\hat{E_k})}\right\lvert \le \frac{6.7}{2^{\frac{k}2}}.
    \end{align}
\end{proposition}
\begin{proof}
We have
 \begin{align}
        \left\lvert \frac{2\pi}{k} \lrb{\frac{-kz^\frac{k+2}{2}}{12}E_2(E_k-\hat{E_k})}\right\lvert\leq \frac{\pi}{6}\left\lvert {z^\frac{k}{2}E_2(E_k-\hat E_k)}\right\rvert
        =\frac{\pi}{6} \left\lvert E_2\right\rvert   \left\lvert z^\frac{k}{2}(E_k-\hat E_k) \right\rvert.
    \end{align}
    Now, by \cite[Lemma 2.3]{imamoglu2014estimates}, we have
    \begin{align}
        \frac{\pi}{6}\left\lvert E_2(e^{\textit{i}\theta}) \right\rvert &\leq \frac{4\pi e^{-2\pi\sin(\theta)}}{(1-e^{-2\pi\sin(\theta)})^3}+\frac{\pi}{6}\\
        &\leq \frac{4\pi e^{\frac{-2\pi\sqrt{2}}{3}}}{(1-e^{\frac{-2\pi\sqrt{2}}{3}})^3}+\frac{\pi}{6}\qquad(\text{since $\tfrac{\pi}{2}<\theta<\tfrac{2\pi}{3}$})\\
        &\leq1.29.\label{eq:B.2-pt-1}
    \end{align}
    
Next, for $N\geq1$ define
    \begin{align}
        \sigma_N(\theta):=\frac{1}{2}\sideset{}{^*}\sum \frac{1}{\left(ce^{\frac{i\theta}{2}}+de^{-\frac{i\theta}{2}}\right)^k}.
    \end{align}
    and observe that
    \begin{align}
        \sigma_1(\theta)&=0, \quad \sigma_2(\theta)=\frac{(-1)^{\frac{k}{2}}}{(2\sin(\frac{\theta}{2}))^k} \\
        \left\lvert \sum_{N\geq5}\sigma_N(\theta) \right\rvert&\leq \left( 2+\frac{1}{\sqrt5}\right)\left(\frac{2^\frac{k+2}{2}}{k-3}\right)\left(\frac{1}{4}\right)^\frac{k-3}{2} \qquad \text{(by \cite[Proof of Lemma 1]{saha-interlacing})}.
    \end{align}

Then, using a similar argument as in the proof of Proposition \ref{lem:derivative-estimate-bound}, we obtain
    \begin{align}
        \left\lvert z^\frac{k}{2}(E_k-\hat{E_k})\right\rvert&=\left\lvert \frac{z^\frac{k}{2}}{2}\sideset{}{^*}\sum \frac{1}{(cz+d)^k}\right\rvert\\
        &=\left\lvert \frac{1}{2}\sideset{}{^*}\sum \frac{1}{\left(ce^{\frac{i\theta}{2}}+de^{-\frac{i\theta}{2}}\right)^k}\right\rvert\\
        &= \left\lvert \frac{(-1)^{\frac{k}{2}}}{\left(2\sin\left(\frac{\theta}{2}\right)\right)^k}+\sum_{N\geq5}\sigma_N(\theta) \right\rvert\\
       &\leq\left(\frac{1}{2^\frac{k}{2}}\right)+ \left( 2+\frac{1}{\sqrt5}\right)\left(\frac{2^\frac{k+2}{2}}{k-3}\right)\left(\frac{1}{4}\right)^\frac{k-3}{2}. \label{eq:B.2-pt-2}
    \end{align}
   
    Combining \eqref{eq:B.2-pt-1} and \eqref{eq:B.2-pt-2}, we obtain 
    \begin{align}
        \frac{\pi}{6} \left\lvert E_2\right\rvert   \left\lvert z^\frac{k}{2}(E_k-\hat{E_k}) \right\rvert&\leq 1.29\left[\left(\frac{1}{2^\frac{k}{2}}\right)+ \left( 2+\frac{1}{\sqrt5}\right)\left(\frac{2^{{\frac{k+2}{2}}}}{k-3}\right)\left(\frac{1}{4}\right)^{\frac{k-3}{2}}\right]\\
        &\leq1.29\left[\frac{1}{2^\frac{k}{2}}+\frac{78.32}{\left(2^{\frac{k}{2}}\right)(k-3)}\right]\\
        &\leq \frac{6.7}{2^\frac{k}{2}} \quad \text{for } k \geq 22,
    \end{align}
    as desired.
\end{proof}

Propositions \ref{lem:derivative-estimate-bound} and \ref{lem:E2Ek-estimate-bound} then immediately imply the following corollary.
\begin{corollary}\label{cor:bound-for-Fk-minus-Fhatk} For $k \geq 10$ and $\theta \in \left(\frac{\pi}{2}, \frac{2\pi}{3}\right)$,
    \begin{align}
        \left \lvert F_k(\theta) - \hat F_k(\theta) \right \lvert \leq \min\lrp{
            \frac{10.5}{2^{\frac{k}2}},
            0.016
        }
    \end{align}
\end{corollary}
\begin{proof}
    The desired bound for $k \ge 22$ comes directly from Propositions \ref{lem:derivative-estimate-bound} and \ref{lem:E2Ek-estimate-bound}. And we have verified this result by direct computation for $10 \le k \le 20$; see \cite{ross-code} for the code.
\end{proof}

\section{Distance between postsample and actual zeros}

We now obtain a lower bound on the derivative of $\hat F_k(\theta)$ for points near its zero $\hat \theta_{k,i}$. We assume Lemma \ref{lem:d-theta-k-deriv-bound-appendix}, which we will prove shortly.
\begin{lemma}\label{lem:f-hat-prime-lb}
    For $k \geq 1000$ and $\theta \in \left(\theta^*_{k,i} - \frac{0.25}{k}, \theta^*_{k,i} + \frac{0.25}{k} \right)$, we have
    \begin{align}
        0.45k < (-1)^{i - \delta_k}\hat F'_k(\theta) < 0.63k
    \end{align}
\end{lemma}

\begin{proof}
    For $\theta \in \left(\theta^*_{k,i} - \frac{0.25}{k}, \theta^*_{k,i} + \frac{0.25}{k} \right)$, 
    \begin{align}
        \frac{k}{2}\theta - B_2(\theta)
        &= \frac{k}{2}\theta^*_{k,i} + \frac{k}{2} (\theta - \theta^*_{k,i}) 
        - B_2(\theta^*_{k,i})
        - B'_2(c_1)(\theta - \theta^*_{k,i} ) 
        \qquad\quad\ \,\\&\qquad \text{(for some $c_1$ between $\theta^*_{k,i} $ and $\theta$)} \\
        &=  \frac{k}{2}\theta^{**}_{k,i} + (\theta - \theta^*_{k,i})\left(\frac{k}{2} - B'_2(c_1)\right)
        \qquad  \ \
        (\text{by Lemma \ref{lem:sample-presample-exact-expression}}).
    \end{align}
    This implies that for $k \geq 1000$,
    \begin{align}
        \lrabs{\lrp{ \frac{k}{2}\theta - B_2(\theta) } - \frac{k}{2}\theta^{**}_{k,i}}
        &\le
        \lrabs{
        (\theta - \theta^*_{k,i})\left(\frac{k}{2} - B'_2(c_1)\right)
        } \\
        &\le
        \left(\frac{0.25}{k}\right) 
        \left(\frac{k}{2} + 1.3 \right) \leq 0.1254 \quad \text{(by Lemma \ref{lem:function-bounds})}.
        \label{eq:k-2-b-2-term-}
    \end{align}
    We then have $\frac{k}{2}\theta - B_2(\theta) =\frac{k}{2}\theta^{**}_{k,i} + \upsilon$, $\left | \upsilon\right | \leq 0.1254$, so
    \begin{align}
        \hat F_k'(\theta)
        &= \frac{d}{d\theta}\lrb{B_1(\theta)\sin\left(\frac{k}{2}\theta - B_2(\theta)\right) + B_3(\theta)d(\theta)^{-k}} \\
        &=  B_1(\theta)\left(\frac{k}{2} - B'_2(\theta)\right)\cos\left(\frac{k}{2}\theta - B_2(\theta)\right) + B_1'(\theta)\sin\left(\frac{k}{2}\theta - B_2(\theta)\right) + \frac{d}{d\theta}\lrb{B_3(\theta)d(\theta)^{-k}} \\
        &=  B_1(\theta)\left(\frac{k}{2} - B'_2(\theta)\right)\cos\left(\frac{k}{2}\theta^{**}_{k,i} +\upsilon\right) + B_1'(\theta)\sin\left((i- \delta_k)\pi +\upsilon\right) + \frac{d}{d\theta}\lrb{B_3(\theta)d(\theta)^{-k}} \\
        &= (-1)^{i - \delta_k}B_1(\theta)\left(\frac{k}{2} - B'_2(\theta)\right)\cos\left(\upsilon\right) + (-1)^{i-\delta_k}B_1'(\theta)\sin(\upsilon) + \frac{d}{d\theta}\lrb{B_3(\theta)d(\theta)^{-k}} \\
        &\qquad \text{(since $\frac{k}{2}\theta^{**}_{k,i} \equiv (i - \delta_k)\pi$ mod 2$\pi$ by \eqref{eqn:_theta**_equiv_(i-deltak)pi})}.
    \end{align}
    Thus
    \begin{align}
        (-1)^{i - \delta_k}\hat F'_k(\theta) &\geq B_1(\theta)\left(\frac{k}{2} - B'_2(\theta)\right) \cos(0.1254 ) - B_1'(\theta)\sin(0.1254) - (0.043k + 0.0571)\\ &\qquad \text{(by Lemma \ref{lem:d-theta-k-deriv-bound-appendix})}\\
        &\geq \frac{k}{2}\cos\left(0.1254\right) -  0.42\sin\left(0.1254\right) - (0.043k+0.0571) \qquad (\text{by Lemma \ref{lem:function-bounds})} \\
        &> 0.45k
    \end{align}
    and
    \begin{align}
        (-1)^{i - \delta_k}\hat F'_k(\theta) & \leq B_1(\theta)\left(\frac{k}{2} - B'_2(\theta)\right) \cdot 1 + B_1'(\theta)\sin(0.1254) + (0.043k + 0.0571) \\ &\qquad \text{(by Lemma \ref{lem:d-theta-k-deriv-bound-appendix})}\\
        &\leq 1.16\left(\frac{k}{2} + 1.3\right) + 0.42\sin\left(0.1254\right) + 0.043k + 0.0571 \\
        &\qquad(\text{by Lemma \ref{lem:function-bounds})} \\
        &<0.63k,
    \end{align}
completing the proof.
\end{proof}
We now prove the lemma we assumed.
\begin{lemma}\label{lem:d-theta-k-deriv-bound-appendix}
     For $k \geq 1000$ and $\theta \in \left( \theta^*_{k,i} - \frac{0.25}{k}, \theta^*_{k,i} + \frac{0.25}{k} \right)$, 
    \begin{align}
        \left | \frac{d}{d\theta} \lrb{B_3(\theta)d(\theta)^{-k}}\right | \leq 0.043k+0.0571.
    \end{align}
\end{lemma}
\begin{proof}
    We have by Lemma \ref{lem:sample-zero-at-least-3.1/k-away-from-2pi/3} that
    \begin{align}
        \theta = \theta^*_{k,i} + (\theta - \theta^*_{k,i})\leq \frac{2\pi}{3} - \frac{3.1}{k} + \frac{0.25}{k}= \frac{2\pi}{3} - \frac{2.85}{k}.\label{eq:max-value-theta-1-k2-away-postsample}
    \end{align}
    Then,
    \begin{align}
         d\left(\theta^*_{k,i} \pm \frac{0.25}{k}\right)^{-k}
         &\leq d\left(\frac{2\pi}{3} - \frac{2.85}{k}\right)^{-k}&\text{(since $d$($\theta$) is decreasing)}\\
         &\leq d\left(\frac{2\pi}{3} - \frac{2.85}{1000}\right)^{-1000}\\ 
         &\leq 0.0851. \qquad &\qquad \label{eq:d-theta-k-bound}
    \end{align}
    Thus,
    \begin{align}
        \left | \frac{d}{d\theta} \lrb{ B_3(\theta)d(\theta)^{-k}}\right | &= \left |B_3'(\theta)d(\theta)^{-k} + k\sin\left(\frac{\theta}{2}\right)B_3(\theta)d(\theta)^{-k-1} \right | \\
        &= \left |B_3'(\theta)d\left(\theta\right)^{-k} + \frac{k}{2}\tan\left(\frac{\theta}{2}\right)B_3(\theta)d\left(\theta\right)^{-k}\right | \\
        &\leq 0.67d(\theta)^{-k} + \frac{k}{2}\sqrt{3}(0.58)d(\theta)^{-k} &(\text{by Lemma \ref{lem:function-bounds})} \\
        &\leq 0.0851\left(0.67 + \frac{k}{2}\sqrt{3}(0.58)\right) &\text{(by \eqref{eq:d-theta-k-bound})} \\
        &\leq 0.043k+0.0571,  \label{eq:appendix-deriv-bound-b3-dtheta}
    \end{align}
    as desired.
\end{proof}
We finally bound the difference between postsample and actual zeros.
\begin{proposition}\label{prop:actual-postsample-zero-bound}
    For $k\geq 1000$, we have
    \begin{align}
        \left \lvert \hat\theta_{k, i} - \theta_{k,i}\right \lvert \leq \frac{0.025}{2^\frac{k}{2}} < \frac{0.00001}{k^2}.
    \end{align}
\end{proposition}
\begin{proof}
    Let
    $\Delta :=  \frac{0.025}{2^\frac{k}{2}} < \frac{0.00001}{k^2}$. Note that
    \begin{align}
    \left(\hat \theta_{k,i} - \Delta, \hat \theta_{k,i} + \Delta\right) &\subseteq \left(\theta^*_{k,i} - \frac{0.25}{k}, \theta^*_{k,i} + \frac{0.25} {k}\right), \quad \text{since} \\
        \left|(\hat \theta_{k,i} + \Delta) - \theta^*_{k,i} \right| &\leq \frac{0.1}{k} + \frac{0.6}{k^2} + \Delta\leq \frac{0.25}{k} &\text{(for $k \geq 1000$ by Proposition \ref{prop:postsample-sample-distance-formula})}.
    \end{align}
    Then observe that
    \begin{align}
        F_k(\hat \theta_{k,i} \pm \Delta) 
        &= \hat F_k(\hat \theta_{k,i} \pm \Delta) + \lrp{F_k(\hat \theta_{k,i} \pm \Delta) - \hat F_k(\hat \theta_{k,i} \pm \Delta)} \\
        &= \hat F_k(\hat \theta_{k,i} \pm \Delta) + \alpha_1 \qquad \text{where } |\alpha_1| \le\tfrac{10.5}{2^\frac{k}{2}}\ \text{ by Corollary \ref{cor:bound-for-Fk-minus-Fhatk}}\\
        &= \pm \hat F_k'(c^{\pm}) \Delta  + \alpha_1, \qquad \text{where $c^{\pm}$ lies between $\hat \theta_{k,i}$ and $\hat \theta_{k,i} \pm \Delta$} 
    \end{align}
    which has opposite signs since 
    \begin{align}
        \lrabs{\hat F_k'(c^{\pm}) \Delta} \ge 0.45k\Delta > \frac{10.5}{2^\frac{k}{2}} \geq \left |\alpha_1 \right | 
        \qquad \text{(by Lemma \ref{lem:f-hat-prime-lb})}.
    \end{align}
    This proves that $\theta_{k,i} \in \lrb{\hat \theta_{k,i} - \Delta, ~ \hat \theta_{k,i} + \Delta}$, completing the proof. 
\end{proof}

\section{Proof of Lemma \ref{lem:f-k-theta-sharp-bound}}\label{sec:proof-of-lem-f-k-theta-sharp-bound}
We first prove two function identities, then prove Lemma \ref{lem:f-k-theta-sharp-bound}.
\begin{lemma}\label{lem:B-function-identities}
    We have
    \begin{align}
        B_3(\theta) &= -\tfrac{1}{2}\left(\tfrac{\sin(\theta)}{1 + \cos(\theta)} + B_0(\theta)\right)\quad \text{and} \quad
        \cos\left(\theta\right) - B_0(\theta)\sin(\theta) = B_1(\theta)\sin(-B_4(\theta)).
    \end{align}
\end{lemma}
\begin{proof}
    The first claim is immediate from the definition of $B_3(\theta)$ and the fact that $\tan\left(\frac{\theta}{2}\right) = \frac{\sin(\theta)}{1 + \cos(\theta)}$.

    It remains to prove the second claim. Observe that
    \begin{align}
        B_1(\theta)\sin(-B_4(\theta)) &= B_1(\theta)\sin\left(\frac{\pi}{2} - \theta - B_2(\theta)\right) \\
        &= B_1(\theta)\cos(\theta + B_2(\theta)) \\
        &= B_1(\theta)\left(\cos(\theta)\cos(B_2(\theta)) - \sin(\theta)\sin(B_2(\theta))\right) \\
        &= B_1(\theta)\left(\cos(\theta)\cdot \frac{1}{\sqrt{1 + B_0(\theta)^2}} - \sin(\theta)\frac{B_0(\theta)}{\sqrt{1 + B_0(\theta)^2}}\right) \\
        &= B_1(\theta)\left(\frac{\cos(\theta)}{B_1(\theta)} - \frac{\sin(\theta)B_0(\theta)}{B_1(\theta)}\right) \\
        &= \cos(\theta) - B_0(\theta)\sin(\theta),
    \end{align}
    as desired.
\end{proof}
    We now bound $(-1)^{i-\delta_k}\hat F_k(\hat \theta^\sharp_{k,i})$ in Lemma \ref{lem:f-k-theta-sharp-bound}.
{
\renewcommand{\thetheorem}{\ref{lem:f-k-theta-sharp-bound}}
\addtocounter{theorem}{-1}
\begin{lemma}
     Let $k\ge1000$. Then,
    \begin{align}
        \frac{0.3}{k}<(-1)^{i-\delta_k}\hat F_k(\hat\theta^\sharp_{k,i})<0.1.
    \end{align}
\end{lemma}
}

\begin{proof}
    Using \eqref{eq:F-hat-second}, we write
    \begin{align}
        &\hat F_k(\hat \theta^\sharp_{k,i})\\ &= \sin\left(\frac{k+2}{2}\hat \theta^\sharp_{k,i}-\hat \theta^\sharp_{k,i}\right)-B_0(\hat \theta^\sharp_{k,i})\cos\left(\frac{k+2}{2}\hat \theta^\sharp_{k,i}-\hat \theta^\sharp_{k,i}\right)+B_3(\hat \theta^\sharp_{k,i})d(\hat \theta^\sharp_{k,i})^{-k} \\
        &= \sin\left(\frac{k+2}{2}\hat \theta^\sharp_{k,i}\right)\left(\cos\left(\hat \theta^\sharp_{k,i}\right) - B_0(\hat \theta^\sharp_{k,i})\sin\left(\hat \theta^\sharp_{k,i}\right)\right) \\ &\quad + \frac{1}{2}d(\hat \theta^\sharp_{k,i})^{-k-2}\left(\sin\left(\hat \theta^\sharp_{k,i} \right) + B_0(\hat \theta^\sharp_{k,i})\cos\left(\hat \theta^\sharp_{k,i}\right)\right) + B_3(\hat \theta^\sharp_{k,i})d(\hat \theta^\sharp_{k,i})^{-k} \\
        &\qquad \left(\text{by expanding $\sin$ and $\cos$, using }-\cos\left(\tfrac{k+2}{2}\hat \theta^\sharp_{k,i}\right) = \tfrac{1}{2}d(\hat \theta^\sharp_{k,i})^{-k-2} \text{ from def. of $\hat F^\sharp_{k}(\theta)$}\right) \\
        &= \sin\left(\frac{k+2}{2}\hat \theta^\sharp_{k,i}\right)\left(\cos\left(\hat \theta^\sharp_{k,i}\right) - B_0(\hat \theta^\sharp_{k,i})\sin\left(\hat \theta^\sharp_{k,i}\right)\right) \\ &\quad + \frac{d(\hat \theta^\sharp_{k,i})^{-k}}{4\left(1 + \cos\left(\hat \theta^\sharp_{k,i}\right)\right)}\left(\sin\left(\hat \theta^\sharp_{k,i} \right) + B_0(\hat \theta^\sharp_{k,i})\cos\left(\hat \theta^\sharp_{k,i}\right)\right) 
         - \frac{1}{2}\left(\frac{\sin(\hat \theta^\sharp_{k,i})}{1 + \cos(\hat \theta^\sharp_{k,i})} + B_0(\hat \theta^\sharp_{k,i})\right)d(\hat \theta^\sharp_{k,i})^{-k} \\[-0.35cm]
        &\qquad \qquad  \left(\text{using $d(\theta)^2 = 2(1 + \cos(\theta))$, and Lemma \ref{lem:B-function-identities}}\right) \\
        &= \sin\left(\frac{k+2}{2}\hat \theta^\sharp_{k,i}\right)\left(\cos\left(\hat \theta^\sharp_{k,i}\right) - B_0(\hat \theta^\sharp_{k,i})\sin\left(\hat \theta^\sharp_{k,i}\right)\right) \\ &\quad - \left(-\frac{\left(\sin\left(\hat \theta^\sharp_{k,i} \right) + B_0(\hat \theta^\sharp_{k,i})\cos\left(\hat \theta^\sharp_{k,i}\right)\right)}{4\left(1 + \cos\left(\hat \theta^\sharp_{k,i}\right)\right)} 
         + \frac{1}{2}\left(\frac{\sin(\hat \theta^\sharp_{k,i})}{1 + \cos(\hat \theta^\sharp_{k,i})} + B_0(\hat \theta^\sharp_{k,i})\right)\right)d(\hat \theta^\sharp_{k,i})^{-k} \\
        &=: \sin\left(\frac{k+2}{2}\hat \theta^\sharp_{k,i}\right)\left(\cos\left(\hat \theta^\sharp_{k,i}\right) - B_0(\hat \theta^\sharp_{k,i})\sin\left(\hat \theta^\sharp_{k,i}\right)\right) - Q(\hat \theta^\sharp_{k,i})d(\hat \theta^\sharp_{k,i})^{-k} \\
        & \qquad \qquad \qquad \left( \text{where } Q(\theta) := \tfrac{\sin\left(\theta\right) + B_0(\theta)(2 + \cos(\theta))}{4(1 + \cos(\theta))} \right) \\
        &= \sin\left(\frac{k+2}{2}\hat \theta^\sharp_{k,i}\right)B_1(\hat \theta^\sharp_{k,i})\sin\left(-B_4(\hat \theta^\sharp_{k,i})\right) -  Q(\hat \theta^\sharp_{k,i})d(\hat \theta^\sharp_{k,i})^{-k}. \qquad \left(\text{by Lemma \ref{lem:B-function-identities}}\right)
    \end{align}
    Thus, by assuming \eqref{eq:sin-theta-sharp-lb}, \eqref{eq:B-sin-B-4-bounds}, and \eqref{eq:Q-bounds} (which we prove subsequently), we have
    \begin{align}
        &(-1)^{i - \delta_k}\hat F_k(\hat \theta^\sharp_{k,i})\\ &= (-1)^{i - \delta_k}\sin\left(\frac{k+2}{2}\hat \theta^\sharp_{k,i}\right)B_1(\hat \theta^\sharp_{k,i})\sin\left(-B_4(\hat \theta^\sharp_{k,i})\right) -  (-1)^{i-\delta_k}Q(\hat \theta^\sharp_{k,i})d(\hat \theta^\sharp_{k,i})^{-k} \\
        &\geq 0.999 \cdot \frac{1}{5}\min \left \{\hat \theta^\sharp_{k,i} - \frac{\pi}{2}, \frac{2\pi}{3} - \hat \theta^\sharp_{k,i}\right \} - \frac{3}{5}\left(\frac{2\pi}{3} - \hat \theta^\sharp_{k,i}\right) \exp\left(-0.66k\left(\frac{2\pi}{3} - \hat \theta^\sharp_{k,i} \right)\right) \\
        &\qquad(\text{using \eqref{eq:sin-theta-sharp-lb}, \eqref{eq:B-sin-B-4-bounds}, \eqref{eq:Q-bounds}, and Lemma \ref{lem:g-and-h-deriv-values})}\\
        &\geq 0.999 \cdot \frac{1}{5}\cdot\frac{3.1 - 0.25}{k} - \frac{3}{5}\left(\frac{2\pi}{3} - \hat \theta^\sharp_{k,i}\right) \exp\left(-0.66k\left(\frac{2\pi}{3} - \hat \theta^\sharp_{k,i}\right)\right)  \\
        &\qquad \text{(by Lemmas \ref{lem:sample-zero-at-least-3.1/k-away-from-2pi/3} and \ref{lem:theta-star-F-sharp-interval}}) \\
        &= \frac{1}{5k}\left(0.999 \cdot 2.85 - 3 \cdot k\left(\frac{2\pi}{3} - \hat \theta^\sharp_{k,i}\right) \exp\left(-0.66k\left(\frac{2\pi}{3} - \hat \theta^\sharp_{k,i}\right)\right)\right) \\
        &\geq \frac{1}{5k}\left(0.999 \cdot 2.85  - 3\left(2.85\exp\left(-0.66\cdot 2.85\right)\right)\right) \\
    &\qquad \left(\text{since the map } x \mapsto x \exp\left(-0.66x\right) \text{ is decreasing for } x \geq 2.85 \right) \\
        &> \frac{0.3}{k}. \label{eq:f-hat-expansion-at-theta-sharp}
    \end{align}
    Similarly, we obtain
    \begin{align}
        &(-1)^{i - \delta_k}\hat F_k(\hat \theta^\sharp_{k,i}) \\&= (-1)^{i - \delta_k}\sin\left(\frac{k+2}{2}\hat \theta^\sharp_{k,i}\right)B_1(\hat \theta^\sharp_{k,i})\sin\left(-B_4(\hat \theta^\sharp_{k,i})\right) -  (-1)^{i-\delta_k}Q(\hat \theta^\sharp_{k,i})d(\hat \theta^\sharp_{k,i})^{-k} \\ 
        &\leq 1 \cdot 1.16 \sin\left(0.06\right) + \frac{3}{5}\left(\frac{2\pi}{3} - \theta\right)d(\hat \theta^\sharp_{k,i})^{-k} \qquad \text{(by Lemma \ref{lem:function-bounds}, \eqref{eq:b-4-defn}, and \eqref{eq:Q-bounds})}
        \\&\leq 1 \cdot 1.16\sin(0.06) + \frac{3}{5}\left(\frac{2\pi}{3} - \frac{\pi}{2}\right) \cdot 0.0851
        \qquad \quad \ \ \qquad \text{(by Lemma \ref{lem:theta-star-F-sharp-interval} and \eqref{eq:d-theta-k-bound})} \\
        &\leq 0.1.
    \end{align}
    We now prove our assumptions. It follows from \eqref{eq:theta-star-close-to-mpi-pi-2} and Lemma \ref{lem:theta-star-F-sharp-interval} that $(-1)^{i - \delta_k}\sin\left(\frac{k+2}{2}\hat \theta^\sharp_{k,i}\right)$ 
    is positive, and can be bounded below by
    \begin{align}
        (-1)^{i-\delta_k} \sin\left(\frac{k+2}{2}\hat \theta^\sharp_{k,i}\right) &= \sqrt{1 - \cos^2\left(\frac{k+2}{2}\hat \theta^\sharp_{k,i}\right)} \\ 
        &=  \sqrt{1 - \frac{1}{4}d(\hat \theta^\sharp_{k,i})^{-2(k+2)}} &\text{(by definition of $\hat F_k^\sharp(\theta)$)}\\
        &\geq \sqrt{1-\frac{1}{4}0.0851^2} &\text{(using \eqref{eq:d-theta-k-2-bound-cuspidal})} \\
        &> 0.999 \label{eq:sin-theta-sharp-lb}.
    \end{align}
    It can also be easily shown through direct computation that
    \begin{align}
        B_1(\theta)\sin\left(-B_4(\theta)\right) &\geq \frac{1}{5}\min \left \{\theta - \frac{\pi}{2}, \frac{2\pi}{3} - \theta\right \} > 0, \label{eq:B-sin-B-4-bounds} \\
        \text{and}\ \ \ \frac{3}{5}\left(\frac{2\pi}{3} - \theta\right) &\geq Q(\theta) \geq 0. \label{eq:Q-bounds}
    \end{align}
    This completes the proof.
\end{proof}

\section{Proof of Lemma \ref{lem:zero-endpoint-behavior-k-ell}}\label{app:sample-zero-at-edges}

{
\renewcommand{\thetheorem}{\ref{lem:zero-endpoint-behavior-k-ell}}
\addtocounter{theorem}{-1}
\begin{lemma} 
    Let $\ell > k \geq 1000$. Then 
    \begin{align}
        \theta^*_{\ell,1} > \theta^*_{k,2} + \frac{3(\ell -k)}{\ell k}
        \quad  \text{and} \quad
        \theta^*_{\ell,n_\ell} <\theta^*_{k,n_k-1} - \frac{3(\ell - k)}{\ell k}. 
    \end{align}
\end{lemma}
}
\begin{proof}
We begin with the first claim of the lemma. By Proposition \ref{prop:sample-zero-locations}, we have
    \begin{align}
        \theta^*_{\ell,1} - \theta^*_{k,2} = \ &\theta^{**}_{\ell,1} - \theta^{**}_{k, 2}
                + \frac{U(\theta^{**}_{\ell,1})k - U(\theta^{**}_{k,2})\ell}{\ell k} + \frac{V(\theta^{**}_{\ell,1})k^2 - V(\theta^{**}_{k,2})\ell^2}{\ell^2k^2}
                + \frac{\alpha^*_{\ell,1}k^3 - \alpha^*_{k,2}\ell^3}{\ell^3k^3}.
    \end{align}
From \eqref{eq:endpoint_theta_1}, and noting that $\gamma_k\geq1$ and $\gamma_\ell\leq3$, we have 
    \begin{align}
        \theta^{**}_{\ell, 1} - \theta^{**}_{k, 2}=
        \theta^{**}_{\ell, 1} - \left(\theta^{**}_{k, 1}-\frac{2\pi}{k}\right)
        &=\left(\frac{2\pi}{3} - \frac{\gamma_\ell}{3}\frac{2\pi}{\ell}\right)-\left(\frac{2\pi}{3} - \frac{\gamma_k}{3}\frac{2\pi}{k}-\frac{2\pi}{k}\right)\\
        &\geq \left(\frac{2\pi}{3} - \frac{6\pi}{3\ell} \right) - \left(\frac{2\pi}{3} - \frac{8\pi}{3k}\right) = \frac{8\pi\ell -6\pi k}{3\ell k}.
    \end{align}
    By Lemma \ref{lem:function-bounds} we obtain
    \begin{align}
        \frac{U(\theta^{**}_{\ell,1})k - U(\theta^{**}_{k,2})\ell}{\ell k} &\geq \frac{U(\theta^{**}_{\ell,1})}{\ell} \geq -\frac{1.05}{\ell}, \\
        \frac{V(\theta^{**}_{\ell,1})k^2 - V(\theta^{**}_{k,2})\ell^2}{\ell^2k^2} &\geq -\frac{V(\theta^{**}_{k,2})}{k^2}\geq -\frac{1.43}{k^2}, \\
        \frac{\alpha^*_{\ell,1}k^3 - \alpha^*_{k,2}\ell^3}{\ell^3k^3} &\geq -5\left(\frac{1}{k^3} + \frac{1}{\ell^3}\right).
    \end{align}
    Then
    \begin{align}
        \theta^*_{\ell, 1} - \theta^*_{k, 2} - \frac{3(\ell - k)}{\ell k}&\geq \frac{(8\pi - 9)\ell -(6\pi -9)k}{3\ell k} - \frac{1.05}{\ell} - \frac{1.43}{k^2} - 5\left( \frac{1}{k^3} + \frac{1}{\ell^3}\right) \\ 
        &\geq \frac{8\pi - 9}{3 k} - \frac{6\pi -9}{3(k+2)} - \frac{1.05}{k+2} - \frac{1.43}{k^2} - 5\left(\frac{1}{k^3} + \frac{1}{(k+2)^3}\right) \\
        &> 0 \quad \text{for }k \geq1000,
    \end{align}
    verifying the first claim of the lemma.
    
We now show the second claim. By Proposition \ref{prop:sample-zero-locations}, we have
    \begin{align}
        &\theta^*_{k,n_{k}-1} -\theta^*_{\ell, n_\ell} \\
        &= \theta^{**}_{k,n_k-1} - \theta^{**}_{\ell, n_\ell}
        + \frac{U(\theta^{**}_{k,n_k-1})\ell - U(\theta^{**}_{\ell,n_\ell})k}{\ell k}
        + \frac{V(\theta^{**}_{k,n_k-1})\ell^2 - V(\theta^{**}_{\ell,n_\ell})k^2}{\ell^2k^2}+ \frac{\alpha^*_{k,n_k-1}\ell^3 - \alpha^*_{\ell,n_\ell}k^3}{\ell^3k^3}.
    \end{align}
    From \eqref{eq:endpoint_theta_2}, and noting that $\gamma'_k \ge 1$ and $\gamma'_\ell\le2$, we have
    \begin{align}
        \theta^{**}_{k, n_k-1} - \theta^{**}_{\ell, n_\ell} = \theta^{**}_{k, n_k} +\frac{2\pi}{k} - \theta^{**}_{\ell, n_\ell}&=\left(\frac{\pi}{2}+\frac{\gamma'_k}{2}\frac{2\pi}{k}+\frac{2\pi}{k}\right)-\left(\frac{\pi}{2}+\frac{\gamma'_\ell}{2}\frac{2\pi}{\ell}\right) \\
        &\ge 
        \left(\frac{\pi}{2}+\frac{3\pi}{k}\right)-\left(\frac{\pi}{2}+\frac{2\pi}{\ell}\right)
        =
        \frac{3\pi\ell - 2\pi k}{\ell k}.
    \end{align} 
    By Lemma \ref{lem:function-bounds} we obtain the following bounds:
    \begin{align}
        \frac{U(\theta^{**}_{k,n_k-1})\ell - U(\theta^{**}_{\ell,n_\ell})k}{\ell k} &\geq
        \frac{U(\theta^{**}_{k,n_k-1})\ell}{\ell k}
        \ge
        - \frac{1.05}{k}, &\text{(since $U(\theta) \leq 0$)}\\
        \frac{V(\theta^{**}_{k,n_k-1})\ell^2 - V(\theta^{**}_{\ell,n_\ell})k^2}{\ell^2k^2} &\geq \frac{- V(\theta^{**}_{\ell,n_\ell})k^2}{\ell^2k^2} \geq - \frac{1.43}{\ell^2}, &(\text{since $V(\theta) \geq 0$})\\
        \frac{\alpha^*_{k,n_k-1}\ell^3 - \alpha^*_{\ell,n_\ell}k^3}{\ell^3k^3} &\geq -5 \left( \frac{1}{k^3} + \frac{1}{\ell^3}\right).
    \end{align}
    Thus
    \begin{align}
        \theta^*_{k,n_k-1} -\theta^*_{\ell, n_\ell} - \frac{3(\ell - k)}{\ell k} &\geq \frac{(3\pi - 3)\ell - (2\pi - 3) k}{\ell k} - \frac{1.05}{k} - \frac{1.43}{\ell^2} - 5\left(\frac{1}{k^3} + \frac{1}{\ell^3}\right) \\
        &\geq \frac{3\pi - 3}{k}- \frac{2\pi-3}{k+2} - \frac{1.05}{k} - \frac{1.43}{(k+2)^2} - 5\left(\frac{1}{k^3} + \frac{1}{(k+2)^3}\right) \\
        &> 0 \quad \text{for } k \geq 1000,
    \end{align}
    verifying the second claim of the Lemma.

\end{proof}

\section{Proof of Lemma \ref{lem:b()^-k_estimate}}

\begin{lemma} \label{lem:g-and-h-deriv-values}
    For $\phi \in [0, \frac{\pi}{6}]$, let
    \begin{align}
        g(\phi) &:= d\lrp{\frac{2\pi}{3}-\phi} \qquad
        \text{and} \qquad
        h(\phi) := \frac{-B_1\lrp{\frac{2\pi}{3}-\phi}}{2B_3\lrp{\frac{2\pi}{3}-\phi }}.
    \end{align}
    Then 
    \begin{align}
        &\log g(0) = 0, \qquad  (\log g)'(0) = \frac{\sqrt 3}{2}, \qquad 
        (\log g)''(0) = -1, \qquad 
        |(\log g)'''(\phi)| \le 1.74, \\
        &\log h(0) = 0, \qquad (\log h)'(0) = \frac{\sqrt 3}{2}, \qquad |(\log h)''(\phi)| \le 5.3, \qquad\\
        &\lrabs{h(\phi)} \le 1.09,\qquad \ h(\phi)\leq1.0116,\  \text{for }0\le\phi\le\frac{2\log1000}{1000}.
    \end{align}
    Additionally,
    \begin{align}
        \log g(\phi) \ge 0.66\, \phi 
        \quad \text{so that} \quad
        d(\theta)^{-k} \leq \exp\left(-0.66k\left(\frac{2\pi}{3} - \theta\right)\right) \quad \text{for}\quad \frac{\pi}{2} \le \theta \le \frac{2\pi}{3}.
    \end{align}
\end{lemma}
\begin{proof}
All the results of Lemma \ref{lem:g-and-h-deriv-values} are immediate by direct calculation, except for $\log h(0) = 0$ and $(\log h)'(0) = \frac{\sqrt 3}{2}$.

To show these last two results, we will prove that $B_0(\frac{2\pi}{3}) = -\frac{\sqrt 3}{3}$ and $B_0'(\frac{2\pi}{3}) = -\frac{2}{3}$. These two facts then imply that
\begin{align}
    B_1\left(\frac{2\pi}{3}\right)=\frac{2\sqrt{3}}{3},\qquad 
    B'_1\left(\frac{2\pi}{3}\right)=\frac{1}{3},
    \qquad 
    B_3\left(\frac{2\pi}{3}\right)=-\frac{\sqrt{3}}{3},
    \qquad
    B'_3\left(\frac{2\pi}{3}\right)=-\frac{2}{3},
\end{align}
yielding
\begin{align}
    \log h(0) &=\log\frac{-B_1(\frac{2\pi}{3})}{2B_3(\frac{2\pi}{3})}= \log 1 = 0, \\
    \text{and}\quad (\log h)'(0) &= \frac{h'(0)}{h(0)} = \frac{B_1'(\frac{2\pi}{3})B_3(\frac{2\pi}{3})-B_1(\frac{2\pi}{3})B'_3(\frac{2\pi}{3}
    )}{2B_3(\frac{2\pi}{3})^2h(0)}=\frac{\sqrt 3}{2},
\end{align}
as desired.

\textbf{Part 1: $B_0(\frac{2\pi}{3}) = -\frac{\sqrt 3}{3}$. \\}
Let
$$
\rho=e^{2\pi i/3}
\qquad\text{and}\qquad
\gamma z=-\frac{1}{z+1}.
$$
Since $\gamma\rho=\rho$ and $(\rho+1)^2=\rho$, the quasimodular
transformation law of \(E_2\) gives
\begin{align}
E_2(\rho)
&= E_2(\gamma\rho) \notag\\
&= (\rho+1)^2E_2(\rho)
   + \frac{6}{\pi i}(\rho+1) \notag\\
&= \rho E_2(\rho)
   + \frac{6}{\pi i}(\rho+1).
\end{align}
Hence,
\begin{align}
E_2(\rho)
&= \frac{6(\rho+1)}{\pi i(1-\rho)}
 = \frac{2\sqrt{3}}{\pi},
 \label{eq:E2(rho)}\\
B_0\left(\frac{2\pi}{3}\right)
&= \frac{\pi}{3}\mathrm{Re}\!\left[\rho E_2(\rho)\right] \notag\\
&= \frac{2\sqrt{3}}{3}\mathrm{Re}(\rho)
 = -\frac{\sqrt{3}}{3}.
\end{align}

\textbf{Part 2: $B_0'(\frac{2\pi}{3}) = -\frac{2}{3}$. \\}
Using
Ramanujan's differential identity
\begin{align}
   E_2'(z)=\frac{\pi i}{6}\left(E_2(z)^2-E_4(z)\right),
   \label{eq:Ramanujan_diff_id}
\end{align}
we have
\begin{align}
B_0'\left(\frac{2\pi}{3}\right)
&=\frac{\pi}{3}
  \mathrm{Re}\left[
      i\rho E_2(\rho)+i\rho^2E_2'(\rho)
  \right] \qquad\big(\text{since for $z=e^{i\theta}$, $\frac{dz}{d\theta}=iz$}\big)
  \notag
  \\
&=\frac{\pi}{3}
  \mathrm{Re}\left[
      i\rho E_2(\rho)
      -\frac{\pi}{6}\rho^2
      \left(E_2(\rho)^2-E_4(\rho)\right)
  \right]  \qquad\big(\text{by \eqref{eq:Ramanujan_diff_id}}\big)
  \notag
  \\
&=\frac{\pi}{3}
  \mathrm{Re}\left[
      i\rho\frac{2\sqrt{3}}{\pi}
      -\frac{\pi}{6}\rho^2
      \left(\frac{2\sqrt{3}}{\pi}\right)^2
  \right] \notag\\
&\qquad\big(\text{by \eqref{eq:E2(rho)} and since  $E_4(\rho)=0$ by \cite[Proposition 5.6.5]{cohen2017modular}}\big)\\
&=\frac{\pi}{3}
  \left[
      -\frac{\sqrt{3}}{2}\cdot\frac{2\sqrt{3}}{\pi}
      -\frac{\pi}{6}
       \left(-\frac12\right)\frac{12}{\pi^2}
  \right] \notag\\
&=\frac{\pi}{3}
  \left(-\frac{3}{\pi}+\frac{1}{\pi}\right)
 =-\frac{2}{3},
\end{align}
as desired.
\end{proof}

\begin{lemma} \label{lem:b()^-k_estimate}
    Let $-0.1 \le W \le 0.1$, and $k \ge 1000$. Then
    \begin{align}
        d \lrp{\theta + \frac{W}{k} \pm \frac{0.6}{k^2}}^{-k} = 
        \frac{-B_1(\theta)}{2B_3(\theta)} e^{-\frac{\sqrt 3}{2} (k(\frac{2\pi}{3}-\theta)-W)} + \frac{\alpha}{k} \\
        \text{where}\qquad 
        \lrabs{\alpha} \le 0.226 
        \qquad \text{for}\qquad 
        \frac{\pi}{2} + \frac{3.1}{k} \le \theta \le \frac{2\pi}{3}-\frac{3.1}{k}.
    \end{align}
\end{lemma}
\begin{proof}
    For $\phi \in [0, \frac{\pi}{6}]$, recall that
    \begin{align}
        g(\phi) &:= d\lrp{\frac{2\pi}{3}-\phi}, 
        &
        h(\phi) 
        &:= \frac{-B_1\lrp{\frac{2\pi}{3}-\phi}}{2B_3\lrp{\frac{2\pi}{3}-\phi }}, \\
        \text{and define}\quad
        g_k(s) &:= g\lrp{\frac{s}{k}},
        &
        h_k(s) 
        &:= h\lrp{\frac{s}{k}}.
    \end{align}
    Then it suffices to prove that
    \begin{align}
        g_k\lrp{s - W \mp \frac{0.6}{k}}^{-k} = h_k(s)\,e^{-\frac{\sqrt 3}{2} (s - W)} + \frac{\alpha}{k} \\
        \text{where}\qquad
        \lrabs{\alpha} \le 0.226
        \qquad \text{for}\qquad 
        3.1 \le s \le \tfrac{\pi}{6} k - 3.1.
    \end{align}
    We break into two cases.

    \textbf{Case 1: $s > 2 \log k$ \\}
    In this case, observe that
    \begin{align}
        h_k(s)\,e^{-\frac{\sqrt 3}{2} (s - W)}
        &\le 1.09\,e^{-\frac{\sqrt{3}}{2}(2 \log k-0.1)} = 1.09 \, k^{-\sqrt 3} e^{0.05 \sqrt 3} \le \frac{0.01}{k},
    \end{align}
   by Lemma \ref{lem:g-and-h-deriv-values} and since $k \ge 1000$.

    Additionally, we have that
    \begin{align}
        g_k\lrp{s-W\mp\frac{0.6}{k}}^{-k} 
        &\le g_k(2 \log k - 0.1006)^{-k} 
        &\text{since $k \ge 1000$}
        \\
        &= \exp\lrp{-k \log g\lrp{\frac{2 \log k - 0.1006}{k}}} \\
        &\le \exp\lrp{-k \cdot 0.66 \tfrac{2 \log k - 0.1006}{k}} &
        \text{by Lemma \ref{lem:g-and-h-deriv-values}}
        \\
        &= \exp\lrp{-1.32 \log k + 0.066396} \\
        &= e^{0.066396} k^{-1.32} \\
        &\le \frac{0.12}{k},&\text{since } k \ge 1000.
    \end{align}

    These two upper bounds then imply that  
    \begin{align}
        g_k\lrp{s - W \mp \frac{0.6}{k}}^{-k} = h_k(s)\,e^{-\frac{\sqrt 3}{2} (s - W)} + \frac{\alpha}{k}  \qquad \text{where}\quad
        \lrabs{\alpha} \le 0.12 + 0.01 \le 0.226,
    \end{align}
    as desired.

    \textbf{Case 2: $s \le 2 \log k$ \\}
    Recall that Lemma \ref{lem:g-and-h-deriv-values} states
    \begin{align}
        &\log g(0) = 0, \qquad  (\log g)'(0) = \frac{\sqrt 3}{2}, \qquad 
        (\log g)''(0) = -1, \qquad 
        |(\log g)'''(\phi)| \le 1.74, 
        \\
        &\text{so that} \quad
        \log g_k(t) = \frac{\sqrt 3}{2k} t - \frac{1}{2k^2} t^2 + \frac{\alpha_1}{k^3} t^3  \quad \text{over $t \in \lrb{0, \tfrac{\pi}{6}k}$}, \quad \text{where $\lrabs{\alpha_1} \le 0.29$}.
    \end{align}
    Additionally, recall that Lemma \ref{lem:g-and-h-deriv-values} states
    \begin{align}
        &\log h(0) = 0, \qquad (\log h)'(0) = \frac{\sqrt 3}{2}, \qquad |(\log h)''(t)| \le 5.3 
        \\
        &\text{so that} \quad
        \log h_k(t) = \frac{\sqrt 3}{2k} t + \frac{\alpha_2}{k^2} t^2 
        \quad \text{over $t \in \lrb{0, \tfrac{\pi}{6}k}$}, \quad
        \text{where} \quad \lrabs{\alpha_2} \le 2.65.
    \end{align}

    Then evaluating $\log g_k(t)$ and $\log h_k(t)$ at $t=s-W\mp \frac{0.6}{k}$ and $t=s$, respectively, we obtain that
    \begin{align}
        &- \log h_k(s) - k \log g_k\lrp{s-W \mp \frac{0.6}{k}}  \\
        &= - \frac{\sqrt 3}{2k} s -  \frac{\alpha_2}{k^2} s^2 
        - k 
        \lrp{
            \frac{\sqrt 3}{2k} 
            \lrp{s-W \mp \frac{0.6}{k}}
            - \frac{1}{2k^2}
            \lrp{s- W \mp \frac{0.6}{k}}^2
            + \frac{\alpha_1}{k^3}
            \lrp{s- W \mp \frac{0.6}{k}}^3
        } \\
        &= - \frac{\sqrt 3}{2} \big(s-W\big)
        + 
        \frac{1}{k}
        \lrb{
            - \frac{\sqrt 3}{2} s - \frac{\alpha_2}{k} s^2 \pm \frac{0.6\sqrt 3}{2}
            + \frac{1}{2} \lrp{s-W \mp \frac{0.6}{k}}^2
            - \frac{\alpha_1}{k} \lrp{s-W \mp \frac{0.6}{k}}^3
        } \\
        &=: - \frac{\sqrt 3}{2} \big(s-W\big)
        + 
        \frac{\alpha_3}{k}, \label{eqn:temp-estimate-for-log-b()^(-k)}
    \end{align}
    where
    \begin{align}
        \lrabs{\alpha_3} &= 
        \lrabs{
            - \frac{\sqrt 3}{2} s - \frac{\alpha_2}{k} s^2 \pm \frac{0.6\sqrt 3}{2}
            + \frac{1}{2} \lrp{s-W \mp \frac{0.6}{k}}^2
            - \frac{\alpha_1}{k} \lrp{s-W \mp \frac{0.6}{k}}^3
        } \\
        &\le
        \lrabs{ 
            \frac{1}{2} \lrp{s-W \mp \frac{0.6}{k}}^2
            - \frac{\sqrt 3}{2} s
        }
        +
        \lrabs{
            \pm \frac{0.6\sqrt 3}{2}
            - \frac{\alpha_2}{k} s^2
            - \frac{\alpha_1}{k} \lrp{s-W \mp \frac{0.6}{k}}^3
        } \\
        &\le 
        \frac{1}{2} (s+0.1006)^2 - \frac{\sqrt 3}{2} s
        +
        \frac{0.6\sqrt 3}{2} +
        0.00265 s^2 + 0.00029 \lrp{s+0.1006}^3 \\
        &\qquad\Big(\text{since $\tfrac{1}{2} \lrp{s-W \mp \tfrac{0.6}{k}}^2-\tfrac{\sqrt 3}{2} s \ge \tfrac{1}{2} \lrp{s-0.1006}^2-\tfrac{\sqrt 3}{2} s \ge 0$}  \\
        &\qquad\quad\, \text{and $\tfrac{1}{2} \lrp{s-W \mp \tfrac{0.6}{k}}^2-\tfrac{\sqrt 3}{2} s \le \tfrac{1}{2} \lrp{s+0.1006}^2-\tfrac{\sqrt 3}{2} s$ for $s \ge 3.1$} \Big) \\
        &=: p(s).
    \end{align}

    Then exponentiating \eqref{eqn:temp-estimate-for-log-b()^(-k)}, we obtain that
    \begin{align}
        g_k\lrp{s-W\mp \frac{0.6}{k}}^{-k} 
        &= h_k(s) \,e^{-\frac{\sqrt 3}{2} (s-W) + \frac{\alpha_3}{k} } \\
        &= h_k(s) \, e^{-\frac{\sqrt 3}{2} (s-W) } + h_k(s) \, e^{-\frac{\sqrt 3}{2} (s-W)} 
        \lrp{e^{\frac{\alpha_3}{k}} - 1} \\
        &=: h_k(s)\, e^{-\frac{\sqrt 3}{2} (s-W) }  + \frac{\alpha}{k},
    \end{align}
    where
    \begin{align}
        \lrabs{\alpha} 
        &= \lrabs{ h_k(s)\,e^{-\frac{\sqrt 3}{2} (s-W)} 
        \lrp{e^{\frac{\alpha_3}{k} }-1} k} \\
        &\le\, \sup_{3.1\le t\le 2\log k}h_k(t)\cdot
        \sup_{3.1 \le t \le 4}
        e^{-\frac{\sqrt 3}{2} (t-0.1)}
        \lrp{ 
        e^{
            \frac{1}{k} 
            p(t)
        }
        -1
        } k 
        \qquad (\text{by \eqref{eq:d/dt>0}})
        \\
        &\le 
         1.012 \sup_{3.1 \le t \le 4} e^{-\frac{\sqrt 3}{2} (t-0.1)}
        \lrp{ 
        e^{
            \frac{1}{1000} 
            p(t)
        }
        -1
        } \cdot 1000 \\
        & \qquad\big(\text{by Lemma \ref{lem:g-and-h-deriv-values} and since $\lrp{e^{
            \frac{1}{k} 
            p(t)
        }
        -1
        } k$ is decreasing in $k$ for each $t \in [3.1,\, 4]$}\big)
        \\
        &\le 1.012 \cdot 0.223 \le 0.226,
    \end{align}
    as desired.
    
    Here, we used the fact that for $t \in [4,\, 2\log k]$ 
    \begin{align}
        &\frac{d}{dt} \lrb{
            e^{-\frac{\sqrt 3}{2} (t-0.1)}
            \lrp{ 
            e^{
                \frac{1}{k} 
                p(t)
            }
            -1
            }k
        } \\
        &=e^{-\frac{\sqrt 3}{2} (t-0.1)} 
        \lrb{
            -\frac{\sqrt 3}{2} \lrp{e^{\frac{1}{k} p(t) }-1}k + e^{\frac{1}{k} p(t) } p'(t)
        } 
        \\
        &\le e^{-\frac{\sqrt 3}{2} (t-0.1)} 
        \lrb{
            -\frac{\sqrt 3}{2} p(t) + e^{\frac{p(2 \log k)}{k} } p'(t)
        } 
        \qquad \text{since $t \le 2 \log k$ and $p'(t) \ge 0$}
        \\
        &\le e^{-\frac{\sqrt 3}{2} (t-0.1)} 
        \lrb{
            -\frac{\sqrt 3}{2} p(t) + e^{0.0867} p'(t)
        } \qquad\quad\! \text{since } k \ge 1000
        \\
        &< 0 \qquad \text{since $t \ge 4$}\label{eq:d/dt>0}.
    \end{align}
    This completes the proof.
\end{proof}

\section{Proof of Lemma \ref{lem:perfect-interlacing}}

Recall that
\begin{align}
    \theta^{**}_{k,1} &= \frac{2\pi}{3} - \frac{\gamma_k}{3} \frac{2\pi}{k} \qquad \text{where } \gamma_k := 
    \begin{cases}
        3 & \text{if } k \equiv 0 \mod 6 \\
        2 & \text{if } k \equiv 2 \mod 6 \\
        1 & \text{if } k \equiv 4 \mod 6
    \end{cases}, 
    \\
    \theta^{**}_{k,n_k} &= \frac{\pi}{2} + \frac{\gamma'_k}{2} \frac{2\pi}{k} \qquad \text{where } \gamma'_k := 
    \begin{cases}
        2 & \text{if } k \equiv 0 \mod 4 \\
        1 & \text{if } k \equiv 2 \mod 4
    \end{cases}.
\end{align}

\begin{lemma}\label{lem:perfect-interlacing}
    For $k\geq1000$, suppose that $k < \ell \le k+26$. Then 
    \begin{align}
        \theta_{\ell,1} > \theta_{k,1} \quad \text{iff} \quad \gamma_\ell  \le \gamma_k 
        \qquad \text{and} \qquad
        \theta_{\ell,n_\ell} < \theta_{k,n_k} \quad \text{iff} \quad \gamma'_\ell  \le \gamma'_k. 
    \end{align}
\end{lemma}
\begin{proof}
Since $k\geq1000$ and $\ell\leq k+26$, we have $\ell\leq1.026k$. We will just show $\theta_{\ell,1} > \theta_{k,1} \ \  \text{iff} \ \  \gamma_\ell  \le \gamma_k$, and the proof of the other equivalence is nearly identical.

First, observe that 
    \begin{align}
        \theta_{\ell,1} - \theta_{k,1} &=
        \lrp{\theta^{**}_{\ell,1} - \theta^{**}_{k,1}}
        +
        \lrp{\theta_{\ell,1} - \theta^{**}_{\ell,1}}
        -
        \lrp{\theta_{k,1} - \theta^{**}_{k,1}} \\
        &= \lrp{\frac{\gamma_k}{3} \frac{2\pi}{k} - \frac{\gamma_\ell}{3} \frac{2\pi}{\ell}}
        +
        \lrp{\theta_{\ell,1} - \theta^{**}_{\ell,1}}
        -
        \lrp{\theta_{k,1} - \theta^{**}_{k,1}} \\
        &= \lrp{
            \gamma_k \frac{2\pi (\ell-k)}{3\ell k} 
            + 
            (\gamma_k - \gamma_\ell)
            \frac{2\pi}{3\ell}
        }
        +
        \lrp{\theta_{\ell,1} - \theta^{**}_{\ell,1}}
        -
        \lrp{\theta_{k,1} - \theta^{**}_{k,1}}.
    \end{align}

When $\gamma_k=\gamma_\ell$, the second explicit term vanishes, so the
first one is the primary term. When $\gamma_k\neq\gamma_\ell$, the
second explicit term is the primary term, while the first one must be
absorbed into the error term. Since $\lrabs{\gamma_k \frac{2\pi(\ell-k)}{3\ell k}} \le 3 \frac{2\pi \cdot 26}{3000 \ell} \le \frac{0.164}{\ell}$, it suffices to prove that
    \begin{align}
        \lrabs{
            \lrp{\theta_{k,1} - \theta^{**}_{k,1}}
            -
            \lrp{\theta_{\ell,1} - \theta^{**}_{\ell,1}}
        }
        <
        \left\{\begin{array}{cc}
        \frac{2\pi (\ell-k)}{3\ell k}     & \mbox{if }\gamma_k=\gamma_\ell, \\
         \frac{2\pi}{3\ell} - \frac{0.164}{\ell}   & \mbox{if }\gamma_k\neq\gamma_\ell.
        \end{array}
        \right.
        \label{eq:theta_error_bound}
    \end{align}

Note that
    \begin{align}
        &\lrabs{\lrp{\theta_{k,1} - \theta^{**}_{k,1}}
        -
        \lrp{\theta_{\ell,1} - \theta^{**}_{\ell,1}}}\\
        =\ &\lrabs{\lrp{\theta_{k,1} - \theta^{*}_{k,1}} +
        \lrp{\theta^{*}_{k,1} - \theta^{**}_{k,1}}
        -
        \lrp{\theta_{\ell,1} - \theta^{*}_{\ell,1}}
        -
        \lrp{\theta^{*}_{\ell,1} - \theta^{**}_{\ell,1}}}\\
        =\ &\lrabs{\frac{U(\theta^{**}_{k,1})}{k}+\frac{V(\theta^{**}_{k,1})}{k^2}+\frac{\alpha^*_{k,1}}{k^3}+\frac{W_{k,1}}{k}+\frac{\alpha_{k,1}}{k^2}-\frac{U(\theta^{**}_{\ell,1})}{\ell}-\frac{V(\theta^{**}_{\ell,1})}{\ell^2}-\frac{\alpha^*_{\ell,1}}{\ell^3}-\frac{W_{\ell,1}}{\ell}-\frac{\alpha_{\ell,1}}{\ell^2}}\\
        &\qquad\left(\text{by Proposition \ref{prop:sample-zero-locations} and Corollary \ref{cor:actual-zero-sample-zero-distance-bound}}.\right)\\
        \le\ &\lrabs{\frac{U(\theta^{**}_{k,1})}{k}-\frac{U(\theta^{**}_{\ell,1})}{\ell}}+\lrabs{\frac{V(\theta^{**}_{k,1})}{k^2}-\frac{V(\theta^{**}_{\ell,1})}{\ell^2}}+\lrabs{\frac{W_{k,1}}{k}-\frac{W_{\ell,1}}{\ell}}+\lrabs{\frac{l^2\alpha_{k,1}-k^2\alpha_{l,1}}{k^2l^2}}+\lrabs{\frac{l^3\alpha^*_{k,1}-k^3\alpha^*_{l,1}}{k^3l^3}}\\
        \le\ &\lrabs{\frac{U(\theta^{**}_{k,1})}{k}-\frac{U(\theta^{**}_{\ell,1})}{\ell}}+\lrabs{\frac{V(\theta^{**}_{k,1})}{k^2}-\frac{V(\theta^{**}_{\ell,1})}{\ell^2}}+\lrabs{\frac{W_{k,1}}{k}-\frac{W_{\ell,1}}{\ell}}+0.60501\left(
        \frac{1}{k^2}+\frac{1}{\ell^2}\right)\\
        &\qquad\left(\text{since $|\alpha_{k,1}|,|\alpha_{\ell,1}|\le0.60001$, $|\alpha^*_{k,1}|,|\alpha^*_{\ell,1}|\leq5$ and $k,\ell\ge1000$}.\right)\\
        \le\ &\lrabs{\frac{U(\theta^{**}_{k,1})}{k}-\frac{U(\theta^{**}_{\ell,1})}{\ell}}+\lrabs{\frac{V(\theta^{**}_{k,1})}{k^2}-\frac{V(\theta^{**}_{\ell,1})}{\ell^2}}+\lrabs{\frac{W_{k,1}}{k}-\frac{W_{\ell,1}}{\ell}}+\frac{0.613(\ell-k)}{\ell k}.\\
        &\qquad\left(\text{since $\ell\leq1.026k$ and $\ell-k\geq2$}\right)
    \end{align}
We now estimate each term separately.
Observe that
\begin{align}
    \theta^{**}_{\ell,1}-\theta^{**}_{k,1}=
    \left\{\begin{array}{cc}
        \gamma_k \frac{2\pi (\ell-k)}{3\ell k} & \mbox{if }\gamma_k=\gamma_\ell,\\
        \frac{2\pi}{3\ell}(\gamma_k-\gamma_\ell+\epsilon_{k,\ell}) & \mbox{if }\gamma_k\neq\gamma_\ell,
    \end{array}
        \right.
\end{align}
where $\epsilon_{k,\ell}=\frac{\gamma_k(\ell-k)}{k}$. Since $\gamma_k\le3$, $\ell-k\le26$ and $k\geq1000$,
\begin{align}
    \epsilon_{k,\ell}\le0.078.
    \label{ineq:epsilon}
\end{align}
Hence, we have
    \begin{align}
        \lrabs{\frac{U(\theta^{**}_{k,1})}{k}-\frac{U(\theta^{**}_{\ell,1})}{\ell}}
        &=\frac{\lrabs{\ell U(\theta^{**}_{k,1})-kU(\theta^{**}_{\ell,1})}}{\ell k}\\
        &\le\frac{\lrabs{(\ell-k)U(\theta^{**}_{k,1})}}{\ell k}+
       \frac{\lrabs{k(U(\theta^{**}_{k,1})-U(\theta^{**}_{\ell,1}))}}{\ell k}\\
       &=\frac{\lrabs{(\ell-k)U(\theta^{**}_{k,1})}}{\ell k}+\frac{\lrabs{U'(\xi)(\theta^{**}_{k,1}-\theta^{**}_{\ell,1})}}{\ell}\\
       &\qquad\left(\text{where $\xi$ lies between $\theta^{**}_{k,1}$ and $\theta^{**}_{\ell,1}$}\right)\\
       &\le1.05\frac{\ell-k}{\ell k}+\frac{2.60\lrabs{\theta^{**}_{k,1}-\theta^{**}_{\ell,1}}}{\ell}\\
       &\qquad\left(\text{by Lemma \ref{lem:function-bounds}, $|U|\leq1.05$ and $|U'|\le2.60$}\right)\\
       &\le\left\{\begin{array}{cc}
        1.05\frac{\ell-k}{\ell k}    +\frac{2.60\cdot2\pi\gamma_k(\ell-k)}{3\ell^2k}& \mbox{if }\gamma_k=\gamma_\ell \\
        1.05\frac{\ell-k}{\ell k} + \frac{2.60\cdot2\pi(|\gamma_k-\gamma_\ell|+\epsilon_{k,\ell})}{3\ell^2}    & \mbox{if }\gamma_k\neq\gamma_\ell
       \end{array}
       \right.\\
       &\le\left\{\begin{array}{l}
        1.07\frac{\ell-k}{\ell k}    \qquad\mbox{if }\gamma_k=\gamma_\ell, \\
        \qquad\left(\text{since $\ell\ge1000$ and  $\gamma_k\le3$}\right)\\
        6.71\frac{\ell-k}{\ell k}  \qquad\mbox{if }\gamma_k\neq\gamma_\ell.\\
        \qquad\left(\text{by \eqref{ineq:epsilon} and since $\ell-k\ge2$ and $|\gamma_k-\gamma_\ell|\leq2$}\right)
       \end{array}
       \right.
    \end{align}

We now divide into two cases.

\textbf{Case 1: $\gamma_k\neq\gamma_\ell$ \\}
In this case, observe that by Lemmas \ref{lem:sample-zero-at-least-3.1/k-away-from-2pi/3} and \ref{lem:exponential-W-bound}, $|W_{k,1}|, |W_{\ell,1}| \leq e^{-3.1\cdot\frac{\sqrt{3}}{2}}<0.076$, and by Lemma \ref{lem:function-bounds}, $|V|<1.43$, and also noting that $k\geq1000$ and $\ell-k\leq26$, we have
    \begin{align}
        &\lrabs{\lrp{\theta_{k,1} - \theta^{**}_{k,1}}
        -
        \lrp{\theta_{\ell,1} - \theta^{**}_{\ell,1}}}\\
        \le\ &\lrabs{\frac{U(\theta^{**}_{k,1})}{k}-\frac{U(\theta^{**}_{\ell,1})}{\ell}}+\lrabs{\frac{V(\theta^{**}_{k,1})}{k^2}}+\lrabs{\frac{V(\theta^{**}_{\ell,1})}{\ell^2}}+\lrabs{\frac{W_{k,1}}{k}}+\lrabs{\frac{W_{\ell,1}}{\ell}}+\frac{0.613(\ell-k)}{\ell k}\\
        \le\ &6.71\frac{\ell-k}{\ell k}+\frac{1.43}{k^2}+\frac{1.43}{\ell^2}+\frac{0.076}{k}+\frac{0.076}{\ell}+\frac{0.613(\ell-k)}{\ell k}\\
        \le\ &\frac{0.35}{\ell}<\frac{2\pi}{3\ell} - \frac{0.164}{\ell},
    \end{align}
therefore \eqref{eq:theta_error_bound} holds.

\textbf{Case 2: $\gamma_k =\gamma_\ell$ \\}
Noting that by Lemma \ref{lem:function-bounds}, $|V'|\le6.75$, we have
    \begin{align}
        \lrabs{\frac{V(\theta^{**}_{k,1})}{k}-\frac{V(\theta^{**}_{\ell,1})}{\ell}}
        &=\frac{\lrabs{\ell V(\theta^{**}_{k,\ell})-kV(\theta^{**}_{\ell,1})}}{\ell k}\\
        &\le\frac{\lrabs{(\ell-k)V(\theta^{**}_{k,1})}}{\ell k}+
       \frac{\lrabs{k(V(\theta^{**}_{k,1})-V(\theta^{**}_{\ell,1}))}}{\ell k}\\
       &=\frac{\lrabs{(\ell-k)V(\theta^{**}_{k,1})}}{\ell k}+\frac{\lrabs{V'(\xi)(\theta^{**}_{k,1}-\theta^{**}_{\ell,1})}}{\ell}\\
       &\qquad\left(\text{where $\xi$ lies between $\theta^{**}_{k,1}$ and $\theta^{**}_{\ell,1}$}\right)\\
       &\le1.43\frac{\ell-k}{\ell k}+\frac{6.75\lrabs{\theta^{**}_{k,1}-\theta^{**}_{\ell,1}}}{\ell}\\
       &\le1.43\frac{\ell-k}{\ell k}    +\frac{6.75\cdot2\pi\gamma_k(\ell-k)}{3\ell^2k}\\
       &\le1.48\frac{\ell-k}{\ell k}.
    \end{align}
Also since $\gamma_k=\gamma_\ell$, we have 
\begin{align}
    s^{**}_{k,1}=k\left(\frac{2\pi}{3}-\theta^{**}_{k,1}\right)=\frac{2\pi\gamma_k}{3}=\frac{2\pi\gamma_\ell}{3}=\ell\left(\frac{2\pi}{3}-\theta^{**}_{\ell,1}\right)=s^{**}_{\ell,1}.
\end{align}
Hence, letting $T^{-1}$ denote the inverse function of $T(W)$ (defined in Proposition \ref{prop:postsample-sample-distance-formula}), we have
    \begin{align}
        &\lrabs{W_{k,1}-W_{\ell,1}}\\
       =\ &\lrabs{T^{-1}\left((-1)^{1+\delta_k}e^{-\frac{\sqrt{3}}{2}s^*_{k,1}}\right)-T^{-1}\left((-1)^{1+\delta_\ell}e^{-\frac{\sqrt{3}}{2}s^*_{\ell,1}}\right)}\\
       =\ &(T^{-1})'(\xi)\lrabs{e^{-\frac{\sqrt{3}}{2}s^*_{k,1}}-e^{-\frac{\sqrt{3}}{2}s^*_{\ell,1}}}\\
       &\qquad\left(\text{where $\xi$ lies between $(-1)^{1+\delta_k}e^{-\frac{\sqrt{3}}{2}s^*_{k,1}}$ and $(-1)^{1+\delta_k}e^{-\frac{\sqrt{3}}{2}s^*_{\ell,1}}$}.\right)\\
       =\ &(T^{-1})'(\xi)e^\eta\frac{\sqrt{3}}{2}|s^*_{k,1}-s^*_{\ell,1}|\\
       &\qquad\left(\text{where $\eta$ lies between ${-\frac{\sqrt{3}}{2}s^*_{k,1}}$ and ${-\frac{\sqrt{3}}{2}s^*_{\ell,1}}$}.\right)\\
       =\ &(T^{-1})'(\xi)e^\eta\frac{\sqrt{3}}{2}\lrabs{\left(s^{**}_{k,1}-{U(\theta^{**}_{k,1})}-\frac{V(\theta^{**}_{k,1})}{k}-\frac{\alpha^*_{k,1}}{k^2}\right)-\left(s^{**}_{\ell,1}-{U(\theta^{**}_{\ell,1})}-\frac{V(\theta^{**}_{\ell,1})}{\ell}-\frac{\alpha^*_{\ell,1}}{\ell^2}\right)}\\
       \le\ &(T^{-1})'(\xi)e^\eta\frac{\sqrt{3}}{2}\left(\lrabs{{U(\theta^{**}_{k,1})}-U(\theta^{**}_{\ell,1})}+\lrabs{\frac{V(\theta^{**}_{k,1})}{k}-\frac{V(\theta^{**}_{\ell,1})}{\ell}}+\lrabs{\frac{\alpha^*_{k,1}}{k^2}-\frac{\alpha^*_{\ell,1}}{\ell^2}}\right)\\
       \le\ &0.084\cdot\left(\frac{2.60\cdot2\pi\gamma_k(\ell-k)}{3\ell k}+\frac{1.48(\ell-k)}{\ell k}+\frac{5}{k^2}+\frac{5}{\ell^2}\right)\\
       &\qquad\bigg(\text{since $(T^{-1})'(\xi)=\tfrac{1}{T'(T^{-1}(\xi))}\le\tfrac{1}{0.84}<1.21$ by inspection of $T'(\cdot)$,}\\
       &\qquad\ \ \text{and $e^\eta\leq e^{-3.1\tfrac{\sqrt{3}}{2}}<0.069$ 
       since  $\eta\le-\tfrac{\sqrt{3}}{2}\min\{s^*_{k,1},s^*_{\ell,1}\}\le-3.1\tfrac{\sqrt{3}}{2}$ by Lemma \ref{lem:sample-zero-at-least-3.1/k-away-from-2pi/3}}\bigg)\\
       \le\ &0.084\cdot\left(\frac{2.60\cdot2\pi(\ell-k)}{\ell k}+\frac{1.48(\ell-k)}{\ell k}+\frac{5.065(\ell-k)}{\ell k}\right)\\
       \le\ &\frac{2(\ell-k)}{\ell k}.
    \end{align}
So, we obtain,
    \begin{align}
        \lrabs{\frac{W_{k,1}}{k}-\frac{W_{\ell,1}}{\ell}}
        &\le\lrabs{\frac{(\ell-k)W_{k,1}}{\ell k}}+
        \lrabs{\frac{k(W_{k,1}-W_{\ell,1})}{\ell k}}\\
       &\le1.1\cdot e^{-\frac{\sqrt{3}}{2}s^*_{k,1}}\frac{\ell-k}{\ell k}+\frac{2(\ell-k)}{\ell^2k}\\
       &\le0.076\frac{\ell-k}{\ell k}+\frac{2(\ell-k)}{\ell^2k}\\
       &\le0.078\frac{\ell-k}{\ell k}.
    \end{align}
Furthermore,

    \begin{align}
        \lrabs{\frac{V(\theta^{**}_{k,1})}{k^2}-\frac{V(\theta^{**}_{\ell,1})}{\ell^2}}
        &\le\lrabs{\frac{(\ell^2-k^2)V(\theta^{**}_{k,1})}{\ell^2k^2}}+\lrabs{\frac{k^2(V(\theta^{**}_{k,1})-V(\theta^{**}_{\ell,1}))}{\ell^2k^2}}\\
        &\le\frac{1.43\cdot2.026(\ell-k)}{\ell^2 k}+\lrabs{\frac{V'(\xi)(\theta^{**}_{k,1}-\theta^{**}_{\ell,1})}{\ell^2}}\\
        &\qquad\left(\text{where $\xi$ lies between $\theta^{**}_{k,1}$ and $\theta^{**}_{\ell,1}$}\right)\\&\le\frac{1.43\cdot2.026(\ell-k)}{\ell^2 k}+\frac{6.75\cdot2\pi(\ell-k)}{\ell^3k}\\
        &\le0.003\frac{\ell-k}{\ell k}.
    \end{align}
Then combining all of the above, we finally obtain
    \begin{align}
        &\lrabs{\lrp{\theta_{k,1} - \theta^{**}_{k,1}}
        -
        \lrp{\theta_{\ell,1} - \theta^{**}_{\ell,1}}}\\
        \le\ &\lrabs{\frac{U(\theta^{**}_{k,1})}{k}-\frac{U(\theta^{**}_{\ell,1})}{\ell}}+\lrabs{\frac{V(\theta^{**}_{k,1})}{k^2}-\frac{V(\theta^{**}_{\ell,1})}{\ell^2}}+\lrabs{\frac{W_{k,1}}{k}-\frac{W_{\ell,1}}{\ell}}+\frac{0.613(\ell-k)}{\ell k}\\
        \le\ &1.07\frac{\ell-k}{\ell k}+0.003\frac{\ell-k}{\ell k}+0.078\frac{\ell-k}{\ell k}+0.613\frac{\ell-k}{\ell k}\\
        =\ &1.764\frac{\ell-k}{\ell k}<\frac{2\pi (\ell-k)}{3\ell k},
    \end{align}
verifying \eqref{eq:theta_error_bound}. 
\end{proof}

\section{Proof of Proposition \ref{prop:stieltjes-interlacing-large-k}}
\label{sec:proof_of_prop:stieltjes-interlacing-large-k}

\subsection{Preliminary Lemmas}

The following coordinates will simplify our work.
\begin{definition}
    Define
    \begin{align}
        s^{**}_{k,i} &:= k\left(\frac{2\pi}{3} - \theta^{**}_{k,i}\right), 
        \qquad
        s^*_{k,i}:= k\left(\frac{2\pi}{3} - \theta_{k,i}^*\right), 
        \qquad
        s_{k,i} := k\left(\frac{2\pi}{3} - \theta_{k,i}\right).
    \end{align}
\end{definition}

We record the following observations:
\begin{lemma}\label{lemma:s-inequalities}
    Let $s^{**}_{k,i}$ and $s^{**}_{\ell,j}$ be such that $s^{**}_{k,i} \neq s^{**}_{\ell,j}$ and $\left | \theta^*_{k,i} - \theta^*_{\ell,j}\right |\le\frac{2\pi}{k}$. Then for $k \geq 1000,$
    \begin{align}
        \left |s^{**}_{k,i} - s^{**}_{\ell,j} \right | \geq \frac{2\pi}{3}, 
        \qquad
        |s^*_{k,i} - s^*_{\ell, j}| > 2.078,
        \qquad
        |s_{k,i} - s_{\ell,j}| &> 1.876,
    \end{align}
    with the sign of each difference the same as the sign of $s^{**}_{k,i} - s^{**}_{\ell,j}$.
\end{lemma}
\begin{proof}
    Since every presample zero is of the form $\frac{2\pi}{3} - \frac{2\pi}{k} \lrp{i - \frac{\delta_k}{3}}$,
    \begin{align}
        s^{**}_{k,i} &= k\left(\frac{2\pi}{k}\left(\frac{3i -\delta_{k}}{3}\right)\right) = \frac{2\pi}{3}(3i - \delta_k).
    \end{align}
    Since $\delta_k$ is an integer, $s^{**}_{k,i}$ is a multiple of $\frac{2\pi}{3}$. Hence if $s^{**}_{k,i} \neq s^{**}_{\ell, j}$, then $\left |s^{**}_{k,i} - s^{**}_{\ell, j} \right | \geq \frac{2\pi}{3}$. This proves the first claim.
    
    It's immediate from Lemma \ref{lem:sample-presample-exact-expression} that
    \begin{align}
        s^*_{k,i} = s^{**}_{k,i} - U(\theta^*_{k,i}). \label{eq:s-star-exact-expression}
    \end{align}
    so we have
      \begin{align}
          \left | s^{*}_{k,i} - s^{*}_{\ell, j} \right | 
          &= \left |\lrp{s^{**}_{k,i} - U(\theta^*_{k,i})} - \left(s^{**}_{\ell, j} - U(\theta^*_{\ell, j})\right) \right|
          &\left(\text{by \eqref{eq:s-star-exact-expression}}\right)
          \\
          &\geq \left |s^{**}_{k,i} - s^{**}_{\ell, j} \right | - \left |U(\theta^{*}_{k,i}) - U(\theta^*_{\ell, j}) \right | 
          \\
          &\geq \frac{2\pi}{3} - 2.60\left |\theta^*_{k,i} - \theta^*_{\ell,j}\right | \quad &(\text{by Lemma \ref{lem:function-bounds})} \\
          &\geq \frac{2\pi}{3} - 2.60\left(\frac{2\pi}{1000}\right) \quad &\left(\text{since $\left | \theta^*_{k,i} - \theta^*_{\ell,j}\right |\le\tfrac{2\pi}{k}$}\right)\\
          &> 2.078. \label{eq:k-ell-s-star-inequality-sign}
      \end{align}
      Since $\tfrac{2\pi}{3} > 2.60\tfrac{2\pi}{1000}$, $s^*_{k,i} - s^*_{\ell, j}$ has the same sign as $s^{**}_{k,i} - s^{**}_{\ell, j}.$ This proves the second claim.

      Finally, observe that
      \begin{align}
          s_{k,i} - s_{\ell,j} = (s^{*}_{k,i} - s^{*}_{\ell, j}) - (s^*_{k,i} - s_{k,i} ) - (s_{\ell, j} - s^*_{\ell,j}) 
      \end{align}
      so since 
      \begin{align}
          s^*_{k,i} - s_{k,i} &= k(\theta_{k,i} - \theta^*_{k,i}) = W_{k,i} + \frac{\alpha_{k,i}}{k} &\text{(by Corollary \ref{cor:actual-zero-sample-zero-distance-bound})},
      \end{align}
      we have
      \begin{align}
          \left | s_{k,i} - s_{\ell, j}\right | &\geq \left | s^{*}_{k,i} - s^{*}_{\ell, j}\right | - \left | s_{k,i} - s^*_{k,i}\right | - \left | s_{\ell, j} - s^*_{\ell,j}\right | \\
          &> 2.078 - \left|W_{k,i} + \tfrac{\alpha_{k,i}}{k} \right| - \left|W_{\ell, j} + \tfrac{\alpha_{\ell,j}}{\ell} \right|
          \\&> 2.078 - 2(0.1007)  &\text{(by bounds from Corollary \ref{cor:actual-zero-sample-zero-distance-bound})}\\
          &> 1.876, \label{eq:l-ks-hat-inequality-sign}
      \end{align}
        with $s_{k,i} - s_{\ell, j}$ the same sign as $s^{*}_{k,i} - s^{*}_{\ell, j}$ (since $2.078 > 2(0.1007)$), which has the same sign as $s^{**}_{k,i} - s^{**}_{\ell,j}$. This proves the third claim.
\end{proof}

We now bound the difference $W_{k,i} - W_{\ell,j}$.
\begin{lemma}\label{lem:W-diff-bound}
    Suppose that $\theta^*_{k,i+1} < \theta^*_{\ell, j} < \theta^*_{k,i}$ and $s^{**}_{k,i} = s^{**}_{\ell,j}$. Then 
    \begin{align}
        \left |W_{k,i} - W_{\ell,j} \right | < \frac{3.7}{k}.
    \end{align}
\end{lemma}
\begin{proof}
The equality $s^{**}_{k,i} = s^{**}_{\ell,j}$ implies that $(i,\delta_k) = (j,\delta_\ell)$, so that $(-1)^{i-\delta_k} = (-1)^{j-\delta_\ell}$. These are precisely the signs appearing in the definitions of $W_{k,i}$ and $W_{\ell,j}$ in Proposition \ref{prop:postsample-sample-distance-formula}; call this sign $\epsilon$, where $\epsilon\in\{+1,-1\}$. 

Observe that 
    \begin{align}
        \left | s^*_{k,i} - s^*_{\ell, j}\right | &= \left |s^{**}_{k,i} - s^{**}_{\ell, j} + U(\theta^{*}_{\ell, j}) -U(\theta^*_{k,i})\right|\qquad \text{(by Lemma \ref{lem:sample-presample-exact-expression})} \\
&=\left|U(\theta^{*}_{\ell, j}) -U(\theta^*_{k,i}) \right| \qquad \text{(since we assume } s^{**}_{k,i} = s^{**}_{\ell,j})\\
        &\leq 2.60 \left |\theta^*_{\ell, j} - \theta^*_{k, i}\right| \qquad\text{(by \eqref{eq:s-star-exact-expression}})\\
        &< 2.60 \left(\frac{2\pi}{k}\right) \qquad \text{(since } |\theta^*_{\ell, j} - \theta^*_{k,i}| < |\theta^*_{k,i} - \theta^*_{k,i+1}| \leq \tfrac{2\pi}{k} \text{ by Corollary \ref{cor:sample-zero-2pi-k-distance-bound}}).
    \end{align}
    By the mean value theorem,
    \begin{align}
        \left |W_{k,i} - W_{\ell, j} \right| = \left | \frac{s^*_{\ell,j} - s^*_{k,i}}{H'(c)} \right |,
    \end{align}
    where
    \begin{align}
        H(W) := \left(-\frac{2}{\sqrt{3}}\right)\log\left(2\epsilon\sin\left(\frac{W}{2}\right)e^{-W\frac{\sqrt{3}}{2}}\right)
    \end{align}
    is a certain function satisfying $H(W_{k,i}) = s^*_{k,i}$ (obtained by solving the identity $T(W_{k,i})=\epsilon\, e^{-\frac{\sqrt 3}{2} s^*_{k,i}}$ for $s^*_{k,i}$). Finally, we have that
    \begin{align}
        \left |H'(W) \right | > 4.5 \quad \text{for } -0.1 < W < 0.1,
    \end{align}
    so
    \begin{align}
        \left |W_{k,i} - W_{\ell, j} \right| < \frac{2.60(2\pi)}{4.5k} < \frac{3.7}{k},
    \end{align}
    completing the proof.
\end{proof}

We also give an upper bound for $W_{k,i}$ in terms of $s^*_{k,i}$.
\begin{lemma} \label{lem:exponential-W-bound}
    With $W_{k,i}$ as defined in Proposition \ref{prop:postsample-sample-distance-formula},
    \begin{align}
        \left|W_{k,i}\right| \leq 1.1e^{-\frac{\sqrt{3}}{2}s^*_{k,i}}.
    \end{align}
\end{lemma}
\begin{proof}
    Since $|W_{k,i}| \leq 0.1$ and $|\!\tan x|\ge |x|$, we have that
    \begin{align}
        \left|T(W_{k,i})\right| = 2\lrabs{\sin\lrp{\frac{W_{k,i}}2}} e^{-\frac{\sqrt{3}}{2} W_{k,i}} &= 2 \lrabs{\tan\lrp{\frac{W_{k,i}}{2}}} \lrabs{\cos\lrp{\frac{W_{k,i}}{2}}} e^{- \frac{\sqrt 3}{2} W_{k,i}} \\
        &\geq |W_{k,i}|\cos\left(\frac{0.1}{2}\right)e^{-0.1\frac{\sqrt{3}}{2}}.
    \end{align}
    This then implies that
    \begin{align}
        |W_{k,i}| &\leq \frac{ \left|T(W_{k,i})\right|}{\cos(0.05)e^{-0.1\frac{\sqrt{3}}{2}}} \leq 1.1|T(W_{k,i})| = 1.1e^{-\frac{\sqrt{3}}{2} s^*_{k,i} },
    \end{align}
    as desired.
\end{proof}

\subsection{Proof of the upper bound of Proposition \ref{prop:stieltjes-interlacing-large-k}}

{
\renewcommand{\thetheorem}{\ref{prop:stieltjes-interlacing-large-k}}
\addtocounter{theorem}{-1}
\begin{proposition}[\textbf{upper bound}]
\label{prop:actual-interlacing-upper-ineq}
    For $k \ge 1000$, let
    \begin{align}
        \theta^*_{k,i+1} + \frac{3(\ell-k)}{\ell k}
        <
        \theta^*_{\ell,j} 
        <
        \theta^*_{k,i} - \frac{3 (\ell-k)}{\ell k}
    \end{align}
    be the sample zeros from Lemma \ref{lem:interlacing-sample-zeros-with-gap}. Then
    \begin{align}
        \theta_{\ell,j} 
        <
        \theta_{k,i}.
    \end{align}
\end{proposition}
}
\begin{proof}
    Using Corollary \ref{cor:actual-zero-sample-zero-distance-bound}, the desired result is straightforward to prove if $\ell \geq 1.1 k$ or if $\theta^*_{\ell,j} \leq \theta^*_{k,i+1}$.  
    
    So we assume $\ell < 1.1 k$ and $\theta^*_{\ell,j} > \theta^*_{k,i+1}$. Recall that
    \begin{align}
        \theta_{k,i} - \theta^*_{k,i} &= \frac{W_{k,i}}{k} + \frac{\alpha_{k,i}}{k^2} &\left |\alpha_{k,i} \right |\leq 0.60001
    \end{align}
    Hence to prove $\theta_{\ell, j} < \theta_{k,i}$, it suffices to show that
    \begin{align}
        \left |\theta_{k,i} - \theta^*_{k,i} - (\theta_{\ell, j} - \theta^*_{\ell,j})\right | = \left | \frac{W_{k, i}}{k} + \frac{ \alpha_{k,i}}{k^2} - \frac{W_{\ell, j}}{\ell} - \frac{ \alpha_{\ell, j}}{\ell^2} \right| < \frac{3 (\ell-k)}{\ell k}.
    \end{align}
    We case on whether $s^{**}_{k,i} = s^{**}_{\ell,j}$. Assume first $s^{**}_{k,i} = s^{**}_{\ell,j}$. Then
    \begin{align}
        \left |\frac{W_{k, i}}{k} + \frac{\alpha_{k,i}}{k^2} - \frac{W_{\ell, j}}{\ell} - \frac{\alpha_{\ell, j}}{\ell^2} \right| &\leq \left |\frac{W_{k, i}}{k}- \frac{W_{\ell, j}}{\ell}\right| + \frac{0.60001}{k^2} + \frac{0.60001}{\ell^2} \\
        &\leq |W_{k,i}|\left(\frac{1}{k} - \frac{1}{\ell}\right) + \frac{\left|W_{k,i} - W_{\ell, j} \right|}{\ell} + \frac{0.60001}{k^2} + \frac{0.60001}{\ell^2} \\
         &< 0.1\left(\frac{1}{k} - \frac{1}{\ell}\right) + \frac{3.7}{\ell k} + \frac{0.60001}{k^2} + \frac{0.60001}{\ell^2} \\
         &\quad (\text{by Proposition \ref{prop:postsample-sample-distance-formula} and Lemma \ref{lem:W-diff-bound}})
         \\
        &\leq \frac{0.1(\ell - k) + 3.7 + 0.60001\left(\frac{\ell}{k} + \frac{k}{\ell}\right)}{\ell k} \\
        &< \frac{0.1(\ell - k) + 4.93}{\ell k}\qquad \text{(since } \tfrac{\ell}{k} < 1.1) \\
        &< \frac{3(\ell - k)}{\ell k} \quad \text{for }\ell - k \geq 2,
    \end{align}
    implying $\theta_{\ell, j} < \theta_{k,i}$.

    The case of $s^{**}_{k,i} > s^{**}_{\ell, j}$ is impossible; by Lemma \ref{lemma:s-inequalities}, this would imply that
    \begin{align}
        \frac{s^{*}_{k,i}}{k} > \frac{s^*_{\ell,j}}{k} > \frac{s^*_{\ell, j}}{\ell},
    \end{align}
    or equivalently
    \begin{align}
        \theta^*_{\ell,j} > \theta^*_{k,i},
    \end{align}
    contradicting our assumptions for this proposition.

    Now assume $s^{**}_{k,i} < s^{**}_{\ell, j}$. Suppose for contradiction that $\theta_{\ell, j} \geq \theta_{k,i}$. Then by Lemma \ref{lem:actual-interlacing-upper-ineq-s-lb},
    \begin{align} \label{eqn:temp-s*ellj-lb}
        s^*_{\ell, j} > s^*_{k,i} > \frac{1.876k}{\ell - k}-0.1007.
    \end{align} 
    Then,
    \begin{align}
        \left |\theta_{k,i} - \theta^*_{k,i} - (\theta_{\ell, j} - \theta^*_{\ell,j})\right | &= \left |\frac{W_{k, i}}{k} + \frac{ \alpha_{k,i}}{k^2} - \frac{W_{\ell, j}}{\ell} - \frac{\alpha_{\ell, j}}{\ell^2} \right| \\ &
        \leq \left |\frac{W_{k,i}}{k} \right | + \left | \frac{W_{\ell, j}}{\ell} \right | + \left |\frac{\alpha_{k,i}}{k^2} \right| + \left| \frac{\alpha_{\ell,j}}{\ell^2}\right| \\
        &\leq \frac{1.1}{k}e^{-s^*_{k,i}\frac{\sqrt{3}}{2}} + \frac{1.1}{\ell}e^{-s^*_{\ell,j}\frac{\sqrt{3}}{2}} + \frac{0.60001}{k^2} + \frac{0.60001}{\ell^2} \quad \text{(by Lemma \ref{lem:exponential-W-bound})} \\
        &\leq 1.1e^{-\frac{\sqrt{3}}{2}\left(1.876 \frac{k}{\ell - k} - 0.1007 \right)}\left(\frac{1}{k} + \frac{1}{\ell}\right) + \frac{0.60001}{k^2} + \frac{0.60001}{\ell^2} \quad \text{(by \eqref{eqn:temp-s*ellj-lb})}\\
        &= \frac{\ell - k}{\ell k}\left[1.1e^{-\frac{\sqrt{3}}{2}\left(1.876 \frac{k}{\ell - k} - 0.1007 \right)}\left(\frac{2 + \frac{\ell - k}{k}}{\frac{\ell - k}{k}}\right) + \frac{0.60001}{\ell - k}\left(\frac{\ell}{k} + \frac{k}{\ell}\right)\right] \\
        &\leq \frac{\ell - k}{\ell k}\left(0.00001 + \frac{0.60001\left(1.1 + \frac{1}{1.1}\right)}{2}\right)
        \qquad \left( \text{since $\tfrac{k}{\ell - k} \geq 10$} \right)
        \\
        &< \frac{3(\ell-k)}{\ell k}.
    \end{align}
    
    This implies $\theta_{\ell, j} < \theta_{k,i}$, a contradiction. We thus have $\theta_{\ell, j} < \theta_{k, i}$ in all three cases, completing the proof.    
\end{proof}

\begin{lemma}\label{lem:actual-interlacing-upper-ineq-s-lb}
    Let $k \geq 1000$. Suppose that
    \begin{align}
        \theta^*_{k,i+1} + \frac{3(\ell - k)}{\ell k}<\theta^*_{\ell,j} 
        <
        \theta^*_{k,i} - \frac{3 (\ell-k)}{\ell k}, 
        \qquad 
        \theta_{\ell, j} \geq \theta_{k,i},
        \qquad \text{and} \qquad
        s^{**}_{k,i} < s^{**}_{\ell,j}.
    \end{align}
    Then 
    \begin{align}
        s^*_{\ell, j} > s^*_{k,i} > \frac{1.876k}{\ell - k} - 0.1007.
    \end{align}
\end{lemma}

\begin{proof}
    Since  $\theta^*_{k,i+1} + \frac{3(\ell-k)}{\ell k} <\theta^*_{\ell,j} 
        <
        \theta^*_{k,i} - \frac{3 (\ell-k)}{\ell k}$, we have $\left | \theta^*_{k,i} - \theta^*_{\ell,j}\right |\le\frac{2\pi}{k}$ by Corollary \ref{cor:sample-zero-2pi-k-distance-bound}.
    Observe that
    \begin{align}
         \theta_{\ell, j} \geq \theta_{k,i} 
         \quad\text{implies}\quad
         \frac{s_{\ell,j}}{\ell} \leq \frac{s_{k,i}}{k} 
         \quad\text{implies}\quad 
         ks_{\ell,j} \leq \ell s_{k,i}.
    \end{align}
    Hence by Lemma \ref{lemma:s-inequalities},
    \begin{align}
        1.876k &< ks_{\ell,j} - ks_{k,i} \leq (\ell - k)s_{k,i},
    \end{align}
    so that
    \begin{align}
        s_{k,i} > \frac{1.876k}{\ell - k}.
    \end{align}
    Since $s^{**}_{\ell,j} > s^{**}_{k,i}$ implies $s^*_{\ell,j} > s^*_{k,i}$ by Lemma \ref{lemma:s-inequalities}, using Corollary \ref{cor:actual-zero-sample-zero-distance-bound} we obtain
    \begin{align}
        s^*_{\ell,j} > s^*_{k,i} &= s_{k,i} +W_{k,i} + \frac{\alpha_{k,i}}{k}\\
        &\geq \frac{1.876k}{\ell - k} - 0.1 - \frac{0.60001}{k}\\
        &> \frac{1.876k}{\ell - k} - 0.1007, \qquad \text{(since $k \ge 1000$)}
    \end{align}
    as desired.
\end{proof}

\subsection{Proof of the lower bound of Proposition \ref{prop:stieltjes-interlacing-large-k}}

{
\renewcommand{\thetheorem}{\ref{prop:stieltjes-interlacing-large-k}}
\addtocounter{theorem}{-1}
\begin{proposition}[\textbf{lower bound}]
    For $k \ge 1000$, let
    \begin{align}
        \theta^*_{k,i+1} + \frac{3(\ell-k)}{\ell k}
        <
        \theta^*_{\ell,j} 
        <
        \theta^*_{k,i} - \frac{3 (\ell-k)}{\ell k}
    \end{align}
    be the sample zeros from Lemma \ref{lem:interlacing-sample-zeros-with-gap}. Then
    \begin{align}
        \theta_{\ell,j} > \theta_{k,i+1}.
    \end{align}
\end{proposition}
}
\begin{proof}
   Using Corollary \ref{cor:actual-zero-sample-zero-distance-bound}, the desired result is straightforward to prove if $\ell \geq 1.1 k$, so we assume $\ell < 1.1k$. It suffices to show
    \begin{align}
        \left |\frac{W_{k, i+1}}{k} + \frac{ \alpha_{k,i+1}}{k^2} - \frac{W_{\ell, j}}{\ell} - \frac{\alpha_{\ell, j}}{\ell^2} \right| < \frac{3 (\ell-k)}{\ell k}.
    \end{align}
    In the case of $s^{**}_{k,i+1}=s^{**}_{\ell,j}$, the proof is identical to the corresponding case in Proposition \ref{prop:actual-interlacing-upper-ineq} (upper bound), so we will not repeat it here.

    Suppose now $s^{**}_{k,i+1} > s^{**}_{\ell,j}$. By Lemma \ref{lemma:s-inequalities} we know $s_{k,i+1} >  s_{\ell,j}$. Since $k < \ell$ we conclude that
    \begin{align}
        \frac{s_{k,i+1}}{k} > \frac{s_{\ell,j}}{\ell},
    \end{align}
    or equivalently 
    \begin{align}
        \theta_{k,i+1} < \theta_{\ell,j},
    \end{align}
    as desired.

    Finally, suppose $s^{**}_{k,i+1} < s^{**}_{\ell,j}$. Then by Lemma \ref{lem:actual-interlacing-upper-ineq-s-lb-i+1}
    \begin{align}
        s^{*\*}_{\ell,j} > s^*_{k,i+1}>\frac{2.078k}{\ell - k},
    \end{align}
    so
    \begin{align} 
        &\left |\frac{W_{k, i+1}}{k} + \frac{\alpha_{k,i+1}}{k^2} - \frac{W_{\ell, j}}{\ell} - \frac{\alpha_{\ell, j}}{\ell^2} \right| \\
        &
        \leq \left |\frac{W_{k,i+1}}{k} \right | + \left | \frac{W_{\ell, j}}{\ell} \right | + \left |\frac{\alpha_{k,i+1}}{k^2} \right| + \left| \frac{ \alpha_{\ell,j}}{\ell^2}\right| \\
        &\leq \frac{1.1}{k}e^{-s^*_{k,i+1}\frac{\sqrt{3}}{2}} + \frac{1.1}{\ell}e^{-s^*_{\ell,j}\frac{\sqrt{3}}{2}} + \frac{0.60001}{k^2} + \frac{0.60001}{\ell^2} \qquad \text{(by Lemma \ref{lem:exponential-W-bound})} \\
        &\leq 1.1e^{-\frac{\sqrt{3}}{2}\left(2.078 \frac{k}{\ell - k}\right)}\left(\frac{1}{k} + \frac{1}{\ell}\right) + \frac{0.60001}{k^2} + \frac{0.60001}{\ell^2} \qquad \text{(by Lemma \ref{lem:actual-interlacing-upper-ineq-s-lb-i+1})} \\
        &= \frac{\ell - k}{\ell k}\Bigg[1.1e^{-\frac{\sqrt{3}}{2}\left(2.078 \frac{k}{\ell - k} \right)}\left(\frac{2 + \frac{\ell - k}{k}}{\frac{\ell - k}{k}}\right) + \frac{0.60001}{\ell - k}\left(\frac{\ell}{k} + \frac{k}{\ell}\right)\Bigg] \\
        &\leq \frac{\ell - k}{\ell k}\left(0.00001 + \frac{0.60001\left(1.1 + \frac{1}{1.1}\right)}{2}\right) \qquad \left(\text{since $\tfrac{k}{\ell - k} \geq 10$}\right)
        \\
        &< \frac{3(\ell-k)}{\ell k}.
    \end{align}

    We thus have $\theta_{k,i+1} < \theta_{\ell, j}$ in all three cases, completing the proof.
\end{proof}

We now prove the inequality we assumed:
\begin{lemma}\label{lem:actual-interlacing-upper-ineq-s-lb-i+1}
    Let $k\ge1000$. Suppose
    \begin{align}
         \theta^*_{k,i+1} + \frac{3(\ell-k)}{\ell k}
        <
        \theta^*_{\ell,j} 
        < \theta^*_{k,i} - \frac{3(\ell-k)}{\ell k}\qquad\text{and}\qquad s^{**}_{\ell,j}>s^{**}_{k,i+1}.
    \end{align}
    Then,  
    \begin{align}
        s^{*\*}_{\ell,j} > s^*_{k,i+1}>\frac{2.078k}{\ell-k}.
    \end{align}
\end{lemma}
\begin{proof}
    Since $\theta^*_{k,i+1} + \frac{3(\ell-k)}{\ell k}
        <
        \theta^*_{\ell,j} 
        < \theta^*_{k,i} - \frac{3(\ell-k)}{\ell k}$, we have  $\left | \theta^*_{k,i+1} - \theta^*_{\ell,j}\right |\le\frac{2\pi}{k}$ by Corollary \ref{cor:sample-zero-2pi-k-distance-bound}. Observe that 
    \begin{align}
        \theta^*_{\ell, j} > \theta^*_{k,i+1} 
       \quad&\text{implies}\quad
       \tfrac{s^*_{\ell,j}}{\ell} < \tfrac{s^*_{k,i+1}}{k},\qquad\text{and}\\
       s^{**}_{\ell,j}>s^{**}_{k,i+1} \hspace{4mm}&\text{implies}\hspace{4mm}s^{*}_{\ell,j}>s^{*}_{k,i+1}\qquad(\text{by Lemma \ref{lemma:s-inequalities}}).
    \end{align}
    Hence, by Lemma \ref{lemma:s-inequalities},
    \begin{align}
       2.078k &< ks^*_{\ell, j} - ks^*_{k,i+1} < \left(\ell - k\right)s^*_{k,i+1} 
    \end{align}
    so that
    \begin{align}
        s^*_{\ell, j} >s^*_{k,i+1}>\frac{2.078k}{\ell-k}
    \end{align}
    as desired.
\end{proof}

\section*{Acknowledgments}
This research was supported by NSF grant DMS-2349174. Hui Xue is supported by Simons Foundation grant MPS-TSM-00007911.

\bibliographystyle{plain} 
\bibliography{bib.bib}

\end{document}